\documentclass[preprint,12pt]{article}
\usepackage{amssymb}
\usepackage{mathrsfs}
\usepackage{amsfonts}
\usepackage{graphicx}
\usepackage{colortbl,dcolumn}
\usepackage{amsmath}
\usepackage{psfrag}
\usepackage{booktabs}
\usepackage{subfigure}
\usepackage{cite}
\usepackage{appendix}
\usepackage{amsthm}
\usepackage{verbatim}
\usepackage{hyperref}
\hypersetup{hypertex=true,colorlinks=true,
linkcolor=blue,
anchorcolor=red,
citecolor=red}

\numberwithin{equation}{section}
\usepackage[top=0.9in, bottom=0.9in, left=0.8in, right=0.8in]{geometry}

\newcommand{\R}{\mathbb{R}}

\newcommand{\E}{\mathbb{E}}

\newcommand{\dd}{\text{d}}

\newtheorem{theorem}{Theorem}[section]
\newtheorem{lemma}[theorem]{Lemma}
\newtheorem{assumption}{Assumption}[section]

\begin{document}
\title{Antithetic multilevel Monte Carlo method for approximations of SDEs with non-globally Lipschitz continuous  coefficients
\footnotemark[1]}

\author{
Chenxu Pang$^{1}$, Xiaojie Wang$^{1,}$\footnotemark[2]
\\
\footnotesize $^{1}$ School of Mathematics and Statistics, HNP-LAMA, Central South University, Changsha, Hunan, P. R. China\\
}

       \maketitle

       \footnotetext{\footnotemark[1] This work was supported by Natural Science Foundation of China (12471394, 12071488, 12371417). 
                }
        \footnotetext{\footnotemark[2]Corresponding author.
        \newline E-mail addresses: x.j.wang7@csu.edu.cn, c.x.pang@csu.edu.cn}
       \begin{abstract}
          {\rm\small
{\color{black}
             In the field of computational finance,
             one is commonly interested in the expected value of a financial derivative whose payoff depends on the solution of stochastic differential equations (SDEs).
             For  multi-dimensional SDEs with non-commutative diffusion coefficients in the globally Lipschitz setting,  
              a kind of one-half order truncated Milstein-type scheme without
            Lévy areas was recently introduced by Giles and Szpruch (2014), which combined with
              the antithetic multilevel Monte Carlo (MLMC) gives the optimal overall computational cost $\mathcal{O}(\epsilon^{-2})$ for the required target accuracy $\epsilon$.
              Nevertheless, many nonlinear SDEs in applications have non-globally  Lipschitz continuous coefficients and the corresponding theoretical guarantees for antithetic MLMC are absent in the literature.
              In the present work, we aim to fill the gap and analyze antithetic MLMC in a non-globally Lipschitz setting.
              First, we propose a family of modified Milstein-type schemes without Lévy areas to approximate SDEs with non-globally Lipschitz continuous coefficients. The expected one-half order of strong convergence is recovered in a non-globally Lipschitz setting,
              where even the diffusion coefficients are allowed to grow superlinearly. This then helps us to analyze the relevant variance of the multilevel estimator
              and the optimal computational cost is finally 
              achieved for the antithetic MLMC.
              Since getting rid of the Lévy areas destroys the martingale properties of the scheme, 
              the analysis of both the convergence rate and the desired variance becomes highly non-trivial in the non-globally Lipschitz setting.
              By introducing an auxiliary approximation process, we develop non-standard arguments to overcome the essential difficulties.
              Numerical experiments are provided to confirm the theoretical findings.
}              
} 


\end{abstract}

%
\section{Introduction}   \label{section:introduction}
Throughout this paper, we consider It\^o  stochastic differential equations (SDEs) 
as follows:
\begin{equation} \label{equation:sde-in-introduction}
\left\{
\begin{array}{l}
\mathrm{d} X_{t}= \mu\left( X_{t} \right) \mathrm{d} t+\sigma \left( X_{t} \right) \mathrm{d} W_{t}, \quad t \in(0, T] \\
X_{0}=x_{0},
\end{array}\right.
\end{equation}
where  the drift coefficient function $\mu: \mathbb{R}^{d} \rightarrow \mathbb{R}^{d}$ and the diffusion coefficient 
function $\sigma: \mathbb{R}^{d} \rightarrow \mathbb{R}^{d \times m}$ are assumed to be non-globally Lipschitz continuous. 
Here $W_{\cdot} = \left(W_{1,\cdot}, \dots, W_{m,\cdot} 
\right)^{T} :[0, T] \times \Omega \rightarrow \mathbb{R}^{m}$ 
denotes the $\mathbb{R}^{m}$-valued standard Brownian motion
with respect to $\left\{\mathcal{F}_{t}\right\}_{t \in [0, T]}$
and the initial data $x_{0}: \Omega \rightarrow \mathbb{R}^{d}$
is assumed to be $\mathcal{F}_{0}$-measurable.

In recent years, SDEs have become a crucial instrument in a wide range of scientific fields and there has been, especially in computational finance, an increased emphasis on the quantification of expectations of some functions of the solution to SDEs \eqref{equation:sde-in-introduction}.
More accurately, one aims to compute
\begin{equation} \label{introduction:expectation-of-functional-of-sde}
    \mathbb{E} 
    \left[
    \phi\left( X_{T}\right) 
    \right]
\end{equation}
for some given function $\phi$.
To this end, the Monte Carlo (MC) method, combined with the Euler-Maruyama time-stepping scheme, is commonly used for this quantification at yet a high computational cost. 
Indeed, to achieve  a root-mean-square error of $\epsilon$ 
of approximating the quantity \eqref{introduction:expectation-of-functional-of-sde}, the computational cost will be $\mathcal{O}(\epsilon^{-3})$ \cite{duffie1995efficient}.

Recently, a breakthrough was brought by Giles who generalized the multigrid idea in \cite{kebaier2005statistical} and proposed the multilevel Monte Carlo (MLMC) method \cite{giles2008multilevel}. The MLMC method is regarded as a more effective algorithm with the aim to achieve a sharper bound for the mean-squared-error for a given computational cost. In the MLMC framework, 
some strongly convergent numerical scheme is used for
the time discretization of SDEs \eqref{equation:sde-in-introduction} and the MC method is used to simulate a sequence of levels of approximations with different timestep sizes (see \eqref{equation:main-idea-of-mlmc}). 
Evidently, 
higher level results in finer estimates but greater cost.
The main idea to minimize the overall computational complexity
of the MLMC is that more simulations are done on the lower level and relatively fewer simulations are needed on the upper level \cite{giles2008multilevel}, leading to a significant reduction in the overall computational cost to $\mathcal{O}(\epsilon ^{-2} (\log(1/\epsilon))^{2})$ when combined with the Euler-Maruyama discretization, and to $\mathcal{O}(\epsilon ^{-2})$ when 
combined with the Milstein scheme for scalar SDEs \cite{giles2008improved}.
Accordingly,  the Milstein discretization with 
ﬁrst-order strong convergence is superior in the case
of dealing with scalar SDEs. This also holds true for 
SDEs with commutative diffusion coefficients, where
the Milstein scheme relies on only Brownian increments.

However,
for multi-dimensional SDEs that do not satisfy the commutativity condition, the Milstein scheme requires additional efforts to simulate the iterated It\^o integrals, also known as Lévy areas (see \eqref{equation:levy-area}), which unavoidably
reduces the efficiency of the scheme. So one may ask if
any scheme with ﬁrst-order strong convergence can be 
constructed without the Lévy areas.
Unfortunately, the authors of \cite{clark1980maximum}
give a negative answer and assert that $\mathcal{O}(h^{\frac{1}{2}})$ strong convergence is the best that can be achieved using only Brownian increments.
Another natural question thus arises whether the optimal complexity $\mathcal{O}(\epsilon ^{-2})$ can be achieved without the simulation of the Lévy areas.
%
By constructing a suitable antithetic estimator despite the $\mathcal{O}(h^{\frac{1}{2}})$ strong convergence,
Giles and Szpruch \cite{giles2014antithetic} have very recently 
answered this question to the positive when 
both the drift and diffusion coefficients are 
assumed to be globally Lipschitz continuous.

What happens if the globally Lipschitz condition is violated?
As asserted by \cite{hutzenthaler2011strong}, the basic Euler-Maruyama method using step-size $h>0$ fails to converge in a weak or strong sense 
in the asymptotic limit $h \rightarrow 0$
when used to solve SDEs with super-linearly growing coefficients.
In \cite{hutzenthaler2011convergence}, Hutzenthaler and Jentzen showed that the Monte Carlo combined with the Euler–Maruyama can achieve convergence almost surely in cases where the underlying Euler–Maruyama scheme diverges.
This can happen when the events causing Euler–Maruyama to diverge are so rare that they are extremely unlikely
to impact on any of the Monte Carlo samples (see \cite{higham2015introduction} for relevant comments).
%
However, as shown by \cite{hutzenthaler2013divergence}, 
the MLMC method combined with the Euler–Maruyama scheme fails to converge to the desired quantity 
\eqref{introduction:expectation-of-functional-of-sde} 
for SDEs with super-linearly growing nonlinearities.
In \cite{hutzenthaler2013divergence},  a so-called tamed Euler was also shown to recover convergence in the multilevel setting.
In other words, the MLMC method seems sensitive to the divergence of the numerical schemes and
%
a strongly convergent numerical method for SDEs with non-globally Lipschitz continuous coefﬁcients is necessary. 
Recent years have witnessed a growing interest in 
construction and analysis of convergent schemes 
in non-globally Lipschitz settings (see 
\cite{higham2002strong,2010Strong,kelly2021adaptive,kumar2019milstein,li2019explicit,neuenkirch2014first,sabanis2016euler,wang2020mean1,wang2013tamed,wang2020mean,hu96semi-implicit} and references therein). 
{\color{black}In particular, we are aware of a recent work \cite{derouich2022interpolated} treating the problem of using MLMC methods in the setting of one-dimensional diffusions
that may have non-globally Lipschitz diffusion coefficient.}

Nevertheless, for multidimensional SDEs whose coefficients 
do not obey the global Lipschitz condition,
whether the antithetic MLMC method 
\cite{giles2014antithetic} can
be adapted to achieve the optimal complexity $\mathcal{O}(\epsilon^{-2})$ is, as far as we know, still an unsolved problem.
%
To solve this problem, we firstly propose strongly convergent  Milstein-type methods without the Lévy areas
for SDEs \eqref{equation:sde-in-introduction} with non-commutative diffusion coefficients.
More formally in this paper we establish a general framework 
for a family of modified Milstein (MM) scheme  without Lévy areas, given by
\begin{equation} \label{introduction:modified-milstein-method}
\left\{
\begin{array}{l}
Y_{n+1}
=\mathscr{P}(Y_{n}) 
+  \mu_{h}\big(\mathscr{P}(Y_{n}) \big) h 
+ \sigma_{h}\big(\mathscr{P}(Y_{n}) \big) \Delta W_{n}
+\sum_{j_{1}, j_{2}=1}^{m} 
\left(
\mathcal{L}^{j_{1}}\sigma_{j_{2}}
\right)_{h} 
\big( \mathscr{P}(Y_{n}) \big) 
\Pi_{j_{1}, j_{2}}^{t_{n}, t_{n+1}}, \\
Y_{0} = x_{0},
\end{array}\right.
\end{equation} 
where $h=\frac{T}{N}$, $N \in \mathbb{N}$, is the uniform timestep size, 
$\mathscr{P}: \mathbb{R}^{d} \rightarrow \mathbb{R}^{d}$ denotes some projection operator and $\mu_{h}(\cdot): \mathbb{R}^{d}\rightarrow \mathbb{R}^{d}$, $\sigma_{h}(\cdot): \mathbb{R}^{d}\rightarrow \mathbb{R}^{d\times m}$, $\left(\mathcal{L}^{j_{1}} \sigma_{j_{2}}\right)_{h}(\cdot): \mathbb{R}^{d}\rightarrow \mathbb{R}^{d}$, $j_{1}, j_{2}\in \{1, \dots,m\}$, can be regarded as certain modifications to the coefficients $\mu, \sigma$ and $ \mathcal{L}^{j_{1}}\sigma_{j_{2}} $ defined by \eqref{equation:milstein-term}.
Moreover, we denote
$\Delta W_{ n}:=W_{ t_{n+1}} - W_{ t_{n}}$, $n\in \{0,1, \dots,N-1 \}$,
$\Delta W_{j, n}:=W_{j, t_{n+1}} - W_{j, t_{n}}, j \in \{ 1, \dots,m\}$
and
\begin{equation} \label{introduction:truncated-process}
\Pi_{j_{1}, j_{2}}^{t_{n}, t_{n+1}}
:= \tfrac{1}{2} \left(\Delta W_{j_{1}, n} \Delta W_{j_{2}, n}-\Omega_{j_{1} j_{2}} h \right) 
\quad \text{with} \quad 
  \Omega_{j_{1} j_{2}} 
  = \left\{
  \begin{array}{l}
  1, \quad j_{1} = j_{2} \\
  0, \quad j_{1} \neq j_{2}
  \end{array} \right.,
  \,
  j_{1}, j_{2}\in \{1, \dots,m\}.
\end{equation}
%

Even neglecting the modifications to the coefficients 
and the projection operator,
the proposed schemes 
\eqref{introduction:modified-milstein-method},
only relying on Brownian increments,
are different from the classical Milstein scheme that involves 
the iterated It\^o integrals and can be viewed as a
truncation version of the Milstein scheme. 
Such a truncation 
to get rid of the Lévy areas 
destroys the martingale properties of the scheme
and prevents the continuous-time extension
of \eqref{introduction:modified-milstein-method} from being
an It\^o process \cite{mao2007stochastic}. Then the 
It\^o formula is not applicable. 
{\color{black}
In the globally Lipschitz setting, the loss of the  It\^o formula does not cause any problem in the analysis, because one can obtain the strong convergence rate and moment bounds of numerical approximations following standard arguments without the use of the It\^o formula \cite[Lemma 4.2]{giles2014antithetic}.
However, the loss of the  It\^o formula and nonglobally Lipschitz coefficients pose substantial challenges in deriving the convergence rate, moment bounds, and the relevant variance in the following analysis.
}

Indeed, when treating SDEs with superlinearly growing coefficients 
by explicit time-stepping schemes, 
it becomes a standard way in the literature to work with 
continuous-time extensions of  the explicit schemes 
and carry out the analysis based on the It\^o formula 
(see, e.g., \cite{sabanis2016euler,kumar2019milstein,hutzenthaler2012strong,wang2013tamed}). In this work, 
non-standard arguments are developed to overcome these difficulties in the analysis.
%
%
%
For example, to obtain  high-order moments of the schemes, 
we work with the discrete scheme \eqref{introduction:modified-milstein-method}
and
rely on discrete strategies based on Taylor expansions 
(see the proof of Theorem \ref{theorem:moments-bound-of-MM}). In the analysis of the strong convergence rate, 
by introducing an auxiliary process 
$\{\widetilde{Y}_{n}\}_{0\leq n \leq N}$, defined by 
\begin{equation}
\label{equation:intro-auxiliary-process}
\widetilde{Y}_{n+1} 
= \widetilde{Y}_{n} 
+ \int_{t_{n}}^{t_{n+1}} \! \mu(X_{s}) \ \mathrm{d} s 
+ \int_{t_{n}}^{t_{n+1}}  \! \sigma(X_{s}) \ \mathrm{d} W_{s} 
+ \sum_{j_{1},j_{2}=1}^{m} 
\!
\left(
\mathcal{L}^{j_{1}} \sigma_{j_{2}}
\right)_{h}
\big(
\mathscr{P}(Y_{n})
\big) 
\Pi_{j_{1}, j_{2}}^{t_{n}, t_{n+1}},
\,
\widetilde{Y}_{0} = x_{0},
\end{equation}
we separate the strong error $\|X_{t_{n}}-Y_{n} \|_{L^{2p}(\Omega, \mathbb{R}^{d})}$ into two parts:
\begin{equation} \label{introduction:strong-convergence-decomposition}
 \|X_{t_{n}}-Y_{n} \|_{L^{2p}(\Omega, \mathbb{R}^{d})} 
 \leq \|X_{t_{n}}-\widetilde{Y}_{n} \|_{L^{2p}(\Omega, \mathbb{R}^{d})} 
  + \|\widetilde{Y}_{n}-Y_{n} \|_{L^{2p}(\Omega, \mathbb{R}^{d})}.
\end{equation}
Here  the first error item 
$
\|X_{t_{n}}-\widetilde{Y}_{n} \|_{L^{2p}(\Omega, \mathbb{R}^{d})} 
$
can be directly and easily estimated (see the proof of Lemma \ref{lemma:strong-convergence-rate-of-sde-and-auxiliary-process}). From \eqref{introduction:modified-milstein-method}
and \eqref{equation:intro-auxiliary-process}, one can easily 
observe that $\widetilde{Y}_{n}-Y_{n} $ can be continuously extended to be an It\^o process and thus
the remaining item $\|\widetilde{Y}_{n}-Y_{n} \|_{L^{2p}(\Omega, \mathbb{R}^{d})}$ can be estimated with 
the aid of the It\^o formula (see the proof of Lemma \ref{lemma:strong-convergence-rate-of-MM-and-auxiliary-process}). Then a strong convergence rate of order $\tfrac12$
is derived for the scheme \eqref{introduction:modified-milstein-method}.

%
Based on the numerical scheme \eqref{introduction:modified-milstein-method} applied to discretize SDEs \eqref{equation:sde-in-introduction}, we adopt the antithetic MLMC approach
originally introduced by \cite{giles2014antithetic} to 
approximate the expectations and propose
the antithetic MLMC-modified Milstein method 
to approximate 
$
\E [
    \phi\left( X_{T}\right) 
    ]
$
for SDEs  with non-globally Lipschitz coefficiets.
In order to show the optimal complexity, the key element is to derive the $\mathcal{O}(h^{\beta}), \beta > 1$ variance of the multilevel estimator, in view of the well-known MLMC complexity theorem (see Theorem \ref{theorem:mlmc-complexity-theorem}).
However, this can not be achieved by trivially extending 
the analysis in \cite{giles2014antithetic} for a globally Lipschitz setting  to the non-globally Lipschitz setting in this work.
More precisely, some arguments used in \cite{giles2014antithetic} relying on the  
globally Lipschitz setting do not work in our setting.
To address this issue, we employ the previously obtained 
one-half convergence order to arrive at a sub-optimal estimate 
$\mathbb{E}
 [
 \big\|
 \overline{Y}_{n}^{f}   - Y_{n}^{c}
 \big\|^{2} 
 ]
 \leq
 C h,$
 which leads to  the $\mathcal{O}(h)$ variance (see Lemma \ref{lemma:general-error-estimate}) first.
Using this sub-optimal estimate and  carrying out 
more careful error estimates  on the mesh grids, 
we can improve the convergence rate to be order $1$, i.e., 
$\mathbb{E}
 [
 \big\|
 \overline{Y}_{n}^{f}   - Y_{n}^{c}
 \big\|^{2} 
 ]
 \leq
 C h^2,$
and hence deduce the $\mathcal{O}(h^{2})$ variance
as required (see Lemma \ref{lemma:final-variance-estimate}).


The contribution of the present article can be summarized 
as follows: 
{\color{black}
\begin{itemize}
\item A general framework of modified Milstein-type
schemes without Lévy areas
is established for SDEs with non-globally Lipschitz
coefficients and the strong convergence rate 
is revealed under a coupled monotonicity condition and certain polynomial growth conditions. The framework covers 
the tamed Milstein scheme and projected Milstein scheme
without Lévy areas as special cases;

\item Combining the proposed  schemes with 
the  antithetic multilevel Monte Carlo, we analyze the 
variance of the multilevel estimator  and obtain the order two
variance so that the optimal complexity $\mathcal{O}(\epsilon^{-2} )$ can be derived. This justifies the use of
antithetic MLMC combined with the newly proposed scheme 
for SDEs with non-globally Lipschitz coefficients.
\end{itemize}
}
As already mentioned above, the analysis of both the strong convergence rate and the desired variance is highly non-trivial 
and essential difficulties are caused by 
the non-globally Lipschitz setting, where the diffusion coefficients are allowed to grow polynomially.

The rest of this article is organized as follows.
The next section revisits the fundamentals of antithetic MLMC for SDEs with globally Lipschitz drift and diffusion. 
In Section \ref{Antithetic MLMC in non-globally Lipschitz setting}, we give the modified Milstein (MM) scheme and  the main result of this article under a non-globally Lipschitz setting.  
Then the strong convergence analysis of our numerical method is presented in Section \ref{section:strong-convergence-order-of-the-MM}.
In Section \ref{section:variance-analysis}, we reveal the variance analysis of the multilevel estimator constructed by the MM scheme.
Section \ref{section:numerical-exapmles} shows some numerical tests to illustrate our findings.
Finally, the Appendix contains the detailed proof of several lemmas and theorems.

\section{The antithetic MLMC revisited}
\label{section: antithetic mlmc}

The following setting is used throughout this paper. 
Let $\mathbb{N}$  be the set of all positive integers and $d,m \in \mathbb{N}$, $T \in (0,\infty)$ as given. Let $\|\cdot\|$ and $\langle\cdot, \cdot\rangle$ denote the Euclidean norm and the inner product of vectors in $\mathbb{R}^{d}$, respectively. 
{\color{black}
We denote $\|A\|_{F}:=\sqrt{\operatorname{trace}\left(A^{T} A\right)}$ as the trace norm of a matrix $A \in \mathbb{R}^{d \times m}$, where  $A^{T}\in \mathbb{R}^{m \times d}$ is represented  as
the transpose of a matrix $A \in \mathbb{R}^{d \times m}$. 
}
Let $(\Omega, \mathcal{F},\left\{\mathcal{F}_{t}\right\}_{t \in[0, T]}, \mathbb{P})$ be a  filtered probability space that satisfies the usual conditions.  
 Also, we use $L^{r}\left(\Omega, \mathbb{R}^{d}\right), r \in \mathbb{N}$, to denote a family of $\mathbb{R}^{d}$-valued random variables $\xi$ satisfying $\mathbb{E}\left[\|\xi\|^{r}\right]<\infty$
{\color{black} and use $L^{r}\left(\Omega, \mathbb{R}^{d \times m}\right), r \in \mathbb{N}$
to denote the $\mathbb{R}^{d \times m}$-valued random variables
$\Xi$ with 
$\mathbb{E}\left[\|\Xi\|_{F}^{r}\right]<\infty$.
}
For two real numbers $a$ and $b$, we denote $a \vee b = \max (a,b)$.  
In the following, we use $\frac{\partial \zeta}{\partial x}$ to denote the Jacobian matrix of the vector function $\zeta: \mathbb{R}^{d} \rightarrow \mathbb{R}^{d}$ 
{\color{black}
as
$\frac{\partial \zeta}{\partial x} 
:= \big(
\frac{\partial \zeta_{i}}{\partial x_{j}}
\big)_{1 \leq i, j \leq d}$.
}
For the real-valued function $\varphi: \mathbb{R}^{d} \rightarrow \mathbb{R}$, we use
{\color{black}
notations $\nabla \varphi$ and
$\nabla^{2} \varphi$
to denote its gradient vector and Hessian matrix, respectively, as
$
\nabla \varphi:=
(\tfrac{\partial \varphi}{\partial x_{1}},
\dots
\tfrac{\partial \varphi}{\partial x_{d}}
)$
and
$\nabla^{2} \varphi
:=
\big(
\frac{\partial^{2} \varphi}{ \partial x_{i} \partial x_{j}}
\big)_{1 \leq i, j \leq d}
$
.
} 
Moreover, let $\textbf{1}_{D}$ be the indicative function of a set $D$ and 
{\color{black}
let $C$ be a generic positive constant, which may be different for each appearance but is independent of the timestep.
}

%
In computational finance, it is common for the quantity of interest to be
\[
\mathbb{E} 
    [
    \phi\left( X_{T}\right)
    ],
\]
where $X_T$ is the solution of SDEs \eqref{equation:sde-in-introduction} at the final time $T$
and $\phi \in C_{b}^{2}(\mathbb{R}^{d}, \mathbb{R})$ is 
some smooth payoff function with first 
and second bounded derivatives, i.e., 
{\color{black}
$
\| 
\nabla \phi
\| 
\vee 
\| 
\nabla^{2} \phi 
\|_{F} 
   < \infty.
$
}
For simplicity, we denote
\begin{equation} \label{def:functional-P-in-mlmc}
P := \phi(X_{T}).
\end{equation} 
In addition, we denote $\hat{P}_{\ell}$ as the approximation of $P$ using a numerical discretization with timestep $h_{\ell} = \frac{T}{2^{\ell}}$, $\ell \geq 1$.
For some $L > 1$, the simulation of $\mathbb{E} [\hat P_{L}]$ can be split into a series of levels of 
{\color{black}resolution
}
as
\begin{equation} \label{equation:main-idea-of-mlmc}
\mathbb{E}
\big[
\hat{P}_{L}
\big]
=
\mathbb{E}
\big[
\hat{P}_{0}
\big]
+
\sum_{\ell=1}^{L}
\mathbb{E}
\big[
\hat{P}_{\ell}-\hat{P}_{\ell-1}
\big] .
\end{equation}

The idea behind MLMC is to independently estimate each of the expectations on the right-hand side of \eqref{equation:main-idea-of-mlmc}  in a way which minimises the overall variance for a given computational cost. Further,
let ${Z}_{0}$ be an estimator of $\mathbb{E}[\hat{P}_{0}]$ with $N_{0}$ Monte Carlo samples and ${Z}_{\ell}$ be an estimator of $\mathbb{E}[\hat{P}_{\ell}-\hat{P}_{\ell-1}]$ with $N_{\ell}$ Monte Carlo samples, i.e.
\begin{equation} \label{definition:estimator-of-multilevel-estimator}
Z_{\ell}
=\left\{
  \begin{array}{ll}
  N_{0}^{-1}\sum_{i=1}^{N_{0}}\hat{P}_{0}^{(i)}, & \ell=0, \\
  N_{\ell}^{-1}\sum_{i=1}^{N_{\ell}}
  \left(
  \hat{P}_{\ell}^{(i)}-\hat{P}_{\ell -1}^{(i)} 
  \right), 
  & \ell>0.
  \end{array}\right.
\end{equation}
Moreover, the final multilevel estimator $Z$ is given by the sum of the level estimators \eqref{definition:estimator-of-multilevel-estimator} as
 \begin{equation} \label{equation:final-multilevel-estimator}
Z=\sum_{\ell=0}^{L} Z_{\ell}  .
\end{equation}
The key point here is that $\hat{P}_{\ell}^{(i)}-\hat{P}_{\ell -1}^{(i)} $ 
should come from two discrete approximations for the same underlying stochastic sample, so that on finer levels of resolution the diﬀerence is small (due to strong convergence) and so the variance is also small. Hence very few samples will be required on ﬁner levels to accurately estimate the expected value.
%
Next we recall the MLMC complexity theorem in \cite{giles2008multilevel}.
\begin{theorem} \label{theorem:mlmc-complexity-theorem}
Let $P$, $Z_{\ell}$ and  $Z$ be defined as  \eqref{def:functional-P-in-mlmc}, \eqref{definition:estimator-of-multilevel-estimator} and \eqref{equation:final-multilevel-estimator}, respectively. 
Let 
{\color{black}
$\hat{P}_{\ell}$
}be the corresponding level $\ell$ numerical approximation. If there exist positive constants $\alpha $, $\beta$, $\theta$, $c_{1}$, $c_{2}$, $c_{3}$ such that $\alpha \geq \frac{1}{2}\min \{\beta , \theta\}$ and:
\begin{equation} 
\begin{aligned}
&\text { (i) }
  \left|
    \mathbb{E}\big[\hat{P}_{\ell}-P\big]
  \right| 
  \leq c_{1} h_{\ell}^{\alpha},  \\  \notag
&\text { (ii) } 
\mathbb{E}\left[Z_{\ell}\right]
=\left\{
  \begin{array}{ll}
  \mathbb{E}\big[\hat{P}_{0}\big], & \ell=0 \\
  \mathbb{E}\big[\hat{P}_{\ell}-\hat{P}_{\ell-1}\big], & \ell>0
  \end{array}\right.  \\
&\text { (iii) } 
  \text{Var} \left[Z_{\ell}\right] \leq c_{2} N_{\ell}^{-1} h_{\ell}^{\beta},  \\
&\text { (iv) } 
  \mathcal{C}_{\ell} \leq c_{3} N_{\ell} h_{\ell}^{-\theta}, 
\end{aligned}
\end{equation}
where $\mathcal{C}_{\ell}$ is the  computational complexity of $Z_{\ell}$, then  there exists a positive constant $c_{4}$ such that for any $\epsilon < e^{-1}$ there are values $L$ and $N_{\ell}$ \hspace{0.05em} for which the multilevel estimator \eqref{equation:final-multilevel-estimator}
has a mean-square-error with bound 
\begin{equation}
M S E \equiv \mathbb{E}\left[(Z-\mathbb{E}[P])^{2}\right]
<\epsilon^{2}  \notag
\end{equation}
with a  computational complexity $\mathcal{C}$ with bound
\begin{equation}  
\mathcal{C} 
\leq
\left\{
\begin{array}{ll}
c_{4} \epsilon^{-2}, & \beta>\theta \\
c_{4} \epsilon^{-2}(\log (1 / \epsilon))^{2}, & \beta=\theta \\  \notag
c_{4} \epsilon^{-2-(\theta-\beta) / \alpha}, & 0<\beta<\theta.
\end{array}\right.
\end{equation}
\end{theorem}
Usually, $\theta = 1$ and as indicated by Theorem \ref{theorem:mlmc-complexity-theorem},
the case $\beta > \theta = 1$ can
promise the optimal computational complexity $O(\epsilon^{-2})$.
Note that the strong Euler-type scheme gives one-half 
order of strong convergence and thus $\beta = \theta = 1$.
However, the Milstein scheme has strong convergence rate of
order $1$ so that $\beta = 2 > \theta = 1$ and the optimal computational complexity can be obtained.

For the underlying SDEs \eqref{equation:sde-in-introduction}, the classical Milstein scheme $\{\hat{X}_{n}\}_{0\leq n\leq N}$ using a uniform timestep $h=\frac{T}{N}$, $N\in \mathbb{N}$, takes the following form:
\begin{equation}
\left\{
\begin{array}{ll}
\hat{X}_{n+1} 
& =
\hat{X}_{n} + \mu(\hat{X}_{n})h + \sigma(\hat{X}_{n}) \Delta W_{n} +   \sum_{j_{1}, j_{2}=1}^{m} 
\left(
\mathcal{L}^{j_{1}} \sigma_{j_{2}}
\right)
(\hat{X}_{n})
\int^{t_{n+1}}_{t_n} \int^{s_2}_{t_n} \dd W_{s_1}^{j_1}
\dd W_{s_2}^{j_2}
\\
\hat{X}_{0} & = x_{0},
\end{array}\right.
\end{equation}
where $\Delta W_{n}:=W_{t_{n+1}}-W_{t_{n}}$, $n \in\{0,1,2, \ldots, N-1\}$, and the vector function $\mathcal{L}^{j_{1}} \sigma_{j_{2}}: \mathbb{R}^{d} \rightarrow \mathbb{R}^{d}$
is defined by
\begin{equation} \label{equation:milstein-term}
\mathcal{L}^{j_{1}} \sigma_{j_{2}}(x)
:=\sum_{k=1}^{d} \sigma_{k, j_{1}} 
\tfrac{
\partial \sigma_{j_{2}}(x)
}
{
\partial x^{k}
}
=\tfrac{
\partial \sigma_{j_{2}}
}
{
\partial x
}
(x) \sigma_{j_{1}}(x), 
\quad x \in \mathbb{R}^{d},
\quad j_{1}, j_{2}\in \{1, \dots ,m \} .
\end{equation}
Equivalently, the Milstein scheme can be rewritten as
\begin{equation} \label{equation:milstein-method}
\left\{
\begin{array}{ll}
\hat{X}_{n+1} 
& = \hat{X}_{n} + \mu(\hat{X}_{n})h + \sigma(\hat{X}_{n}) \Delta W_{n} +   \sum_{j_{1}, j_{2}=1}^{m} 
\left(
\mathcal{L}^{j_{1}} \sigma_{j_{2}}
\right)
(\hat{X}_{n}) 
\left(
\Pi_{j_{1}, j_{2}}^{t_{n}, t_{n+1}} + L_{j_{1}, j_{2}}^{t_{n}, t_{n+1}} 
\right)\\
\hat{X}_{0} & = x_{0},
\end{array}\right.
\end{equation}
where
the process $\Pi_{j_{1}, j_{2}}^{t_{n}, t_{n+1}} $ is given by
\begin{equation} \label{equation:truncated-process}
\Pi_{j_{1}, j_{2}}^{t_{n}, t_{n+1}}:= \tfrac{1}{2} \left(\Delta W_{j_{1}, n} \Delta W_{j_{2}, n}-\Omega_{j_{1} j_{2}} h \right), 
\end{equation}
with $\Delta W_{j, n}:=W_{j, t_{n+1}} - W_{j, t_{n}}$ 
for $j\in \{1,\dots,m \}$ and 
for $j_{1}, j_{2} \in \{1, \ldots, m\}$,
\begin{equation}
\Omega_{j_{1} j_{2}} 
= \left\{
\begin{array}{l}
1, \quad j_{1} = j_{2}, \\
0, \quad j_{1} \neq j_{2} .
\end{array} \right.
\end{equation}
Moreover, the term $L_{j_{1}, j_{2}}^{t_{n}, t_{n+1}}$ is the Lévy area denoted as
\begin{equation} \label{equation:levy-area}
    L_{j_{1}, j_{2}}^{t_{n}, t_{n+1}}
    := \int_{t_n}^{t_{n+1}}
    \left(
    W_{j_{1},t}-W_{j_{1},t_{n}}
    \right) 
    \mathrm{d} W_{j_{2},t}
    -\int_{t_n}^{t_{n+1}}
    \left(
    W_{j_{2},t}-W_{j_{2},t_{n}}
    \right) 
    \mathrm{d} W_{j_{1},t}.
\end{equation}

For SDEs with commutative diffusion coefficients 
(i.e. $d=m=1$),  
the  Milstein scheme \eqref{equation:milstein-method}, 
relying on only Brownian increments and thus easily implementable, can be combined with the multilevel Monte Carlo (MLMC) method to give $O(h^2)$ variance in $(iii)$ of 
Theorem \ref{theorem:mlmc-complexity-theorem} 
(i.e., $\beta = 2$) so that the optimal overall
computational cost $\mathcal{O}(\epsilon^{-2})$ can be achieved \cite{giles2008improved}.
However, for multi-dimensional SDEs without commutativity condition, to obtain the first order strong convergence, simulation of Lévy areas \eqref{equation:levy-area} is unavoidable but expensive. 
Once setting \eqref{equation:levy-area} to be zero, the best strong convergence order of such a truncated Milstein method is $\frac{1}{2}$, which only gives $O(h)$ variance in $(iii)$ of 
Theorem \ref{theorem:mlmc-complexity-theorem} 
(i.e., $\beta = 1$)  and the overall computational cost
is reduced to $\mathcal{O}(\epsilon^{-2}(\log (1 / \epsilon))^{2})$.
A natural question thus arises: can one obtain the optimal complexity without improving the strong convergence order of the numerical scheme?

To this end,  one needs to exploit some flexibility in the construction of the multilevel estimator. 
{
\color{black}
Instead of using
the same estimator for  $\hat{P}_{\ell}$ on every level $\ell$, as done in  \eqref{definition:estimator-of-multilevel-estimator}, 
Giles \cite{giles2008improved} 
showed that it could be more beneficial to use different estimators for the finer and coarser of the two levels being considered.
Inspired by this idea, we let $\hat{P}_{\ell}^f$ be the estimate of $P$ when level $\ell$ is
the finer level, and let $\hat{P}_{\ell}^c$ be the estimate of $P$ when level $\ell$ is the 
coarser level. 
}
In this setting we let $\hat{P}_{0}^{f} = \hat{P}_{0}^{c}$ and require
%
\begin{equation} \label{equation:requirement-of-modified-mlmc}
\mathbb{E} [\hat{P}_{\ell}^{f}] 
= \mathbb{E} [\hat{P}^{c}_{\ell}],
\quad \ell = 1, \dots, L,
\end{equation} 
so that the telescoping summation \eqref{equation:main-idea-of-mlmc} remains valid and turns to be
\begin{equation} \label{equation:idea-of-the-modified-mlmc}
\mathbb{E}
\left[
\hat{P}_{L}^{f}
\right]
=\mathbb{E}
\left[
\hat{P}_{0}^{f}
\right]
+
\sum_{\ell=1}^{L} 
\mathbb{E}
\left[
\hat{P}_{\ell}^{f}-\hat{P}_{\ell-1}^{c}
\right].
\end{equation}
With this modified estimator, the MLMC complexity theorem (cf. Theorem \ref{theorem:mlmc-complexity-theorem}) is still applicable and it gives the flexibility to construct approximations for which  $\hat{P}_{\ell}^{f}-\hat{P}_{\ell-1}^{c}$ is much smaller than the original $\hat{P}_{\ell}-\hat{P}_{\ell-1}$, giving a larger rate $\beta$ of  the relevant variance.

Following this idea, Giles and Szpruch \cite{giles2014antithetic} proposed 
an antithetic treatment which achieves the optimal complexity despite the $\frac{1}{2}$ strong convergence with globally Lipschitz coefficients. More precisely, using the coarser timestep $h \propto 2^{-\ell}$, the coarser path approximation $\{X^{c}\}$ is defined by a truncated Milstein method without Lévy areas as
\begin{equation}
\label{equation:truncated-milstein-method-in-a-coarse-step}
X^{c}_{n+1} 
= X^{c}_{n} + \mu(X^{c}_{n}) h 
+ \sigma(X^{c}_{n}) \Delta W_{n} 
+  \sum_{j_{1}, j_{2}=1}^{m} 
\left(
\mathcal{L}^{j_{1}} \sigma_{j_{2}}
\right)(X^{c}_{n}) 
\Pi_{j_{1}, j_{2}}^{t_{n}, t_{n+1}}.   
\end{equation}
The finer approximation $\{X^{f}\}$ consists of two steps,
first of which uses the same discretization as $\{ X^{c} \}$ with time stepsize $\frac{h}{2}$,
\begin{equation}
\label{equation:truncated-milstein-method-in-the-first-fine-step}
X^{f}_{n+1/2} 
= X^{f}_{n} + \mu(X^{f}_{n})\tfrac{h}{2} 
 + \sigma(X^{f}_{n}) \delta W_{n} 
 + \sum_{j_{1}, j_{2}=1}^{m} 
 \left(
 \mathcal{L}^{j_{1}} \sigma_{j_{2}}
 \right)(X^{f}_{n}) 
 \Pi_{j_{1}, j_{2}}^{t_{n}, t_{n+1/2}},  
\end{equation}
and the second of which reads
\begin{equation}
\label{equation:truncated-milstein-method-in-the-second-fine-step}
X^{f}_{n+1} 
= X^{f}_{n+1/2} + \mu(X^{f}_{n+1/2})\tfrac{h}{2} 
 + \sigma(X^{f}_{n}) \delta W_{n+1/2} 
 + \sum_{j_{1}, j_{2}=1}^{m} 
  \left(
  \mathcal{L}^{j_{1}} \sigma_{j_{2}}
  \right)(X^{f}_{n+1/2}) 
  \Pi_{j_{1}, j_{2}}^{t_{n+1/2}, t_{n+1}},  
\end{equation}
where
\begin{equation}
    \delta W_{n}=W_{t_{n+1/2}}-W_{t_{n}}, \quad 
    \delta W_{n+1/2}=W_{t_{n+1}}-W_{t_{n+1/2}}. \notag
\end{equation}
The antithetic counterpart $\{X^{a}\}$ is 
defined by exactly the same discretization as $\{X^{f}\}$,
with the exception that the Brownian increments $\delta W_{n}$ and $\delta W_{n+1/2}$ are swapped,
\begin{align} 
\label{equation: antithetictruncated-milstein-method-in-the-first-fine-step}
X^{a}_{n+1/2} 
& = X^{a}_{n} + \mu(X^{a}_{n})\tfrac{h}{2} 
 + \sigma(X^{a}_{n}) \delta W_{n+1/2} 
 + \sum_{j_{1}, j_{2}=1}^{m} 
 \left(
 \mathcal{L}^{j_{1}} \sigma_{j_{2}}
 \right)(X^{a}_{n}) 
 \Pi_{j_{1}, j_{2}}^{t_{n+1/2}, t_{n+1}},  
\\
\label{equation:antithetic-truncated-milstein-method-in-the-second-fine-step}
X^{a}_{n+1} 
& = X^{a}_{n+1/2} + \mu(X^{a}_{n+1/2})\tfrac{h}{2} 
 + \sigma(X^{a}_{n}) \delta W_{n} 
 + \sum_{j_{1}, j_{2}=1}^{m} 
 \left(
  \mathcal{L}^{j_{1}} \sigma_{j_{2}}
  \right)(X^{a}_{n+1/2}) 
  \Pi_{j_{1}, j_{2}}^{t_{n}, t_{n+1/2}}. 
\end{align}
Further, we let
\begin{equation} \label{definition:modified-estimator-Pf-and-Pc}
\hat{P}_{\ell}^{f} 
= \tfrac{1}{2} \big(
 \phi(X^{f}_{2^{\ell-1}}) + \phi(X^{a}_{2^{\ell-1}})  
 \big), 
 \quad 
 \hat{P}_{\ell-1}^{c} = \phi(X^{c}_{2^{\ell-1}}), 
 \quad \ell = 1, 2, \dots, L.
\end{equation}
As $\delta W_{n}$ and $\delta W_{n+1/2}$ are independent and identically distributed, $\{X^{a}\}$ has the same distribution as $\{X^{f}\}$ so that \eqref{equation:requirement-of-modified-mlmc} can be easily checked. 
With the help of 
Theorem \ref{theorem:mlmc-complexity-theorem}, Giles and Szpruch \cite{giles2014antithetic} deduced the optimal 
computational cost $\mathcal{O}(\epsilon ^{-2})$ for the 
 antithetic estimator they proposed under globally Lipschitz 
 conditions, which is summarized as follows.
\begin{theorem} 
\label{theorem:complexity-of-antithetic-mlmc-with-lipschitz-sde}
Let both drift and diffusion coefficients of SDE \eqref{equation:sde-in-introduction} be globally Lipschitz continuous with first and second uniformly bounded derivative.
Let $P$, $\hat{P}^{f}_{\ell}$  and $\hat{P}^{c}_{\ell-1}$, $\ell = 1, 2, \dots, L$, be defined by \eqref{def:functional-P-in-mlmc} and \eqref{definition:modified-estimator-Pf-and-Pc}. Moreover, let ${Z}_{0}$ be an estimator of $\mathbb{E}[\hat{P}^{f}_{0}]$ with $N_{0}$ Monte Carlo samples and ${Z}_{\ell}$ be an estimator of $\mathbb{E}[\hat{P}^{f}_{\ell}-\hat{P}^{c}_{\ell-1}]$ with $N_{\ell}$ Monte Carlo samples, i.e.
\begin{equation} 
Z_{\ell}
=\left\{
 \begin{array}{ll}
 N_{0}^{-1}\sum_{i=1}^{N_{0}}\hat{P}_{0}^{c,(i)}, & \ell=0, \\
 N_{\ell}^{-1}\sum_{i=1}^{N_{\ell}}
 \left(
  \hat{P}_{\ell}^{f,(i)}-\hat{P}_{\ell -1}^{c,(i)} 
  \right), & \ell>0.
\end{array}\right.
\end{equation}
The final multilevel estimator $Z$ is given by
 \begin{equation} 
Z=\sum_{\ell=0}^{L} Z_{\ell}  .
\end{equation}
Then, there exists a positive constant $C$ such that
\begin{equation}
    \text{Var}\big[Z_{\ell} \big] 
    \leq C N_{\ell}^{-1} h^{2}_{\ell}.
\end{equation}
Given the mean-square-error of $Z$ with bound
\begin{equation}
M S E 
\equiv 
\mathbb{E}
\left[
(Z-\mathbb{E}[P])^{2}
\right]
<\epsilon^{2}, \notag
\end{equation}
there exists a constant $C$ such that the complexity $\mathcal{C}_{antitheticMLMC}$ has the bound
\begin{equation}
   \mathcal{C}_{antitheticMLMC} \leq C \epsilon^{-2}.
\end{equation}
\end{theorem}
In the forthcoming sections, 
we turn to SDEs with non-globally Lipschitz continuous coefficients and try to recover the above results in a
non-globally Lipschitz setting. To this aim, we first need to
introduce a modified and strongly convergent Milstein-type numerical scheme and later construct 
antithetic multilevel estimators following the same idea of
\cite{giles2014antithetic}. 
However, the analysis for both the convergence rate and 
the desired variance  is highly non-trivial and non-standard arguments are developed to overcome essential difficulties 
caused by the super-linearly growing coefficients.

\section{Antithetic MLMC in a non-globally Lipschitz setting} \label{Antithetic MLMC in non-globally Lipschitz setting}
In this section, we set up a non-globally Lipschitz framework by presenting some assumptions.
\subsection{Settings and the modified Milstein scheme}
We consider the following assumptions required to establish our main result.
\begin{assumption} \label{assumption:coercivity-condition}
(Coercivity condition) For some $p_{0} \geq 1$, there exists a positive constant $C$ such that
\begin{equation} 
2\langle x, \mu(x) \rangle
+(2p_{0}-1)
{\color{black}
\|\sigma(x)\|_{F}^{2} 
}
\leq C\left(1+\|x\|^{2}\right), 
\quad \forall x \in \mathbb{R}^{d}.
\end{equation}
\end{assumption}
\begin{assumption}\label{assumption:polynomial-growth-condition}
(A coupled {\color{black} monotonicity} condition and globally polynomial growth conditions). Assume the drift coefficients $\mu:\mathbb{R}^{d}\rightarrow \mathbb{R}^{d}$ of the SDEs \eqref{equation:sde-in-introduction} are continuously differentiable and the diffusion coefficients $\sigma_{j}:\mathbb{R}^{d}\rightarrow \mathbb{R}^{d}$ are twice continuously differentiable. Moreover, there exist some positive constants $\gamma \in [1,\infty)$, $p_{1} \in (1, \infty)$, 
such that
\begin{equation}
2\langle 
x-\tilde{x}, \mu(x) - \mu(\tilde{x})
\rangle
+ (2p_{1}-1) 
{\color{black}
\|\sigma(x) - \sigma(\tilde{x})\|_{F}^{2}
}
\leq C\|x - \tilde{x}\|^{2}, 
\quad \forall x , \tilde{x} \in \mathbb{R}^{d} ,
\end{equation}
and for $j \in\{1, \ldots, m\}$,
\begin{equation} \label{equation:growth-of-partial-drift-and-diffusion-in-assumption}
 \begin{aligned}
\left\|
 \left(
 \tfrac{\partial \mu}{\partial x}(x)
 -\tfrac{\partial \mu}{\partial x}(\tilde{x}) 
 \right) y 
\right\| 
&\leq 
C(1+\|x\|+\|\tilde{x}\|)^{\gamma-2}
\|x-\tilde{x}\| 
\cdot \| y \| , 
\quad \forall x, \tilde{x}, y \in \mathbb{R}^{d} ,\vspace{3ex} \\
\left\|
\left(
 \tfrac{\partial \sigma_{j}}{\partial x}(x)-\tfrac{\partial \sigma_{j}}{\partial x}(\tilde{x}) 
 \right) y 
 \right\|^{2} 
 &\leq 
 C(1+\|x\|+\|\tilde{x}\|)^{\gamma-3}
 \|x-\tilde{x}\|^{2} 
 \cdot \| y \|^{2} , 
 \vspace{3ex} 
 \quad \forall x, \tilde{x}, y \in \mathbb{R}^{d}.\\
\end{aligned}
\end{equation}
Assume that the vector functions $ \mathcal{L}^{j_{1}} \sigma_{j_{2}}: \mathbb{R}^{d} \rightarrow \mathbb{R}^{d},  j_{1}, j_{2} \in\{1, \ldots, m\}$ are continuously differentiable and
\begin{equation} \label{equation:growth-of-partial-milstein-diffusion-term-in-assumption}
 \left\|
 \left(
 \tfrac{
 \partial
 \left(
\mathcal{L}^{j_{1}} \sigma_{j_{2}}
\right)
     }
{\partial x}(x)
-\tfrac{
\partial
\left(
\mathcal{L}^{j_{1}} \sigma_{j_{2}}
\right)
      }
{\partial x}(\tilde{x})
\right) y 
\right\| 
\leq 
C\left(1+\|x\|+\|\tilde{x}\|\right)^{\gamma-2}\|x-\tilde{x}\| \cdot\| y \| , 
 \quad \forall x, \tilde{x}, y \in \mathbb{R}^{d}.
\end{equation}
Further, the initial data $X_{0}$ is supposed to be $\mathcal{F}_{0}$-measurable, satisfying
\begin{equation}
\mathbb{E}\left[\left\|X_{0}\right\|^{2p_{0} }\right] 
< \infty,
\end{equation}
where $p_0$ is determined by Assumption \ref{assumption:coercivity-condition}.
\end{assumption} 
Assumption \ref{assumption:polynomial-growth-condition} is considered as a kind of polynomial growth condition and in proofs which follow we will need some implications of  this assumption. It follows immediately that, for $j \in\{1, \ldots, m\}$,
\begin{equation} \label{equation:growth-of-partial-drift-and-diffusion}
\begin{aligned}
&\left\|
 \tfrac{\partial \mu}{\partial x}(x) y 
 \right\| 
 \leq C(1+\|x\|)^{\gamma-1} \| y \|, 
 \quad \forall x, y \in \mathbb{R}^{d} , \\
& \left\|
 \tfrac{\partial \sigma_{j}}{\partial x}(x) y 
 \right\|^{2} 
 \leq C(1+\|x\|)^{\gamma-1} \| y \|^{2}, 
 \quad \forall x ,y \in \mathbb{R}^{d},   \\
\end{aligned}
\end{equation}
which in turn gives
\begin{equation} \label{equation:growth-of-drift}
\begin{aligned}
\|\mu(x)-\mu(\tilde{x})\| 
& \leq C(1+\|x\|+\|\tilde{x}\|)^{\gamma-1}\|x-\tilde{x}\|, 
\quad \forall x, \tilde{x} \in \mathbb{R}^{d}, \\
\|\mu(x)\| 
& \leq C(1+\|x\|)^{\gamma}, 
\quad \forall x \in \mathbb{R}^{d},  \\
\end{aligned}
\end{equation}
and, for $j_{1}, j_{2} \in\{1, \ldots, m\}$,
\begin{equation} \label{equation:growth-of-diffusion}
\begin{aligned}
\left\|
 \sigma_{j}(x)-\sigma_{j}(\tilde{x})
\right\|^{2}
&\leq 
C(1+\|x\|+\|\tilde{x}\|)^{\gamma-1}\|x-\tilde{x}\|^{2}, 
 \quad \forall x, \tilde{x} \in \mathbb{R}^{d},    \\
\|\sigma_{j}(x)\|^{2} 
&\leq C(1+\|x\|)^{\gamma+1}, 
\quad \forall x \in \mathbb{R}^{d}.  \\
\end{aligned}
\end{equation}
Similarly, \eqref{equation:growth-of-partial-milstein-diffusion-term-in-assumption} in Assumption \ref{assumption:polynomial-growth-condition} yields, 
\begin{equation} \label{equation:growth-of-milstein-diffusion-term}
\begin{aligned}
\left\|
\mathcal{L}^{j_{1}} \sigma_{j_{2}}(x)
 -\mathcal{L}^{j_{1}} \sigma_{j_{2}}(\tilde{x})
\right\| 
& \leq 
C(1+\|x\|+\|\tilde{x}\|)^{\gamma-1}\|x-\tilde{x}\|, 
\quad \forall x, \tilde{x} \in \mathbb{R}^{d}, \\
\left\|
\mathcal{L}^{j_{1}} \sigma_{j_{2}}(x)
\right\|
& \leq 
C(1+\|x\|)^{\gamma}, 
\quad \forall x \in \mathbb{R}^{d}.  \\
\end{aligned}
\end{equation}
Note that Assumption \ref{assumption:polynomial-growth-condition} with Assumption \ref{assumption:coercivity-condition} suffices to imply that SDE \eqref{equation:sde-in-introduction} possesses a unique $\{\mathcal{F}_{t} \}_{t\in [0,T]}$-adapted solution with continuous sample paths and 
\begin{equation} 
\mathbb{E} 
\left[
 \left\|X_{t}\right\|^{2p}
 \right]
 \leq C
   \left(
    1+ \mathbb{E} \left[\left\|X_{0}\right\|^{2p}\right] 
    \right)
  <\infty , 
  \quad p\in [1,p_{0}], \quad t\in [0,T].
\end{equation}

These assumptions cover a much broader class of SDEs than the globally Lipschitz setting.
{\color{black}
However, under Assumptions \ref{assumption:coercivity-condition}, \ref{assumption:polynomial-growth-condition},  both drift and diffusion coefficients of  SDEs \eqref{equation:sde-in-introduction} are allowed to grow polynomially and violate the globally Lipschitz condition. In this situation, as revealed by \cite{hutzenthaler2011strong}, the usual explicit time-stepping scheme such as the Euler-Maruyama (EM) scheme would produce numerical approximations with divergent moments, when the time-step size $h>0$ shrinks to zero.
This leads to divergence of the explicit scheme in both strong and weak senses and also divergence of the multilevel Monte Carlo Euler method (see \cite{hutzenthaler2011strong} and \cite{hutzenthaler2013divergence} for relevant divergence results).
Taking these findings into account and noting that the truncated Milstein scheme \eqref{equation:truncated-milstein-method-in-a-coarse-step} coincides with the EM scheme for SDEs with additive noise, one can not expect to obtain convergence of the truncated Milstein scheme and the antithetic MLMC method in the above non-globally Lipschitz setting.
}
To remedy it, we propose a family of 
modified Milstein (MM) scheme on a uniform timestep $h=\frac{T}{N}$, $N \in \mathbb{N}$, given by
\begin{equation} \label{equation:modified-milstein-method}
\left\{
\begin{array}{l}
Y_{n+1}
= \mathscr{P}(Y_{n}) +  \mu_{h}\big(\mathscr{P}(Y_{n}) \big) h + \sigma_{h}\big(\mathscr{P}(Y_{n}) \big) \Delta W_{n}+\sum_{j_{1}, j_{2}=1}^{m} \left(\mathcal{L}^{j_{1}} \sigma_{j_{2}}\right)_{h}\big(\mathscr{P}(Y_{n}) \big) \Pi_{j_{1}, j_{2}}^{t_{n}, t_{n+1}} \\
Y_{0} = x_{0}.
\end{array}\right.
\end{equation}
where  $\mathscr{P}:\mathbb{R}^{d} \rightarrow \mathbb{R}^{d}$
is 
{\color{black}
a kind of modification function}
that can be customised and
$\mu_{h}(x)$, $\left(\mathcal{L}^{j_{1}} \sigma_{j_{2}}\right)_{h}(x)$, $j_{1}, j_{2}\in \{1, \dots,m \}$,
$\sigma_{h}(x)$ are measurable functions taking values in $\mathbb{R}^{d}$ or $\mathbb{R}^{d\times m}$.
In addition, 
we need some conditions imposed on the framework for 
the time-stepping scheme.
\begin{assumption} \label{assumption:condition-on-the-MM}
(1) For any $ x\in \mathbb{R}^{d}$ and $j_{1}, j_{2}
\in \{1, \dots,m \}$,
\begin{equation}
\begin{aligned}
\|\mathscr{P}(x) \| 
&\leq \| x\| , \\
\left\|
\mu_{h}
\left(
\mathscr{P}(x)
\right) 
\right\| 
& \leq 
\color{black}
Ch^{-\frac{1}{2}}(1+\| x\|), 
\\
\color{black}
\|\sigma_{h}(x) \|_{F} 
& \leq  
\color{black}
\| \sigma(x)\|_{F} , \\
\left\| 
\left(
\mathcal{L}^{j_{1}} \sigma_{j_{2}}
\right)_{h}
\left(
\mathscr{P}(x)
\right)
\right\| 
&\leq 
\color{black}
Ch^{-\frac{1}{2}} (1+\| x\|).
\end{aligned}
\end{equation}
(2) For any $ x, y \in \mathbb{R}^{d}$, 
\begin{equation}
 \|\mathscr{P}(x) - \mathscr{P}(y) \| 
 \leq \|x-y \|.
\end{equation}
(3) There exists some constant $\mathbf{a} \in [1, \infty)$ such that, for some $p \in [1, \infty)$ and $\zeta \in L^{2\mathbf{a}p}(\Omega, \mathbb{R}^{d})$,
\begin{equation}
\|\zeta-\mathscr{P}(\zeta) \|_{L^{2p}(\Omega, \mathbb{R}^{d})} 
\leq 
Ch^{2} \|\zeta \|_{L^{2\mathbf{a}p}(\Omega, \mathbb{R}^{d})}.
\end{equation}
(4) For any $x \in \mathbb{R}^{d}$,
\begin{equation}
2\langle x, \mu_{h}(x)\rangle
+(2p_{0}-1)
{\color{black}
\|\sigma_{h}(x)\|_{F}^{2}
}
\leq C\left(1+\|x\|^{2}\right).
\end{equation}
(5) There exist some constant $\alpha_{1}, \alpha_{2}, \alpha_{3} \in [0,\infty)$ such that, for  $j_{1}, j_{2}\in \{1, \dots,m \}$ and $\forall x \in \mathbb{R}^{d}$,
\begin{equation}
\begin{aligned}
\|\mu(x) - \mu_{h}(x) \| 
&\leq 
Ch (1+ \| x\|^{\alpha_{1}}), \\
\color{black}
\|\sigma(x) - \sigma_{h}(x) \|_{F} 
&\leq
Ch (1+ \| x\|^{\alpha_{2}}),\\
\left\|
\mathcal{L}^{j_{1}} \sigma_{j_{2}}(x) - \left(\mathcal{L}^{j_{1}} \sigma_{j_{2}}\right)_{h}(x)
\right\| 
&\leq 
Ch (1+ \| x\|^{\alpha_{3}}).
\end{aligned}
\end{equation}
(6)There exists some constant $\mathbf{b} \in [1, \infty)$ such that, for some $p \in [1, \infty)$ and $\zeta,\upsilon \in L^{2\mathbf{b}p}(\Omega, \mathbb{R}^{d})$,
\begin{equation}
\left\|
\tfrac{1}{2}
\big(\mathscr{P}(\zeta)+\mathscr{P}(\upsilon) \big)
- \mathscr{P}\big(\tfrac{1}{2}(\zeta+\upsilon)\big)
\right\|_{L^{2p}(\Omega, \mathbb{R}^{d})} 
\leq 
Ch^{2} 
\left(
\|\zeta \|_{L^{2\mathbf{b}p}(\Omega, \mathbb{R}^{d})}
+ \|\upsilon \|_{L^{2\mathbf{b}p}(\Omega, \mathbb{R}^{d})}
\right).
\end{equation}
\end{assumption}
The framework given by Assumption \ref{assumption:condition-on-the-MM} is general and covers the tamed Milstein scheme and projected Milstein scheme
without Lévy areas as special cases.

\textbf{Example 1:  Tamed Milstein schemes  without Lévy areas.} 

\textit{Model 1 (TMS1):}
Set
\begin{equation}
\label{equation:tamed-milstein-scheme-model-1}
\begin{aligned}
\mathscr{P}&:=I,\\
\mu_{h}\left(\mathscr{P}(x)\right) 
:= \tfrac{\mu(x)}{1+ \| \mu(x) \|^{2}h}, \quad
\sigma_{h}\left(\mathscr{P}(x)\right) 
&:= \tfrac{\sigma(x)}{1+ \| \mu(x) \|^{2}h}, \quad \left(\mathcal{L}^{j_{1}} \sigma_{j_{2}}\right)_{h}\left(\mathscr{P}(x)\right) := \tfrac{\mathcal{L}^{j_{1}} \sigma_{j_{2} }\left(x\right)}{1+ \| \mathcal{L}^{j_{1}} \sigma_{j_{2} } (x)\|^{2}h} ,
\end{aligned}
\end{equation}
where we denote $I$ as the identity operator.
In this case, one observes  that conditions \textit{(1)}-\textit{(3)} and \textit{(6)} in Assumption \ref{assumption:condition-on-the-MM} is obviously satisfied with $\mathbf{a}=\mathbf{b}=1$. Moreover, it is worth noting that such tamed coefficients $\mu_{h}(x)$ and $\sigma_{h}(x)$ can preserve the condition \textit{(4)} as, for any $x\in \mathbb{R}^{d}$,
\begin{equation} 
\begin{aligned}
2\langle  x, \mu_{h} (x) \rangle 
+  (2p_{0}-1)
{\color{black}\| \sigma_{h}(x) \|_{F}^{2} }
&= \tfrac{ 2 \langle  x, \mu(x) \rangle}
{ 1+ \| \mu(x) \|^{2} h } 
+ 
\tfrac{
(2p_{0}-1)
{\color{black}
\|\sigma(x)\|_{F}^{2}
}
}
{
\left(1+ \| \mu(x) \|^{2}h\right)^{2}
}  
& \leq 
C \left(1+\|x\|^{2}\right).
\end{aligned}
\end{equation}
By Assumption \ref{assumption:polynomial-growth-condition}, we are able to check that \textit{(5)} is met with $\alpha_{1}=\alpha_{3}=3\gamma$, $\alpha_{2}=\frac{5\gamma+1}{2}$.

\textit{Model 2 (TMS2): }
Set
\begin{equation} 
\begin{aligned}
\mathscr{P}&:=I,\\
\mu_{h}\left(\mathscr{P}(x)\right) 
:= \tfrac{\mu(x)}{1+ \| x \|^{2(\gamma-1)}h}, \quad
\sigma_{h}\left(\mathscr{P}(x)\right) 
&:= \tfrac{\sigma(x)}{1+ \| x \|^{2(\gamma-1)}h}, \quad
\left(
\mathcal{L}^{j_{1}} \sigma_{j_{2}}
\right)_{h}
\left(\mathscr{P}(x)\right) 
:= \tfrac{
\mathcal{L}^{j_{1}} \sigma_{j_{2} }\left(x\right)
}{
1+ \| x\|^{2(\gamma-1)}h
} .
\end{aligned}
\end{equation}
This model (\textit{TMS2}) above is derived from \cite{kumar2019milstein}. Similar to the analysis for \textit{Model 1 (TMS1)}, all conditions in Assumption \ref{assumption:condition-on-the-MM} are satisfied with $\mathbf{a}=\mathbf{b}=1$, $\alpha_{1}=\alpha_{3}=3\gamma-2$ and $\alpha_{2}=\frac{5\gamma-3}{2}$.

\textbf{Example 2:  Projected Milstein schemes without Lévy areas.} 

{\color{black}
\textit{Model 3 (PMS):}
}
Set
\begin{equation} \label{equation:projected-Milstein-scheme}
\begin{aligned}
\mathscr{P}(x)&:= \min\{1, h^{-\frac{1}{2\gamma}} \| x\|^{-1} \}x, \\
\mu_{h}\left(\mathscr{P}(x)\right)
:=\mu\left(\mathscr{P}(x)\right), \quad
\sigma_{h}\left(\mathscr{P}(x)\right)
&:=\sigma\left(\mathscr{P}(x)\right), \quad
\left(\mathcal{L}^{j_{1}} \sigma_{j_{2}}\right)_{h}\left(\mathscr{P}(x)\right) 
:= \mathcal{L}^{j_{1}} \sigma_{j_{2}}\left(\mathscr{P}(x)\right) .
\end{aligned}
\end{equation}
Evidently, conditions \textit{(1)}, \textit{(4)} and \textit{(5)} can be fulfilled by \eqref{equation:growth-of-drift}-\eqref{equation:growth-of-milstein-diffusion-term} with $\alpha_{1}=\alpha_{2}=\alpha_{3}=0$ since the coefficients are not modified. Conditions \textit{(2)} and \textit{(3)} can be easily derived with  $\textbf{a}=4\gamma+1$, by following similar ideas in \cite{beyn2016stochastic,beyn2017stochastic}.
Now it remains to check  the condition \textit{(6)}.
%
For simplicity we denote $\overline{\zeta}:=\frac{1}{2}(\zeta+\upsilon)$
and for two random variables $\zeta, \upsilon$ 
we introduce two measurable sets:
\begin{equation}
\mathcal{A}_{h} 
:= 
\{
\omega \in \Omega : \|\zeta(\omega) \| \leq h^{-\frac{1}{2\gamma}} \ \text{and} \ 
\|\upsilon(\omega) \| \leq h^{-\frac{1}{2\gamma}} 
\}, 
\quad \mathcal{A}_{h}^{c} 
:= \Omega \backslash \mathcal{A}_{h} .
\end{equation}
Then, by using the H\"older inequality, we have, for $p\in [1,\infty)$ and $q, q' \in (1,\infty)$ with $\tfrac{1}{q}+\tfrac{1}{q'}=1$, 
\begin{equation}
\begin{aligned}
&\mathbb{E}
\left[ \left\|
\tfrac{1}{2} \big(\mathscr{P}(\zeta)+\mathscr{P}(\upsilon) \big)
- \mathscr{P}\big(\overline{\zeta}\big)
\right\|^{2p}\right] \\
&=\mathbb{E}
\left[ \left\|
\tfrac{1}{2} \big(\mathscr{P}(\zeta)+\mathscr{P}(\upsilon) \big)
- \mathscr{P}\big(\overline{\zeta}\big)
\right\|^{2p} \textbf{1}_{\mathcal{A}_{h}^{c}}
\right] \\
& \leq \left\|
\tfrac{1}{2} \big(\mathscr{P}(\zeta)+\mathscr{P}(\upsilon) \big)
- \mathscr{P}\big(\overline{\zeta}\big)
\right\|_{L^{2pq}(\Omega, \mathbb{R}^{d})}^{2p} \times
\|\textbf{1}_{\mathcal{A}_{h}^{c}} \|^{2p}_{L^{q'} (\Omega, \mathbb{R})},
\end{aligned}
\end{equation}
where the Markov inequality implies
\begin{equation}
\begin{aligned}
\|\textbf{1}_{\mathcal{A}_{h}^{c}} \|_{L^{q'} (\Omega, \mathbb{R})} 
&= \bigg( 
\mathbb{P} 
\left( 
\omega \in \Omega :
\|\zeta(\omega) \| > h^{-\frac{1}{2\gamma}} \ \text{or} 
\ 
\|\upsilon(\omega) \| > h^{-\frac{1}{2\gamma}} 
\right)
\bigg)^{\frac{1}{q'}} \\
& \leq  
\bigg( 
\mathbb{P} 
\left( 
\omega \in \Omega :
\|\zeta(\omega) \| > h^{-\frac{1}{2\gamma}}  
\right)
+ 
\mathbb{P} 
\left( 
\omega \in \Omega :
\|\upsilon(\omega) \| > h^{-\frac{1}{2\gamma}}  
\right)
\bigg)^{\frac{1}{q'}} \\
& \leq  
h^{\frac{\beta}{2\gamma q'}}
\Big(
\|\zeta \|^{\frac{\beta}{q'}}_{L^{\beta} (\Omega, \mathbb{R}^{d})}
+ 
\|\upsilon \|^{\frac{\beta}{q'}}_{L^{\beta} (\Omega, \mathbb{R}^{d})}
\Big).
\end{aligned}
\end{equation}
By choosing $q=4\gamma+1$, $q'=1+\tfrac{1}{4\gamma}$,  $\beta=8\gamma+2$, the condition \textit{(6)} is fulfilled
with $\textbf{b}=4\gamma+1$ and all conditions in 
Assumption \ref{assumption:condition-on-the-MM} 
are validated for the  proposed projected Milstein scheme.

Before proceeding further, we would like to give some useful estimates which follows from Assumption \ref{assumption:condition-on-the-MM}. In view of conditions
$\textit{(1)}$, $\textit{(4)}$ and the Cauchy-Schwarz inequality, we obtain a prior estimate of 
{\color{black}
$\|\sigma_{h}(\mathscr{P}(x)) \|_{F}$ 
}as follows:
\begin{equation} \label{equation:upper-bound-of-diffusion}
\begin{aligned}
{\color{black}
\| \sigma_{h}(\mathscr{P}(x)) \|_{F}^{2} 
}
&\leq C  \left(1+\|\mathscr{P}(x)\|^{2} \right) 
+ \| \mathscr{P}(x) \|\|\mu_{h}(\mathscr{P}(x)) \|   
\leq
Ch^{-1/2} \left(1+\|\mathscr{P}(x)\|^{2} \right).
\end{aligned}
\end{equation}
Moreover, combining condition \textit{(3)}, \eqref{equation:growth-of-drift}-\eqref{equation:growth-of-milstein-diffusion-term} with the H\"older inequality yields, for some $p\in [1,\infty)$, $j_{1}, j_{2}\in \{1, \dots,m \}$ and some random variable $\zeta \in \mathbb{R}^{d}$,
\begin{equation}
\begin{aligned}
\|
\mu(\zeta) - \mu(\mathscr{P}(\zeta)) 
\|_{L^{2p}(\Omega, \mathbb{R}^{d}) } 
&\leq Ch^{2} \|\zeta \|_{L^{2p(\mathbf{a} + \gamma -1)}(\Omega, \mathbb{R}^{d})}, \\
\|
\sigma(\zeta) - \sigma(\mathscr{P}(\zeta)) 
\|_{L^{2p}(\Omega, \mathbb{R}^{\color{black}d \times m}) } 
&\leq Ch^{2} \|\zeta \|_{L^{p(2\mathbf{a} + \gamma -1)}(\Omega, \mathbb{R}^{d})}, \\
\|
\mathcal{L}^{j_{1}} \sigma_{j_{2}}(\zeta) - \mathcal{L}^{j_{1}} \sigma_{j_{2}}( \mathscr{P}(\zeta) ) 
\|_{L^{2p}(\Omega, \mathbb{R}^{d}) } 
&\leq 
Ch^{2} \|\zeta \|_{L^{2p(\mathbf{a} + \gamma -1)}(\Omega, \mathbb{R}^{d})}.
\end{aligned}
\end{equation}
%
\subsection{The antithetic MLMC method and main result}

Under Assumptions \ref{assumption:coercivity-condition}, \ref{assumption:polynomial-growth-condition}, we propose 
the antithetic MLMC-MM method to approximate \eqref{introduction:expectation-of-functional-of-sde},
where time is discretized by 
{\color{black}the MM method \eqref{equation:modified-milstein-method}} and expectations are approximated by the antithetic MLMC method.
Similar to Section \ref{section: antithetic mlmc},  
with the coarser timestep $h\propto 2^{-\ell}$ we  construct the coarser path approximation $\{Y^{c}\}$ as
\begin{equation} \label{equation:MM-in-coarse-step}
\begin{aligned}
Y_{n+1}^{c}
&=\mathscr{P}(Y^{c}_{n}) 
+ \mu_{h}
\big(
\mathscr{P}(Y^{c}_{n})
\big) h
+\sigma_{h}
\big(
\mathscr{P}(Y^{c}_{n})
\big) \Delta W_{n} 
+ \sum_{j_{1},j_{2}=1}^{m} 
\left(
\mathcal{L}^{j_{1}} \sigma_{j_{2}}
\right)_{h}
\big(
\mathscr{P}(Y^{c}_{n})
\big) 
\Pi_{j_{1}, j_{2}}^{t_{n}, t_{n+1}} .\\
\end{aligned}
\end{equation}
Subsequently, 
we define the corresponding two half-timesteps of the first finer path approximation $\{Y^{f}\}$ with the finer timestep $\frac{h}{2}$ as follows,
\begin{equation} \label{equation:MM-in-first-fine-timestep}
\begin{aligned}
Y_{n+1/2}^{f}
&=
\mathscr{P}\big(Y_{n}^{f} \big) 
+ \mu_{\frac{h}{2}}
\big(
\mathscr{P}\big(Y_{n}^{f} \big)
\big) \tfrac{h}{2}
+\sigma_{\frac{h}{2}}
\big(
\mathscr{P}\big(Y_{n}^{f} \big)
\big) \delta W_{n} 
+ \sum_{j_{1},j_{2}=1}^{m} 
\left(
\mathcal{L}^{j_{1}} \sigma_{j_{2}}
\right)_{\frac{h}{2}}
\big(
\mathscr{P}\big(Y_{n}^{f} \big)
\big) 
\Pi_{j_{1}, j_{2}}^{t_{n}, t_{n+1/2}} ,\\
\end{aligned}
\end{equation}
and
\begin{equation} \label{equation:MM-in-second-fine-timestep}
\begin{aligned}
Y_{n+1}^{f}
=\mathscr{P}\big(Y_{n+1/2}^{f} \big) 
+ \mu_{\frac{h}{2}}
\big(
\mathscr{P}\big(Y_{n+1/2}^{f} \big)
\big) \tfrac{h}{2}
&+\sigma_{\frac{h}{2}}
\big(
\mathscr{P}\big(Y_{n+1/2}^{f} \big)
\big) \delta W_{n+1/2} \\
&+\sum_{j_{1},j_{2}=1}^{m} 
\left(
\mathcal{L}^{j_{1}} \sigma_{j_{2}}
\right)_{\frac{h}{2}}
\big(
\mathscr{P}\big(Y_{n+1/2}^{f} \big)
\big)
\Pi_{j_{1}, j_{2}}^{t_{n+1/2}, t_{n+1}}.
\end{aligned}
\end{equation}
And the antithetic estimator $\{Y^{a}\}$ is given by 
\begin{equation}
\label{equation:antithetic-estimator-in-first-fine-step}
\begin{aligned}
Y_{n+1/2}^{a}
&=\mathscr{P}\big(Y_{n}^{a}\big) 
+ \mu_{\frac{h}{2}}
\big(
\mathscr{P}\big(Y_{n}^{a}\big)
\big) \tfrac{h}{2}
+\sigma_{\frac{h}{2}}
\big(
\mathscr{P}\big(Y_{n}^{a}\big)
\big) \delta W_{n+1/2} 
+ \sum_{j_{1},j_{2}=1}^{m} 
\left(
\mathcal{L}^{j_{1}} \sigma_{j_{2}}
\right)_{\frac{h}{2}}
\big(
\mathscr{P}\big(Y_{n}^{a}\big)
\big)
\Pi_{j_{1}, j_{2}}^{t_{n+1/2}, t_{n+1}},\\
\end{aligned}
\end{equation}
and
\begin{equation}
\label{equation:antithetic-estimator-in-second-fine-step}
\begin{aligned}
Y_{n+1}^{a}
=\mathscr{P}\big(Y_{n+1/2}^{a}\big) 
+ \mu_{\frac{h}{2}}
\left(
\mathscr{P}\big(Y_{n+1/2}^{a}\big)
\right) \tfrac{h}{2}
&+\sigma_{\frac{h}{2}}
\left(
\mathscr{P}\big(Y_{n+1/2}^{a}\big)
\right) \delta W_{n}  \\
&+ \sum_{j_{1},j_{2}=1}^{m} 
\left(
\mathcal{L}^{j_{1}} \sigma_{j_{2}}
\right)_{\frac{h}{2}}
\left(
\mathscr{P}\big(Y_{n+1/2}^{a}\big)
\right)
\Pi_{j_{1}, j_{2}}^{t_{n}, t_{n+1/2}}.\\
\end{aligned}
\end{equation}
The main result of the paper is formulated as follows,
giving the optimal complexity $\mathcal{O}(\epsilon^{-2})$. 
\begin{theorem}\label{theorem:main-result-of-complexity}
(Complexity of the antithetic MLMC-MM method) Let Assumptions \ref{assumption:coercivity-condition}, \ref{assumption:polynomial-growth-condition}, \ref{assumption:condition-on-the-MM} hold with $p_{0} \in [\max\{\boldsymbol{\alpha}, \boldsymbol{\alpha}+ \gamma -2\}, \infty)$, where $\boldsymbol{\alpha}:= \max\{4\gamma-2, 2\alpha_{1}, 2\alpha_{2}, \alpha_{3}, 2\mathbf{a}, \mathbf{a}+\gamma-1, \mathbf{b}+\gamma-1  \}$. Moreover, let $\{Y^{c}\}$, $\{Y^{f}\}$ and $\{{Y}^{a}\}$ be defined as \eqref{equation:MM-in-coarse-step}, \eqref{equation:MM-in-first-fine-timestep}-\eqref{equation:MM-in-second-fine-timestep} and \eqref{equation:antithetic-estimator-in-first-fine-step}-\eqref{equation:antithetic-estimator-in-second-fine-step}, respectively. For some smooth payoffs $\phi \in C_{b}^{2}(\mathbb{R}^{d}, \mathbb{R})$  
we let 
\begin{equation}
\label{equation:multilevel-estimator-in-main-result}
    \hat P^{f}_{\ell}
    := \tfrac{1}{2} 
    \left( 
    \phi(Y_{2^{\ell-1}}^{f}) + \phi(Y_{2^{\ell-1}}^{a})
    \right), 
    \quad \hat P^{c}_{\ell-1}
    := \phi(Y_{2^{\ell-1}}^{c}),
\end{equation}
with the corresponding estimators $\hat{Z}_{\ell}$ defined as
\begin{equation} 
\hat{Z}_{\ell}
=\left\{
\begin{array}{ll}
\hat{N}_{0}^{-1}\sum_{i=1}^{\hat{N}_{0}}\hat{P}_{0}^{f,(i)}, & \ell=0, \\
\hat{N}_{\ell}^{-1}
\sum_{i=1}^{\hat{N}_{\ell}}
\left(
\hat{P}_{\ell}^{f,(i)}-\hat{P}_{\ell -1}^{c,(i)} 
\right), & \ell>0,
\end{array}\right.
\end{equation}
and the final estimator $\hat{Z}$ given by
\begin{equation}
    \hat{Z}= \sum_{\ell=0}^{L} \hat{Z}_{\ell}.
\end{equation}
Then, there exists a positive constant $C$ such that
\begin{equation} \label{equation:variance-in-main-result}
    \text{Var}\big[\hat Z_{\ell} \big] 
    \leq C \hat N_{\ell}^{-1} h^{2}_{\ell}.
\end{equation}
Given the mean-square-error of $\hat Z$ with bound
\begin{equation}
M S E 
\equiv 
\mathbb{E}
\big[(\hat{Z}-\mathbb{E}[P])^{2}\big]
<
\epsilon^{2}, \notag
\end{equation}
there exists a uniform constant $C$ such that the complexity of approximating \eqref{introduction:expectation-of-functional-of-sde} using the antithetic MLMC-MM method has the bound
\begin{equation}
    \mathcal{C}_{antithetic MLMC-MM} \leq C \epsilon^{-2}.
\end{equation}
\end{theorem}
To arrive at this result, we  prove 
the moment bounds and strong convergence rate of 
the MM scheme \eqref{equation:modified-milstein-method},
based on which we further carry out the variance analysis 
of the multilevel estimators,
as shown in the forthcoming sections. 
%

\section{Strong convergence order of the MM scheme}
\label{section:strong-convergence-order-of-the-MM}
In this section, we aim to reveal the strong convergence order 
of the MM scheme \eqref{equation:modified-milstein-method}. 
One of key elements for this and the subsequent variance analysis is to establish the uniformly bounded moment of the MM scheme \eqref{equation:modified-milstein-method}.
\begin{theorem} \label{theorem:moments-bound-of-MM}
(Moment bounds of the MM scheme) Under Assumptions \ref{assumption:coercivity-condition}, \ref{assumption:condition-on-the-MM}, for any $n\in \{ 1,2,\dots,N \}$, $N \in \mathbb{N}$, $p\in [1,p_{0})$, there exists a uniform constant $C$, depending on $\frac{1}{p_{0}-p}$, such that the MM scheme \eqref{equation:modified-milstein-method} satisfies
\begin{equation}
\sup _{N \in \mathbb{N} }
\sup _{0\leq n \leq N } 
\mathbb{E} 
\left[
\left(1+\|Y_{n}\|^{2}\right)^{p}
\right] 
\leq C 
\mathbb{E} 
\left[ 
\left(1+\|X_{0}\|^{2}\right)^{p} 
\right] < \infty.  
\end{equation}
\end{theorem} 
The proof of Theorem \ref{theorem:moments-bound-of-MM} is postponed to Appendix \ref{section:proof-of-moments-bound}.
%
Even equipped with the moment bounds of the numerical approximations, the analysis of the strong convergence rate 
is non-trivial in the non-globally Lipschitz setting, 
by noting that the modified 
Milstein scheme without Lévy areas can not be 
continuously extended to be an It\^o process  and the 
It\^o formula is hence not available.
%
%
To overcome this problem, we introduce an auxiliary approximation process $\{\widetilde{Y}_{n}\}_{0\leq n \leq N}$, defined by 
\begin{equation} \label{equation:auxiliary-process}
\left\{
\begin{array}{l}
\widetilde{Y}_{n+1} 
= \widetilde{Y}_{n} 
+ \int_{t_{n}}^{t_{n+1}} \mu(X_{s}) \ \mathrm{d} s 
+ \int_{t_{n}}^{t_{n+1}} \sigma(X_{s}) \ \mathrm{d} W_{s} 
+ \sum_{j_{1},j_{2}=1}^{m} 
\left(
\mathcal{L}^{j_{1}} \sigma_{j_{2}}
\right)_{h}
\big(
\mathscr{P}(Y_{n})
\big) 
\Pi_{j_{1}, j_{2}}^{t_{n}, t_{n+1}}, \\
\widetilde{Y}_{0} = x_{0},
\end{array}\right.
\end{equation}
where $\{X_{s}\}_{s\in [0,T]}$ and $\{Y_{n}\}_{0\leq n \leq N}$ are given by \eqref{equation:sde-in-introduction} and \eqref{equation:modified-milstein-method}, respectively.
Then we separate the error $\|X_{t_{n}}-Y_{n} \|_{L^{2p}(\Omega, \mathbb{R}^{d})}$ into two parts as follows,
\begin{equation} \label{equation:strong-convergence-decomposition}
\|X_{t_{n}}-Y_{n} \|_{L^{2p}(\Omega, \mathbb{R}^{d})} 
\leq 
\|X_{t_{n}}-\widetilde{Y}_{n} \|_{L^{2p}(\Omega, \mathbb{R}^{d})} 
+ \|\widetilde{Y}_{n}-Y_{n} \|_{L^{2p}(\Omega, \mathbb{R}^{d})}.
\end{equation}
For the first term, one can 
directly obtain the one-half order of strong convergence.
\begin{lemma} \label{lemma:strong-convergence-rate-of-sde-and-auxiliary-process} (Strong error estimate of SDE and the auxiliary process)
Let Assumptions \ref{assumption:coercivity-condition}, \ref{assumption:polynomial-growth-condition}, \ref{assumption:condition-on-the-MM} hold. And assume  $\{  X_{t} \}_{t \in [0,T]}$ and $\{  \widetilde{Y}_{n} \}_{0\leq n \leq N}$, $N\in \mathbb{N}$, are solutions to the SDE \eqref{equation:sde-in-introduction} and the auxiliary process \eqref{equation:auxiliary-process}, respectively. Then, 
for $p \in [1 , \frac{p_{0}}{\gamma}]$,
\begin{equation}
\sup _{0\leq n \leq N } 
\mathbb{E} 
\big[
\big\| X_{t_{n}} - \widetilde{Y}_{n} \big\|^{2p}
\big] 
\leq Ch^{p} .
\end{equation}
\end{lemma}
The proof of Lemma \ref{lemma:strong-convergence-rate-of-sde-and-auxiliary-process} is postponed to Appendix \ref{proof:lemma-strong-convergence-rate-of-sde-and-auxiliary-process}.
Next we introduce a continuous-time version of the MM method \eqref{equation:modified-milstein-method} on  $ [t_{n}, t_{n+1}]$ as follows: for $s\in [t_{n}, t_{n+1}]$, $n \in \{0,1,\dots,N-1 \}$, 
\begin{equation} \label{equation:continuous-time-version-of-MM-method}
\mathbb{Y}^n ( s )
=
\mathscr{P}(Y_{n})
+ \mu_{h}
\left(
\mathscr{P}(Y_{n}) 
\right) (s-t_{n}) 
+ \sigma_{h}
\left(
\mathscr{P}(Y_{n})
\right) (W_{s}-W_{t_{n}})  
+\sum_{j_{1}, j_{2}=1}^{m} 
\left(
\mathcal{L}^{j_{1}} \sigma_{j_{2}}
\right)_{h}
\left(
\mathscr{P}(Y_{n})
\right) 
\Pi_{j_{1}, j_{2}}^{t_{n}, s},
\end{equation}
where we denote, for $j_{1}, j_{2} \in\{1, \ldots, m\}$
and $s\in [t_{n}, t_{n+1}]$, $n \in \{0,1,\dots,N-1 \}$,
\begin{equation}
\begin{aligned} 
\Pi_{j_{1}, j_{2}}^{t_{n}, s} 
:= \tfrac{1}{2} 
\big[
\left( W_{j_{1}, s} - W_{j_{1}, t_{n}} \right)  
\left( W_{j_{2}, s} - W_{j_{2}, t_{n}} \right)
-\Omega_{j_{1} j_{2}} (s-t_{n})  
\big].
\end{aligned} 
\end{equation}
It is evident that $\mathbb{Y}^n ( t_n ) = \mathscr{P}(Y_{n})$ and
$\mathbb{Y}^n ( t_{n+1} ) = Y_{n+1}$.
Note that $\mathbb{Y}^n ( s )$ is not an
It\^o process even on the interval $[t_n, t_{n+1}]$.
Also,  a continuous-time version of the auxiliary process \eqref{equation:auxiliary-process} is defined by,
for $s\in [t_{n}, t_{n+1}]$, $n \in \{0,1,\dots,N-1 \}$, 
\begin{equation} \label{equation:continuous-time-version-of-auxiliary-process}
\widetilde{Y}(s) 
= \widetilde{Y}_{n} 
+ \int_{t_{n}}^{s} \mu(X_{r}) \ \mathrm{d} r 
+ \int_{t_{n}}^{s} \sigma(X_{r}) \ \mathrm{d} W_{r} 
+ \sum_{j_{1},j_{2}=1}^{m} 
\left(
\mathcal{L}^{j_{1}} \sigma_{j_{2}}
\right)_{h}
\left(
\mathscr{P}\big(Y_{n} \big)
\right) 
\Pi_{j_{1}, j_{2}}^{t_{n}, s}.
\end{equation}
Obviously, $\widetilde{Y}(s)$ is continuous on $[0, T]$ and 
$\widetilde{Y}(t_{n}) = \widetilde{Y}_{n}$ for $n \in \{0,1,\dots,N \}$.
To analyze the error items in \eqref{equation:strong-convergence-decomposition}, we need some properties of the continuous process 
$\{ \mathbb{Y}^n ( s ) \}_{s\in[t_{n},t_{n+1}]}$. 
\begin{lemma} \label{lemma:holder-continuity-of-MM}
Let Assumptions \ref{assumption:coercivity-condition}, \ref{assumption:polynomial-growth-condition}, \ref{assumption:condition-on-the-MM} hold and let $\{ \mathbb{Y}^n ( s ) \}_{s\in[t_{n},t_{n+1}]}$, 
$n\in \{0,1, \dots, N-1\}$, $N \in \mathbb{N}$, be defined by \eqref{equation:continuous-time-version-of-MM-method}. Then, 
for $p\in [1,p_{0})$,
\begin{equation} \label{eq:moment-bounds-of-the-continuous-time-version-of-MM-method}
\begin{aligned} 
\mathbb{E} 
\left[
\big\|\mathbb{Y}^{n}(s)\big\|^{2p}
\right] 
\leq
C 
\left(
     1+ \mathbb{E} \left[\|X_{0}\|^{2p}\right] 
\right).
\end{aligned} 
\end{equation}
For $p\in [1,\frac{p_{0}}{\gamma}]$,
\begin{equation} \label{equation:holder-continuity-of-the-MM-method}
\mathbb{E} 
\left[
\big\| 
\mathbb{Y}^{n}(s) -\mathbb{Y}^{n}(t_{n}) 
\big\|^{2p}
\right] 
\leq Ch^{p}.
\end{equation}
Moreover, for $p \in [1, \frac{p_{0}}{2\gamma-1} ]$ and $j_{1}, j_{2} \in\{1, \ldots, m\}$, the following estimates hold,
\begin{equation}
\begin{aligned} 
\mathbb{E} 
\left[
\big\| 
\mu\big(\mathbb{Y}^{n}(s)\big)
-\mu\big(\mathbb{Y}^{n}(t_{n}) \big)  
\big\|^{2p}
\right] 
\bigvee 
{\color{black}
\mathbb{E} 
\left[
\big\| 
\sigma\big(\mathbb{Y}^{n}(s)\big)
-\sigma\big(\mathbb{Y}^{n}(t_{n}) \big)  
\big\|_{F}^{2p}
\right]  
\leq Ch^{p} .
}
\end{aligned} 
\end{equation}
\end{lemma}
The proof of Lemma \ref{lemma:holder-continuity-of-MM} is 
easy and put in Appendix \ref{proof:lemma-holder-continuity-of-MM}.
As a direct consequence of Lemma \ref{lemma:strong-convergence-rate-of-sde-and-auxiliary-process}, we obtain the following assertions.
%
%
\begin{lemma} \label{lemma:errors-between-sde-and-auxiliary-process-in-continuous-version}
Let Assumptions \ref{assumption:coercivity-condition}, \ref{assumption:polynomial-growth-condition}, \ref{assumption:condition-on-the-MM} hold and let $\{  X_{s} \}_{s \in [t_{n},t_{n+1}]}$ and $\{ \widetilde{Y}(s) \}_{s \in [t_{n},t_{n+1}]}$, $n \in \{1,2, \dots, N \}$, $N\in \mathbb{N}$, be defined as \eqref{equation:sde-in-introduction} and \eqref{equation:continuous-time-version-of-auxiliary-process}, respectively.  For $p \in [1,\frac{p_{0}}{\gamma}]$, 
\begin{equation}
 \mathbb{E} 
 \big[
 \big\| X_{s} - \widetilde{Y}(s) \big\|^{2p}
 \big] 
 \leq 
 Ch^{p}.
\end{equation}
Moreover, the following estimates hold for $p \in [1, \frac{p_{0}}{2\gamma-1} ]$ and $j_{1}, j_{2} \in\{1, \ldots, m\}$,
\begin{small}
\begin{equation} \label{equation:estimates-in-lemma:errors-between-sde-and-auxiliary-process-in-continuous-version}
\begin{aligned} 
\mathbb{E} 
\left[
\big\| 
\mu\left( X_{s}\right) -\mu\big( \widetilde{Y}(s) \big) 
\big\|^{2p}
\right] 
\bigvee 
{\color{black}
\mathbb{E} 
\left[
\big\| 
\sigma\left( X_{s}\right) -\sigma\big( \widetilde{Y}(s) \big)
\big\|_{F}^{2p}
\right]  
}
\bigvee 
\mathbb{E}
\left[
\big\| 
\mathcal{L}^{j_{1}} \sigma_{j_{2} }\left( X_{s}\right)
-\mathcal{L}^{j_{1}} \sigma_{j_{2} }\big(\widetilde{Y}(s)\big)  
\big\|^{2p}
\right] 
\leq 
Ch^{p} .
\end{aligned} 
\end{equation}
\end{small}
\end{lemma}
The proof of Lemma \ref{lemma:errors-between-sde-and-auxiliary-process-in-continuous-version} is deferred to Appendix \ref{proof:lemma-errors-between-sde-and-auxiliary-process-in-continuous-version}.
For any $ x\in \mathbb{R}^{d}$, we introduce a series of shorthand notations as follows: 
\begin{equation} \label{equation:tamed-correction-function}
\begin{aligned}
&\mathbb{T}_{h,\mu}(x) 
:= \mu(x) - \mu_{h}(x), 
\quad 
\mathbb{T}_{h,\sigma}(x) 
:= \sigma(x) - \sigma_{h}(x),\\
&\mathbb{T}_{h,\mathcal{L}_{j_{1}j_{2}}}(x) 
:= \mathcal{L}^{j_{1}} 
\sigma_{j_{2} }\left( x \right) 
- \left(
\mathcal{L}^{j_{1}} \sigma_{j_{2}}
\right)_{h}
\left(x\right), \quad
\mathcal{E}(x)
:= x-\mathscr{P}(x),
\end{aligned}
\end{equation}
which are used frequently in the proof of Lemma \ref{lemma:strong-convergence-rate-of-MM-and-auxiliary-process} and the following section.
Combining these results above, we are fully prepared to show the strong error estimate of the MM scheme \eqref{equation:modified-milstein-method} and the auxiliary process \eqref{equation:auxiliary-process}, i.e. $\|\widetilde{Y}_{n}-Y_{n} \|_{L^{2p}(\Omega, \mathbb{R}^{d})}$ in \eqref{equation:strong-convergence-decomposition}. 
\begin{lemma} \label{lemma:strong-convergence-rate-of-MM-and-auxiliary-process} (Strong error estimate of MM and the auxiliary process)
Let Assumptions \ref{assumption:coercivity-condition}, \ref{assumption:polynomial-growth-condition}, \ref{assumption:condition-on-the-MM} hold. Let  $\{  Y_{n} \}_{0 \leq n \leq N }$ and $\{  \widetilde{Y}_{n} \}_{0\leq n \leq N}$, $n \in \{1,2, \dots, N \}$, $N\in \mathbb{N}$, be given by  \eqref{equation:modified-milstein-method} and  \eqref{equation:auxiliary-process},
respectively. Then, for $p \in [1,p_{1})\cap [1, \frac{p_{0}}{\max\{2\gamma-1, \alpha_{1}, \alpha_{2}, \mathbf{a} \}}] $, there exists a uniform constant $C$, depending on $\frac{2p_{1}-2p}{2p-1}$, such that,  
\begin{equation} 
\sup _{0\leq n \leq N } 
\mathbb{E} 
\big[
\| \widetilde{Y}_{n} - Y_{n}\|^{2p}
\big] 
\leq Ch^{p} .
\end{equation}
\end{lemma}
The proof of Lemma \ref{lemma:strong-convergence-rate-of-MM-and-auxiliary-process} is put in Appendix \ref{proof:lemma-strong-convergence-rate-of-MM-and-auxiliary-process}.
Combining this with Lemma \ref{lemma:strong-convergence-rate-of-sde-and-auxiliary-process}, 
one can derive the one-half order of strong convergence 
based on \eqref{equation:strong-convergence-decomposition} .

\begin{theorem} \label{theorem:strong-convergence-rates}
 ($L^{p}$-strong convergence rate of the MM scheme) Let Assumptions \ref{assumption:coercivity-condition}, \ref{assumption:polynomial-growth-condition}, \ref{assumption:condition-on-the-MM} be fulfilled. Then SDE \eqref{equation:sde-in-introduction} 
 admits  unique adapted solutions in $\mathbb{R}^{d}$, denoted by $\{  X_{t} \}_{t \in [0,T]}$.
For the timestep size $h = \frac{T}{N} $ with  $N \in \mathbb{N}$, assume $\{  Y_{n} \}_{0\leq n \leq N}$ is produced
by the MM scheme \eqref{equation:modified-milstein-method}.
Then, for $p \in [1,p_{1})\cap [1, \frac{p_{0}}{\max\{2\gamma-1, \alpha_{1}, \alpha_{2}, \mathbf{a} \}}] $, there exists a positive constant $C $ depending on $\frac{2p_{1}-2p}{2p-1}$ such that,  
\begin{equation}
\sup _{0\leq n \leq N }
\mathbb{E} 
\left[
\| X_{t_{n}} - Y_{n} \|^{2p}
\right] 
\leq Ch^{p} .
\end{equation}
\end{theorem} 
%

\section{Variance analysis for multilevel estimators} 
\label{section:variance-analysis}
In this section we proceed to give a variance analysis 
for multilevel estimators, i.e., $\text{Var}[\hat{P}^{f}_{\ell} -\hat{P}^{c}_{\ell-1} ]$ in \eqref{equation:variance-in-main-result}. 
Recall that the estimators $\hat{Z}_{\ell}$ with $\hat{N}_{\ell}$ Monte Carlo samples are given by
\begin{equation} 
\hat{Z}_{\ell}
=\left\{
\begin{array}{ll}
\hat{N}_{0}^{-1}\sum_{i=1}^{\hat{N}_{0}}\hat{P}_{0}^{c,(i)}, & \ell=0, \\
\hat{N}_{\ell}^{-1}\sum_{i=1}^{\hat{N}_{\ell}}\left(\hat{P}_{\ell}^{f,(i)}-\hat{P}_{\ell -1}^{c,(i)} \right), & \ell>0.
\end{array}\right.
\end{equation}
Therefore, one derives
\begin{equation} 
\text{Var}[\hat{Z}_{\ell}]
=\left\{
\begin{array}{ll}
\hat{N}_{0}^{-1}\text{Var}\big[\hat{P}_{0}^{c}\big], & \ell=0, \\
\hat{N}_{\ell}^{-1}\text{Var}\big[\hat{P}_{\ell}^{f}-\hat{P}_{\ell -1}^{c}\big] , & \ell>0,
\end{array}\right.
\end{equation}
where we recall
\begin{equation} \label{equation:antithetic-multilevel-estimator-in-variance-analysis}
\hat P^{f}_{\ell}
:= 
\tfrac{1}{2} 
\left(
\phi(Y_{2^{\ell-1}}^{f}) + \phi(Y_{2^{\ell-1}}^{a}) 
\right), 
\quad 
\hat P^{c}_{\ell-1}
:= \phi(Y_{2^{\ell-1}}^{c}), 
\quad 
\phi \in C_{b}^{2}(\mathbb{R}^{d}, \mathbb{R}).
\end{equation}
It follows directly from Theorem \ref{theorem:moments-bound-of-MM} and $h_{0}=T$ that,
$$\text{Var}[\hat{Z}_{0}] \leq C \hat{N}_{0}^{-1} h^{2}_{0}.$$
Subsequently we focus on $\text{Var}[\hat{P}_{\ell}^{f}-\hat{P}_{\ell -1}^{c}]$ for $ \ell \in \{  1, 2, \dots, L\}$. For convenience, the parameter $\ell$ is sometimes omitted in the notation of $h_{\ell}$ since $h_{\ell-1}=2h_{\ell}$. 
The main result of this section is formulated as follows.
\begin{theorem} \label{theorem:convergence-order-of-variance}
Let $\hat{P}^{f}_{\ell}$ and $\hat{P}^{c}_{\ell-1}$, $ \ell \in \{  1, 2, \dots, L\}$, be defined as \eqref{equation:antithetic-multilevel-estimator-in-variance-analysis}.  
Let Assumptions \ref{assumption:coercivity-condition}, \ref{assumption:polynomial-growth-condition}, \ref{assumption:condition-on-the-MM} hold with $p_{0} \in [ \max\{\boldsymbol{\alpha}, \boldsymbol{\alpha}+\gamma -2\}, \infty)$ and $p_{1}\in (1,\infty)$, where we denote $\boldsymbol{\alpha}:= \max\{4\gamma-2, 2\alpha_{1}, 2\alpha_{2}, \alpha_{3}, 2\mathbf{a}, \mathbf{a}+\gamma-1  \}$.
Then
\begin{equation} \label{equation:variance}
\text{Var}
\,
[\hat{P}^{f}_{\ell} - \hat{P}^{c}_{\ell-1}] 
\leq Ch_{\ell}^{2}.
\end{equation}
\end{theorem}
We begin with a lemma quoted from Lemma 2.2 in \cite{giles2014antithetic}, which helps us decompose the variance $\text{Var}[\hat{P}^{f}_{\ell} - \hat{P}^{c}_{\ell-1}]$ into two parts, 
\begin{lemma} \label{lemma:variance-expression}
For $\phi \in C_{b}^{2}(\mathbb{R}^{d}, \mathbb{R})$
and any $u, v, w\in \mathbb{R}^{d}$, it holds
\begin{equation} \label{equation:estimates-of-variance-analysis-in-lemma}
\mathbb{E} 
\left[ 
\left|
\tfrac{1}{2} \left( \phi(u) + \phi(v) \right)
-\phi(w)
\right|^{2}
\right] 
\leq C 
\left(
\mathbb{E}
\left[
\left\|\tfrac{1}{2}\left( u+v \right)-w \right\|^{2}
\right] 
+ 
\mathbb{E}
\left[
\left\|u-v \right\|^{4}
\right] 
\right).
\end{equation}
\end{lemma}
In what follows,
we let $\{ \overline{Y}^{f}\}$ be defined as the average process of $\{ Y^{f}\}$ and $\{ Y^{a}\}$, i.e.
\begin{equation} \label{equation:definition-of-the-average-process-bar-y}
\overline{Y}_{n}^{f}
:= 
\tfrac{1}{2} 
\left( 
Y_{n}^{f} + Y_{n}^{a}
\right),
\end{equation}
where $\{ Y^{f}\}$ and $\{ Y^{a}\}$ are determined by 
\eqref{equation:MM-in-first-fine-timestep}-\eqref{equation:MM-in-second-fine-timestep} and \eqref{equation:antithetic-estimator-in-first-fine-step}-\eqref{equation:antithetic-estimator-in-second-fine-step}, respectively.
Hence the variance $\text{Var}[\hat{P}^{f}_{\ell} - \hat{P}^{c}_{\ell-1}]$ is divided into two parts as below,
\begin{equation} \label{equation:two-key-estimates-of-the-variance-analysis}
 \text{Var}
 [
 \hat{P}^{f}_{\ell} - \hat{P}^{c}_{\ell-1}
 ] 
 \leq 
 C\left(
 \mathbb{E}
 \left[
 \big\|
 \overline{Y}_{2^{\ell-1}}^{f}   - Y_{2^{\ell-1}}^{c}
 \big\|^{2} 
 \right] 
 + 
 \mathbb{E}
 \left[ 
 \big\| 
 Y_{2^{\ell-1}}^{f} - Y_{2^{\ell-1}}^{a} 
 \big\|^{4}  
 \right] 
 \right),
\end{equation}
where Lemma \ref{lemma:variance-expression} is used with
$u=Y_{2^{\ell-1}}^{f}$, $v=Y_{2^{\ell-1}}^{a}$ 
and 
$w=Y_{2^{\ell-1}}^{c}$.
The second term in \eqref{equation:two-key-estimates-of-the-variance-analysis} can be directly obtained in the next lemma.
\begin{lemma} 
\label{lemma:strong-convergence-rate-of-yf-and-ya}
Let Assumptions \ref{assumption:coercivity-condition}, 
\ref{assumption:polynomial-growth-condition},
\ref{assumption:condition-on-the-MM} hold
and
let $\{Y^{f}\}$  and $\{{Y}^{a}\}$ be defined by \eqref{equation:MM-in-first-fine-timestep}-\eqref{equation:MM-in-second-fine-timestep} and \eqref{equation:antithetic-estimator-in-first-fine-step}-\eqref{equation:antithetic-estimator-in-second-fine-step}, respectively. Then, for $p \in [1,p_{1})\cap [1, \frac{p_{0}}{\max\{2\gamma-1, \alpha_{1}, \alpha_{2}, \mathbf{a} \}}] $, 
\begin{equation}
\sup_{0 \leq n \leq N }
\mathbb{E}
\big[ 
\big\| Y_{n}^{f} - {Y}_{n}^{a} \big\|^{2p}  
\big] 
\leq Ch^{p}.
\end{equation}
\end{lemma}
\begin{proof} [Proof of Lemma \ref{lemma:strong-convergence-rate-of-yf-and-ya}]
Based on the construction of the antithetic estimator, the Brownian increments for $\{Y^{f}\}$ and $\{Y^{a}\}$ have the same distribution, conditional on the Brownian increments for $Y^{c}$. Indeed, $Y^{f}-Y^{c}$ has the same distribution as $Y^{a}-Y^{c}$. Moreover,
according to Theorem \ref{theorem:strong-convergence-rates} and the elementary inequality, for $p \in [1,p_{1})\cap [1, \frac{p_{0}}{\max\{2\gamma-1, \alpha_{1}, \alpha_{2}, \mathbf{a} \}}] $, one gets
\begin{equation}
\begin{aligned}
\left\|  
Y^{f}_{n} - Y^{a}_{n}  
\right\|_{L^{2p}(\Omega, \mathbb{R}^{d})}  
&\leq 
\left\|  
Y^{f}_{n} - Y^{c}_{n}  
\right\|_{L^{2p}(\Omega, \mathbb{R}^{d})} 
+ 
\left\|
Y^{a}_{n} - Y^{c}_{n}  
\right\|_{L^{2p}(\Omega, \mathbb{R}^{d})}  \\
&= 
2\left\|  
Y^{f}_{n} - Y^{c}_{n}  
\right\|_{L^{2p}(\Omega, \mathbb{R}^{d})}  \\
& \leq 
2\left(
\left\|  Y^{f}_{n} - X_{t_{n}}  \right\|_{L^{2p}(\Omega, \mathbb{R}^{d})} 
+\left\|  Y^{c}_{n} - X_{t_{n}}  \right\|_{L^{2p}(\Omega, \mathbb{R}^{d})} 
\right)  \\
& \leq 
Ch^{\frac{1}{2}}.
\end{aligned}
\end{equation}
The proof is completed.
\end{proof}
Before proceeding further, we would like to point out that, 
arguments used for the desired estimate $ \E [ \| \overline{Y}^{f}_{n}- {Y}_{n}^{c}\|^2 ] \leq C h^2 $ in the globally Lipschitz setting \cite{giles2014antithetic} 
do not  work in the non-globally Lipschitz setting. 
In what follows, new arguments are developed to achieve it. 
As the first step,  
we employ the previously obtained one-half convergence order 
to easily arrive at $\mathcal{O}(h)$ bound for $\E [ \| \overline{Y}^{f}_{n}- {Y}_{n}^{c}\|^2 ]$  (see Lemma \ref{lemma:general-error-estimate}).
Further, we follow the basic lines in \cite{giles2014antithetic} to give representations of $\{ Y^{f}\}$, $\{ Y^{a}\}$ and $\{ \overline{Y}^{f}\}$ in the coarse timestep (see Lemmas \ref{lemma:expression-of-yf-in-a-coarse-timestep} - \ref{lemma:representation-of-the-average-process-in-a-coarse-timestep}), where the proof is not substantially different from the Lipschitz case and put in Appendix \ref{proof:expression-of-yf-in-a-coarse-timestep}-\ref{proof:representation-of-the-average-process-in-a-coarse-timestep} for completeness. 
This together with the sub-optimal estimate (Lemma \ref{lemma:general-error-estimate}) enables us to 
improve the convergence rate to be order $1$, i.e., 
$\mathbb{E}
 [
 \big\|
 \overline{Y}_{n}^{f}   - Y_{n}^{c}
 \big\|^{2} 
 ]
 \leq
 C h^2,$
and hence deduce the $\mathcal{O}(h^{2})$ variance
as required (see Lemma \ref{lemma:final-variance-estimate}).

%
%
\begin{lemma} \label{lemma:general-error-estimate}
Let $\{{Y}^{c}\}$, $\{Y^{f}\}$, $\{{Y}^{a}\}$ and $\{\overline{Y}^{f}\}$ be defined as \eqref{equation:MM-in-coarse-step}, \eqref{equation:MM-in-first-fine-timestep}-\eqref{equation:MM-in-second-fine-timestep}, \eqref{equation:antithetic-estimator-in-first-fine-step}-\eqref{equation:antithetic-estimator-in-second-fine-step} and \eqref{equation:expression-of-the-average-process-in-lemma}, respectively. 
Let Assumptions \ref{assumption:coercivity-condition}, \ref{assumption:polynomial-growth-condition}, \ref{assumption:condition-on-the-MM} hold. Then
for $p \in [1,p_{1})\cap [1, \frac{p_{0}}{\max\{2\gamma-1, \alpha_{1}, \alpha_{2}, \mathbf{a} \}}] $,
\begin{equation}\label{eq:lemma:suboptimal-estimate}
\sup_{0 \leq n \leq 2^{\ell -1} } 
\mathbb{E}
\left[
\big\| \overline{Y}_{n}^{f}   - {Y}_{n}^{c} \big\|^{2p}
\right] 
\leq Ch^{p}.
\end{equation}
\end{lemma}
\begin{proof}[Proof of Lemma \ref{lemma:general-error-estimate}]
As indicated in the proof of Lemma \ref{lemma:strong-convergence-rate-of-yf-and-ya}, for 
$p \in [1,p_{1})\cap [1, \frac{p_{0}}{\max\{2\gamma-1, \alpha_{1}, \alpha_{2}, \mathbf{a} \}}] $, it follows from Theorem \ref{theorem:strong-convergence-rates} that
\begin{equation}
\begin{aligned}
\Big\|  
\overline{Y}_{n}^{f} - {Y}_{n}^{c}  
\Big\|_{L^{2p}(\Omega, \mathbb{R}^{d})}  
&=   
\Big\| 
\tfrac{1}{2}\left({Y}_{n}^{f} - {Y}_{n}^{c}\right) 
+ \tfrac{1}{2}\left({Y}_{n}^{a} - {Y}_{n}^{c}\right) 
\Big\|_{L^{2p}(\Omega, \mathbb{R}^{d})} \\
&\leq  
\left\| 
{Y}_{n}^{f} - {Y}_{n}^{c} 
\right\|_{L^{2p}(\Omega, \mathbb{R}^{d})} \\
& \leq 
\left\|  
Y^{f}_{n} - X_{t_{n}}  
\right\|_{L^{2p}(\Omega, \mathbb{R}^{d})} 
+\left\|  
{Y}_{n}^{c} - X_{t_{n}}  
\right\|_{L^{2p}(\Omega, \mathbb{R}^{d})} \\
&\leq  
Ch^{\frac{1}{2}}.
\end{aligned}
\end{equation}
The proof is completed.
\end{proof}

The Taylor expansion formula, which will be used frequently throughout this section, is recalled as follows, for any differentiable functions $\psi : \R^{d} \rightarrow \R^{d}$,
\begin{small}
\begin{equation} \label{equation:taylor-expansion-in-mlmc}
\begin{aligned}
&\psi\big(Y_{n+1/2}^{f}\big) 
-\psi\big(
\mathscr{P}\big(Y_{n}^{f} \big)
\big)\\
&= 
\tfrac{\partial \psi}{\partial y}
\left( 
\mathscr{P}\big(Y_{n}^{f} \big)
\right)
\Big(
Y_{n+1/2}^{f} - \mathscr{P}\big(Y_{n}^{f} \big) 
\Big) 
+ \mathcal{R}_{\psi}
\left(
\mathscr{P}\big(Y_{n}^{f} \big) , Y_{n+1/2}^{f}
\right)\\
&= \underbrace{
\tfrac{\partial \psi}{\partial y}
\big(
\mathscr{P}\big(Y_{n}^{f} \big)
\big)
\sigma_{\frac{h}{2}}
\big(
\mathscr{P}\big(Y_{n}^{f} \big)
\big) \delta W_{n}
}_{=:\mathcal{M}_{n+1/2,f}^{(\psi)}} \\
& \ + 
\underbrace{ 
 \tfrac{\partial \psi}{\partial y}
 \big(
 \mathscr{P}\big(Y_{n}^{f} \big)
 \big)  
 \bigg( 
 \mu_{\frac{h}{2}}
 \big(
 \mathscr{P}\big(Y_{n}^{f} \big)
 \big) \tfrac{h}{2} 
 + \sum_{j'_{1},j'_{2}=1}^{m} 
 \left(
 \mathcal{L}^{j'_{1}} \sigma_{j'_{2}}
 \right)_{\frac{h}{2}}
 \big(
 \mathscr{P}\big(Y_{n}^{f} \big)
 \big)
 \Pi_{j'_{1},j'_{2}}^{t_{n}, t_{n+1/2}}
 \bigg)
 + 
 \mathcal{R}_{\psi}
 \left(
 \mathscr{P}\big(Y_{n}^{f} \big), Y_{n+1/2}^{f}
 \right) 
 }_{=:\mathcal{N}_{n+1/2,f}^{(\psi)}},
\end{aligned}
\end{equation}
\end{small}
where
\begin{equation} \label{equation:taylor-expansion-in-section-antithetic-mlmc}
\begin{aligned}
&\mathcal{R}_{\psi}
\left(
\mathscr{P}\big(Y_{n}^{f} \big), Y_{n+1/2}^{f}
\right) \\
&:=
\int_{0}^{1}
\left[
\tfrac{\partial \psi}{\partial x}
\left(
\mathscr{P}\big(Y_{n}^{f} \big)
+r\big
(Y_{n+1/2}^{f}
-\mathscr{P}\big(Y_{n}^{f} \big)
\big)
\right)
-\tfrac{\partial \psi}{\partial x}
\big(
\mathscr{P}\big(Y_{n}^{f} \big)
\big)
\right]
\left(Y_{n+1/2}^{f}-\mathscr{P}\big(Y_{n}^{f} \big)\right) 
\mathrm{d} r  .
\end{aligned}
\end{equation}
Analogously, recalling \eqref{equation:antithetic-estimator-in-first-fine-step} one can show
\begin{small}
\begin{equation} \label{equation:taking-conditional-expectation-of-ya-in-mlmc}
\begin{aligned}
&\psi\big(Y_{n+1/2}^{a}\big) 
-\psi\big(\mathscr{P}\big(Y_{n}^{a}\big)\big)\\
&= \underbrace{
\tfrac{\partial \psi}{\partial y}
\big(
\mathscr{P}\big(Y_{n}^{a}\big)
\big)
\sigma_{\frac{h}{2}}
\big(
\mathscr{P}\big(Y_{n}^{a}\big)
\big) \delta W_{n+1/2}
}_{=:\mathcal{M}_{n+1,a}^{(\psi)}} \\
&\ + 
\underbrace{ 
\tfrac{\partial \psi}{\partial y}
\big(
\mathscr{P}\big(Y_{n}^{a}\big)
\big)  
\bigg( 
\mu_{\frac{h}{2}}
\big(
\mathscr{P}\big(Y_{n}^{a}\big)
\big) \tfrac{h}{2} 
+ \sum_{j'_{1},j'_{2}=1}^{m} 
\left(
\mathcal{L}^{j'_{1}} \sigma_{j'_{2}}
\right)_{\frac{h}{2}}
\big(
\mathscr{P}\big(Y_{n}^{a}\big)
\big)
\Pi_{j'_{1},j'_{2}}^{t_{n+1/2}, t_{n+1}}
\bigg)
+ 
\mathcal{R}_{\psi}
\left(
\mathscr{P}\big(Y_{n}^{a}\big), Y_{n+1/2}^{a}
\right) 
}_{=:\mathcal{N}_{n+1,a}^{(\psi)}}.
\end{aligned}
\end{equation}
\end{small}
In the following, we define several short-hand notations.  
For any $ x,y \in \mathbb{R}^{d}$, we denote
\begin{equation}\label{equation:projected-corrected-function}
\begin{aligned}
&\mathcal{E}_{\mu}(x)
:= \mu(x)- \mu(\mathscr{P}(x)), 
\quad 
\mathcal{E}_{\sigma}(x)
:= \sigma(x)- \sigma(\mathscr{P}(x)),\\ &\mathcal{E}_{\mathcal{L}_{j_{1}j_{2}}}(x)
:= \mathcal{L}_{j_{1}}\sigma_{j_{2}}(x) 
- \mathcal{L}_{j_{1}j}\sigma_{j_{2}}(\mathscr{P}(x)), 
\quad
\Delta \mathscr{P}(x,y)
:= \tfrac{1}{2} 
\left(
\mathscr{P}(x) + \mathscr{P}(y) 
\right) 
- \mathscr{P}\big( \tfrac{1}{2}(x+y) \big).
\end{aligned}
\end{equation}
Now we give a representation of the finer level approximation $\{Y^{f}\}$   over the coarser timestep.
\begin{lemma}  \label{lemma:expression-of-yf-in-a-coarse-timestep}
Let Assumptions \ref{assumption:coercivity-condition}, \ref{assumption:polynomial-growth-condition}, \ref{assumption:condition-on-the-MM} be fulfilled,
the finer path approximation  $\{Y^{f}\}$ defined as \eqref{equation:MM-in-first-fine-timestep}-\eqref{equation:MM-in-second-fine-timestep}, can be expressed as 
\begin{small}
\begin{equation} \label{equation:expression-of-yf-in-lemma}
\begin{aligned}
Y_{n+1}^{f} 
&= 
\mathscr{P}\big(Y_{n}^{f} \big)
+\mu_{h}\big(\mathscr{P}\big(Y_{n}^{f} \big)\big)h 
+ \sigma_{h}
\big(\mathscr{P}\big(Y_{n}^{f} \big)\big) \Delta W_{n} 
+  \sum_{j_{1},j_{2}=1}^{m} 
\left(
\mathcal{L}^{j_{1}} \sigma_{j_{2}}
\right)_{h}
\big(
\mathscr{P}\big(Y_{n}^{f} \big)
\big)
\Pi_{j_{1}, j_{2}}^{t_{n}, t_{n+1}} \\
&\ -\tfrac{1}{2}\sum_{j_{1},j_{2}=1}^{m}
\left(
\mathcal{L}^{j_{1}} \sigma_{j_{2}}
\right)_{h}
\big(
\mathscr{P}\big(Y_{n}^{f} \big)
\big) 
\left(
\delta W_{j_{1}, n} \delta W_{j_{2}, n+1/2} 
- \delta W_{j_{1}, n+1/2} \delta W_{j_{2}, n} 
\right) & \\
& \ +
\underbrace{
M_{n+1,f}^{(1)} +M_{n+1,f}^{(2)} 
+ \tfrac{h}{2}
\mathcal{M}_{n+1/2,f}^{(\mu)}
}_{=:M_{n+1}^{f}} \\
&\ + 
\underbrace{
  \bigg(
  \mathcal{N}_{n+1/2,f}^{(\mu)} + \mathbb{T}_{\frac{h}{2},\mu}
  \left(
  \mathscr{P}\big(Y_{n}^{f} \big)
  \right) 
  - 
  \mathbb{T}_{\frac{h}{2},\mu}
  \Big(\mathscr{P}\big(Y_{n+1/2}^{f} \big) \Big) 
  + 2\mathbb{T}_{h,\mu}
  \big(\mathscr{P}\big(Y_{n}^{f} \big)\big)
  \bigg)\tfrac{h}{2}
  - \mathcal{E}(Y^{f}_{n+1/2})
  }_{=:B_{n+1}^{f}} ,
\end{aligned}
\end{equation}
\end{small}
where  
\begin{small}
\begin{equation} \label{equation: martingale term in yf}
\begin{aligned}
M_{n+1,f}^{(1)} 
&= \sigma_{\frac{h}{2}}
\big(
\mathscr{P}\big(Y_{n}^{f} \big)
\big) \delta W_{n} 
+ \sigma_{\frac{h}{2}}
\big(Y_{n+1/2}^{f}\big) \delta W_{n+1/2} 
- \sigma_{h}
\big(
\mathscr{P}\big(Y_{n}^{f} \big)
\big) \Delta W_{n} \\
& \hspace{15em}
- \sum_{j_{1}, j_{2}=1}^{m}
\left(
\mathcal{L}^{j_{1}} \sigma_{j_{2}}
\right)_{h}
\big(
\mathscr{P}\big(Y_{n}^{f} \big)
\big) \delta W_{j_{2}, n}
\delta W_{j_{1}, n+1/2}, \\
M_{n+1,f}^{(2)} 
&=
\sum_{j_{1},j_{2}=1}^{m} 
\bigg[
\left(
\left(  
\mathcal{L}^{j_{1}} \sigma_{j_{2}}
\right)_{\frac{h}{2}}
\big(
\mathscr{P}\big(Y_{n}^{f} \big)
\big) 
- \left(
\mathcal{L}^{j_{1}} \sigma_{j_{2}}
\right)_{h}
\big(
\mathscr{P}\big(Y_{n}^{f} \big)
\big) 
\right) \Pi_{j_{1}, j_{2}}^{t_{n}, t_{n+1/2}} \\
& \hspace{12em} 
+ \left(
\left(
\mathcal{L}^{j_{1}} \sigma_{j_{2}}
\right)_{\frac{h}{2}}
\big(
\mathscr{P}\big(Y_{n+1/2}^{f} \big)
\big)
- 
\left(
\mathcal{L}^{j_{1}} \sigma_{j_{2}}
\right)_{h}
\big(
\mathscr{P}\big(Y_{n}^{f} \big)
\big) 
\right)
\Pi_{j_{1}, j_{2}}^{t_{n+1/2}, t_{n+1}}
 \bigg].
\end{aligned}
\end{equation}
\end{small}
In addition, $\mathcal{M}_{n+1/2,f}^{(\mu)}$ and $\mathcal{N}_{n+1/2,f}^{(\mu)}$ are defined by replacing $\psi$ in \eqref{equation:taylor-expansion-in-mlmc} by $\mu$
and $\mathbb{T}_{h,\mu}(\mathscr{P}(Y_{n}^{f} ))$, $\mathbb{T}_{\frac{h}{2},\mu}(\mathscr{P}(Y_{n}^{f} ))$, $\mathbb{T}_{\frac{h}{2},\mu}(\mathscr{P}(Y_{n+1/2}^{f} ) )$ and $\mathcal{E}(Y^{f}_{n+1/2})$ are defined by \eqref{equation:tamed-correction-function}.
Moreover,
\begin{equation}
\mathbb{E} [M_{n+1}^{f} | \mathcal{F}_{t_{n}}] 
= 0,
\end{equation}
and, for $p \in [1,\frac{p_{0}}{\max\{3\gamma-2, 2\gamma, \mathbf{a}+\gamma-1, \alpha_{1}, \alpha_{2}, \alpha_{3} \}}]$, 
\begin{equation}
\sup_{0 \leq n \leq 2^{\ell -1}-1 } 
\mathbb{E} 
\left[ 
\big\|  M_{n+1}^{f}   \big\| ^{2p}
\right] 
\leq Ch^{3p},
\quad 
\sup_{0 \leq n \leq 2^{\ell -1}-1 } 
\mathbb{E} 
\left[
\big\|  B_{n+1}^{f}   \big\| ^{2p}
\right] 
\leq Ch^{4p}.
\end{equation}
\end{lemma}
The proof of Lemma  \ref{lemma:expression-of-yf-in-a-coarse-timestep} is deferred to Appendix \ref{proof:expression-of-yf-in-a-coarse-timestep}.
The following lemma is to show the corresponding representation for the antithetic counterpart $\{Y^{a}\}$ in the coarser step.
\begin{lemma} \label{lemma:expression-of-ya-in-a-coarse-timestep}
Let Assumptions \ref{assumption:coercivity-condition}, \ref{assumption:polynomial-growth-condition}, \ref{assumption:condition-on-the-MM} hold,
the antithetic estimator $\{{Y}^{a}\}$, defined as \eqref{equation:antithetic-estimator-in-first-fine-step}-\eqref{equation:antithetic-estimator-in-second-fine-step}, can be expressed  as 
\begin{small}
\begin{equation}
\label{equation:expression-of-ya-in-a-coarse-timestep}
\begin{aligned}
Y_{n+1}^{a} 
&= 
\mathscr{P}\big(Y_{n}^{a}\big) 
+ \mu_{h}
\big(
\mathscr{P}\big(Y_{n}^{a}\big)
\big)h 
+ 
\sigma_{h}
\big(
\mathscr{P}\big(Y_{n}^{a}\big)
\big) \Delta W_{n} 
+ 
\sum_{j_{1},j_{2}=1}^{m} 
\left(
\mathcal{L}^{j_{1}} \sigma_{j_{2}}
\right)_{h}
\big(
\mathscr{P}\big(Y_{n}^{a}\big)
\big)
\Pi_{j_{1}, j_{2}}^{t_{n}, t_{n+1}} \\
&\ +
\tfrac{1}{2}
\sum_{j_{1},j_{2}=1}^{m} 
\left(
\mathcal{L}^{j_{1}} \sigma_{j_{2}}
\right)_{h}
\big(
\mathscr{P}\big(Y_{n}^{a}\big)
\big) 
\left(
\delta W_{j_{1}, n} \delta W_{j_{2}, n+1/2} 
- \delta W_{j_{1}, n+1/2} \delta W_{j_{2}, n} 
\right) & \\
&\ +
\underbrace{
M_{n+1,a}^{(1)} 
+M_{n+1,a}^{(2)} 
+ \tfrac{h}{2}\mathcal{M}_{n+1,a}^{(\mu)}
}_{=:M_{n+1}^{a}} \\
&\ + 
\underbrace{
\left(
\mathcal{N}_{n+1,a}^{(\mu)} 
+ \mathbb{T}_{\frac{h}{2},\mu}
(
\mathscr{P}\big(Y_{n}^{a}\big)
) 
- 
\mathbb{T}_{\frac{h}{2},\mu}
\big(
\mathscr{P}\big(Y_{n+1/2}^{a}\big)
\big) 
+ 2\mathbb{T}_{h,\mu}
(
\mathscr{P}\big(Y_{n}^{a}\big)
)
\right)\tfrac{h}{2} 
-\mathcal{E}(Y^{a}_{n+1/2}) 
}_{=:B_{n+1}^{a}} ,
\end{aligned}
\end{equation}
\end{small}
where
\begin{small}
\begin{equation} \label{equation:martingale-term-in-ya}
\begin{aligned}
M_{n+1,a}^{(1)} 
&= 
\sigma_{\frac{h}{2}}
\big(
\mathscr{P}\big(Y_{n}^{a}\big)
\big) \delta W_{n+1/2} 
+ \sigma_{\frac{h}{2}}
\big(Y_{n+1/2}^{a}\big) \delta W_{n} 
- \sigma_{h}
\big(
\mathscr{P}\big(Y_{n}^{a}\big)
\big) \Delta W_{n} \\
&\hspace{15em}
- \sum_{j_{1}, j_{2}=1}^{m}
\left(
\mathcal{L}^{j_{1}} \sigma_{j_{2}}
\right)_{h}
\big(
\mathscr{P}\big(Y_{n}^{a}\big)
\big) \delta W_{j_{1}, n}\delta W_{j_{2}, n+1/2}, \\
M_{n+1,a}^{(2)} 
&=\sum_{j_{1},j_{2}=1}^{m} 
\bigg[ 
\left(
\left(  
\mathcal{L}^{j_{1}} \sigma_{j_{2}}
\right)_{\frac{h}{2}}
\big(
\mathscr{P}\big(Y_{n}^{a}\big)
\big) 
- 
\left(
\mathcal{L}^{j_{1}} \sigma_{j_{2}}
\right)_{h}
\big(
\mathscr{P}\big(Y_{n}^{a}\big)
\big) 
\right) \Pi_{j_{1}, j_{2}}^{t_{n+1/2}, t_{n+1}} \\
& \hspace{12em} 
+ 
\left( 
\left(
\mathcal{L}^{j_{1}} \sigma_{j_{2}}
\right)_{\frac{h}{2}}
\big(
\mathscr{P}\big(Y_{n+1/2}^{a}\big)
\big)
- 
\left(
\mathcal{L}^{j_{1}} \sigma_{j_{2}}
\right)_{h}
\big(
\mathscr{P}\big(Y_{n}^{a}\big)
\big) 
\right)
\Pi_{j_{1}, j_{2}}^{t_{n}, t_{n+1/2}}
 \bigg].
\end{aligned}
\end{equation}
\end{small}
In addition, $\mathcal{M}_{n+1,a}^{(\mu)}$ and $\mathcal{N}_{n+1,a}^{(\mu)}$ are defined by replacing $\psi$ in \eqref{equation:taking-conditional-expectation-of-ya-in-mlmc} by $\mu$. The notations $\mathbb{T}_{h,\mu}(\mathscr{P}(Y_{n}^{a}))$, $\mathbb{T}_{\frac{h}{2},\mu}(\mathscr{P}\big(Y_{n}^{a}\big))$, $\mathbb{T}_{\frac{h}{2},\mu}(\mathscr{P}(Y_{n+1/2}^{a}))$ and $\mathcal{E}(Y^{a}_{n+1/2})$ adopt a similar definition to \eqref{equation:tamed-correction-function}.
Moreover,
\begin{equation} \label{equation:estimate-of-martingale-term-of-ya}
\mathbb{E} [M_{n+1}^{a} | \mathcal{F}_{t_{n}}] 
= 0,
\end{equation}
and, for $p \in [1,\frac{p_{0}}{\max\{3\gamma-2, 2\gamma, \mathbf{a}+\gamma-1, \alpha_{1}, \alpha_{2}, \alpha_{3} \}}]$,
\begin{equation}
\sup_{0 \leq n \leq 2^{\ell -1}-1 } 
\mathbb{E} 
\left[ 
\big\|  M_{n+1}^{a}   \big\| ^{2p}
\right] \leq Ch^{3p},
\quad 
\sup_{0 \leq n \leq 2^{\ell -1}-1} 
\mathbb{E} 
\left[
\big\|  B_{n+1}^{a}   \big\| ^{2p}
\right] 
\leq Ch^{4p}.
\end{equation}
\end{lemma}
The proof of Lemma \ref{lemma:expression-of-ya-in-a-coarse-timestep} can be found in Appendix \ref{proof:expression-of-ya-in-a-coarse-timestep}.
As a consequence of Lemma \ref{lemma:expression-of-yf-in-a-coarse-timestep} and Lemma \ref{lemma:expression-of-ya-in-a-coarse-timestep}, the representation of $\{\overline{Y}^{f}\}$, defined by \eqref{equation:expression-of-the-average-process-in-lemma}, in a coarser interval is derived in the following lemma.
\begin{lemma} \label{lemma:representation-of-the-average-process-in-a-coarse-timestep}
Let Assumptions \ref{assumption:coercivity-condition}, \ref{assumption:polynomial-growth-condition}, \ref{assumption:condition-on-the-MM} hold and
let $\{ Y^{f}\}$, $\{ {Y}^{a}\}$ and $\{ \overline{Y}^{f}\}$ be defined by \eqref{equation:MM-in-first-fine-timestep}-\eqref{equation:MM-in-second-fine-timestep}, \eqref{equation:antithetic-estimator-in-first-fine-step}-\eqref{equation:antithetic-estimator-in-second-fine-step} and \eqref{equation:definition-of-the-average-process-bar-y}, respectively.
Then the approximation $\{\overline{Y}^{f}\}$ has the following expression over the coarser timestep,
\begin{equation} \label{equation:expression-of-the-average-process-in-lemma}
\begin{aligned}
\overline{Y}^{f}_{n+1}  
&=
\mathscr{P}\big(\overline{Y}_{n}^{f}\big)
+ \mu_{h}
\big(
\mathscr{P}\big(\overline{Y}_{n}^{f}\big)
\big)h
+ \sigma_{h}
\big(
\mathscr{P}\big(\overline{Y}_{n}^{f}\big)
\big) \Delta W_{n}
+ \sum_{j_{1},j_{2}=1}^{m} 
\left(
\mathcal{L}^{j_{1}} \sigma_{j_{2}}
\right)_{h} 
\big(
\mathscr{P}\big(\overline{Y}_{n}^{f}\big)
\big)
\Pi_{j_{1}, j_{2}}^{t_{n}, t_{n+1}}  &\\
& \ + 
\underbrace{
M_{n+1}^{(1)} + M_{n+1}^{(2)} + M_{n+1}^{(3)} + \tfrac{1}{2}\left(M_{n+1}^{f}+ M_{n+1}^{a} \right)
}_{=:M_{n+1}}
+ \underbrace{
B_{n}^{(1)} 
+  
\tfrac{1}{2}
\left( B_{n+1}^{f} + B_{n+1}^{a}\right) 
+ \Delta \mathscr{P}\big(Y^{f}_{n},Y^{a}_{n}\big)
     }_{=:B_{n+1}},
\end{aligned}
\end{equation}
where $\Delta \mathscr{P}(Y^{f}_{n},Y^{a}_{n})$ is denoted as \eqref{equation:projected-corrected-function}. 
Further, $M_{n+1}^{f}$, $B_{n+1}^{f}$ and $M_{n+1}^{a}$, $B_{n+1}^{a}$ are defined in Lemma \ref{lemma:expression-of-yf-in-a-coarse-timestep} and Lemma \ref{lemma:expression-of-ya-in-a-coarse-timestep}, respectively, and 
\begin{equation}
\begin{aligned}
B_{n}^{(1)}
&=  
\left[
\tfrac{1}{2} 
\Big(
\mu_{h}\big(\mathscr{P}\big(Y_{n}^{f} \big)\big) 
+ \mu_{h}\big(\mathscr{P}\big(Y_{n}^{a}\big)\big) 
\Big) 
- 
\mu_{h}
\big(
\mathscr{P}\big(\overline{Y}_{n}^{f}\big)
\big)
\right]h, \\
M_{n+1}^{(1)} 
&= 
\left[
\tfrac{1}{2} 
\Big(
\sigma_{h}\big(\mathscr{P}\big(Y_{n}^{f} \big)\big) 
+ \sigma_{h}\big(\mathscr{P}\big(Y_{n}^{a}\big)\big) 
\Big) 
- \sigma_{h}
\big(
\mathscr{P}\big(\overline{Y}_{n}^{f}\big)
\big)
\right] \Delta W_{n} ,&\\
M_{n+1}^{(2)}
& =
\sum_{j_{1},j_{2}=1}^{m} 
\left[
\tfrac{1}{2} 
\Big(
\left(
\mathcal{L}^{j_{1}} \sigma_{j_{2}}
\right)_{h}
\big(\mathscr{P}\big(Y_{n}^{f} \big)\big)
+\left(
\mathcal{L}^{j_{1}} \sigma_{j_{2}}
\right)_{h}
\big(
\mathscr{P}\big(Y_{n}^{a}\big)
\big)
\Big)
-\left(
\mathcal{L}^{j_{1}} \sigma_{j_{2}}
\right)_{h}
\big(
\mathscr{P}\big(\overline{Y}_{n}^{f}\big)
\big)
\right]
\Pi_{j_{1}, j_{2}}^{t_{n}, t_{n+1}} ,&\\
M_{n+1}^{(3)} 
&=
\tfrac{1}{2}
\sum_{j_{1},j_{2}=1}^{m} 
\Big(
\left(
\mathcal{L}^{j_{1}} \sigma_{j_{2}}
\right)_{h}
\big(
\mathscr{P}\big(Y_{n}^{a}\big)
\big) 
- 
\left(
\mathcal{L}^{j_{1}} \sigma_{j_{2}}
\right)_{h}
\big(
\mathscr{P}\big(Y_{n}^{f} \big)
\big) 
\Big) 
\left(
\delta W_{j_{1}, n} \delta W_{j_{2}, n+1/2} 
- \delta W_{j_{1}, n+1/2} \delta W_{j_{2}, n} 
\right). &\\
\end{aligned}
\end{equation}
Moreover,
\begin{equation}
    \mathbb{E}\left[M_{n+1} | \mathcal{F}_{t_{n}} \right]  = 0.
\end{equation}
Then, for $p \in [1,\frac{p_{0}}{\max\{\boldsymbol{\alpha}, \boldsymbol{\alpha}+\gamma-2\}} ]$,
we have
\begin{equation}
\sup_{0 \leq n \leq 2^{\ell -1}-1 }  \mathbb{E} \left[ \|  M_{n+1}  \| ^{2p} \right]  \leq Ch^{3p},\quad \sup_{0 \leq n \leq 2^{\ell -1}-1 } \mathbb{E}\left[ \|  B_{n+1}  \| ^{2p} \right]\leq Ch^{4p},
\end{equation}
where we denote $\boldsymbol{\alpha}:= \max\{4\gamma-2, 2\alpha_{1}, 2\alpha_{2}, \alpha_{3}, 2\mathbf{a}, 2\mathbf{b}, \mathbf{a}+\gamma-1  \}$.
\end{lemma}
%
The proof of Lemma \ref{lemma:representation-of-the-average-process-in-a-coarse-timestep} is shown in Appendix \ref{proof:representation-of-the-average-process-in-a-coarse-timestep}.
Employing the above sub-optimal estimate \eqref{eq:lemma:suboptimal-estimate} and  
carrying out  more careful error estimates  on the mesh grids, 
we can improve the convergence rate to be order $1$.
\begin{lemma} \label{lemma:final-variance-estimate}
Let Assumptions \ref{assumption:coercivity-condition}, \ref{assumption:polynomial-growth-condition}, \ref{assumption:condition-on-the-MM} hold with $p_{0} \in [ \max\{\boldsymbol{\alpha}, \boldsymbol{\alpha}+\gamma -2\}, \infty)$ and $p_{1}\in (1,\infty)$, where we denote $\boldsymbol{\alpha}:= \max\{4\gamma-2, 2\alpha_{1}, 2\alpha_{2}, \alpha_{3}, 2\mathbf{a}, \mathbf{a}+\gamma-1 ,\mathbf{b}+\gamma-1  \}$.
Let $\{{Y}^{c}\}$, $\{Y^{f}\}$, $\{{Y}^{a}\}$ and $\{\overline{Y}^{f}\}$ be defined as \eqref{equation:MM-in-coarse-step}, \eqref{equation:MM-in-first-fine-timestep}-\eqref{equation:MM-in-second-fine-timestep}, \eqref{equation:antithetic-estimator-in-first-fine-step}-\eqref{equation:antithetic-estimator-in-second-fine-step} and \eqref{equation:expression-of-the-average-process-in-lemma}, respectively.  
Then
\begin{equation}
\sup_{0 \leq n \leq 2^{\ell -1} } 
\mathbb{E}
\big[ 
\big\| \overline{Y}_{n}^{f}  - {Y}_{n}^{c} \big\|^{2} 
\big] 
\leq 
Ch^{2}.
\end{equation}
\end{lemma}
%
\begin{proof}[Proof of Lemma \ref{lemma:final-variance-estimate}] \label{subsection: proof of the variance analysis}
For short we denote, for $n\in \{0, 1, \dots, 2^{\ell-1} \}$ and $ \ell \in \{1, \dots, L\}$,
\begin{small}
\begin{equation}
\begin{aligned}
& e_{n}:=\overline{Y}^{f}_{n} - Y_{n}^{c} ,
\
e^{\mathscr{P}}_{n}
:= \mathscr{P}\big(\overline{Y}_{n}^{f}\big) 
- \mathscr{P}\big({Y}_{n}^{c}\big) , 
\ 
\Delta \mu_{n} 
:= \mu_{h}
\big(
\mathscr{P}(\overline{Y}_{n}^{f})
\big) 
- 
\mu_{h}\big(\mathscr{P}(Y^{c}_{n})\big) ,\\ 
&\Delta  \sigma_{n} 
:= \sigma_{h}
\big(
\mathscr{P}(\overline{Y}_{n}^{f})
\big) 
- \sigma_{h}
\left(
\mathscr{P}(Y^{c}_{n})
\right), \ 
\Delta 
\left(
{\mathcal{L}^{j_{1}}  \sigma_{j_{2}}} 
\right)_{n}
:= 
\left(
\mathcal{L}^{j_{1}} \sigma_{j_{2}}
\right)_{h}
\big(
\mathscr{P}(\overline{Y}_{n}^{f})
\big) 
- \left(
\mathcal{L}^{j_{1}} \sigma_{j_{2}}
\right)_{h}
\left(
\mathscr{P}(Y^{c}_{n})
\right). \notag
\end{aligned}
\end{equation}
\end{small}
In view of \eqref{equation:MM-in-coarse-step} and Lemma \ref{lemma:representation-of-the-average-process-in-a-coarse-timestep}, one obtains
\begin{equation}
\begin{aligned}
e_{n+1}  
=e^{\mathscr{P}}_{n} 
+ \Delta  \mu_{n} h 
+ \Delta  \sigma_{n} \Delta W_{n}  
+ \sum_{j_{1},j_{2} = 1}^{m}
\Delta 
\left(
\mathcal{L}^{j_{1}}  \sigma_{j_{2}}
\right)_{n} 
\Pi_{j_{1}, j_{2}}^{t_{n}, t_{n+1}} 
+  M_{n+1}  + B_{n+1}.
\end{aligned}
\end{equation}
Squaring both sides of the above equality yields
\begin{footnotesize}
\begin{equation} \label{equation:squared-en}
\begin{aligned}
&\|  e_{n+1} \|^{2} \\
& = 
\Big\|  
e^{\mathscr{P}}_{n} 
+ \Delta  \mu_{n} h 
+ \Delta  \sigma_{n} \Delta W_{n}  
+ \sum_{j_{1},j_{2} = 1}^{m}
\Delta 
\left(
\mathcal{L}^{j_{1}}  \sigma_{j_{2}} 
\right)_{n} 
\Pi_{j_{1}, j_{2}}^{t_{n}, t_{n+1}} 
+  M_{n+1} + B_{n+1} 
\Big\|^{2}  \\
& = 
\|
e^{\mathscr{P}}_{n} 
\|^{2} 
+
h^{2} 
\|
\Delta  \mu_{n}  
\|^{2} 
+ 
\|
\Delta  \sigma_{n} \Delta W_{n} 
\|^{2} 
+ 
\Big\|
\sum_{j_{1},j_{2} = 1}^{m} 
\Delta 
\left(
\mathcal{L}^{j_{1}}  \sigma_{j_{2}} 
\right)_{n} 
\Pi_{j_{1}, j_{2}}^{t_{n}, t_{n+1}}
\Big\| ^{2} 
+ \|M_{n+1}  \| ^{2} + \|B_{n+1} \| ^{2}   \\
&  \  
+ 2h\langle
e^{\mathscr{P}}_{n} 
,  
\Delta  \mu_{n}
\rangle 
+ 
2\langle  
e^{\mathscr{P}}_{n} 
,  
\Delta  \sigma_{n} \Delta W_{n}
\rangle 
+ 2\sum_{j_{1},j_{2} = 1}^{m} 
\left\langle  
e^{\mathscr{P}}_{n} 
, 
\Delta 
\left(
\mathcal{L}^{j_{1}}  \sigma_{j_{2}} 
\right)_{n} 
\Pi_{j_{1}, j_{2}}^{t_{n}, t_{n+1}}
\right\rangle 
+ 
2
\langle  
e^{\mathscr{P}}_{n}, M_{n+1}    
\rangle \\
&\ 
+
2\langle  
e^{\mathscr{P}}_{n}
, 
B_{n+1}  
\rangle  
+ 
2h
\langle \Delta  \mu_{n}
, 
\Delta  \sigma_{n} \Delta W_{n}  
\rangle 
+ 
2h\sum_{j_{1},j_{2} = 1}^{m}
\left\langle 
\Delta  \mu_{n}
, 
\Delta 
\left(
\mathcal{L}^{j_{1}}  \sigma_{j_{2}} 
\right)_{n} 
\Pi_{j_{1}, j_{2}}^{t_{n}, t_{n+1}} 
\right\rangle  
+ 
2h
\langle 
\Delta  \mu_{n}
,
M_{n+1}  
\rangle \\
&\ 
+ 2h
\langle
\Delta  \mu_{n}
,
B_{n+1}   
\rangle
+ 
2\sum_{j_{1},j_{2} = 1}^{m} 
\left\langle 
\Delta  \sigma_{n} \Delta W_{n}
, 
\Delta 
\left(
\mathcal{L}^{j_{1}}  \sigma_{j_{2}} 
\right)_{n} 
\Pi_{j_{1}, j_{2}}^{t_{n}, t_{n+1}}  
\right\rangle 
+
2\langle 
\Delta  \sigma_{n} \Delta W_{n}
, 
M_{n+1}   
\rangle 
+ 
2\langle 
\Delta  \sigma_{n} \Delta W_{n}
,
B_{n+1}   
\rangle \\
&  \ 
+2 \sum_{j_{1},j_{2} = 1}^{m}
\left\langle 
\Delta 
\left(
\mathcal{L}^{j_{1}}  \sigma_{j_{2}}
\right)_{n} 
\Pi_{j_{1}, j_{2}}^{t_{n}, t_{n+1}}
, 
M_{n+1}   
\right\rangle 
+2\sum_{j_{1},j_{2} = 1}^{m} 
\left\langle 
\Delta 
\left(
\mathcal{L}^{j_{1}}  \sigma_{j_{2}}
\right)_{n} 
\Pi_{j_{1}, j_{2}}^{t_{n}, t_{n+1}}
, 
B_{n+1}   
\right\rangle 
+ 
2\left\langle 
M_{n+1}
,
B_{n+1} 
\right\rangle.
\end{aligned}
\end{equation}
\end{footnotesize}
Since $\Delta W_{n}$ is independent of $\mathcal{F}_{t_{n}}$,  one can arrive at, for $j_{1}, j_{2} \in \{1,2, \dots, m \}$,
\begin{equation} \label{equation:estimate-1-in-variance-analysis}
\begin{aligned}
&\mathbb{E}
\left[
\langle  
e^{\mathscr{P}}_{n} 
, 
\Delta  \sigma_{n} \Delta W_{n}
\rangle 
\right] 
=  
\mathbb{E}
\left[
\left\langle  
e^{\mathscr{P}}_{n} 
, 
\Delta 
\left(
\mathcal{L}^{j_{1}}  \sigma_{j_{2}} 
\right)_{n} 
\Pi_{j_{1}, j_{2}}^{t_{n}, t_{n+1}}
\right\rangle 
\right] 
= 
\mathbb{E}
\left[ 
\langle 
\Delta  \mu_{n}
, 
\Delta  \sigma_{n} \Delta W_{n}
\rangle 
\right] = 0,\\
& \mathbb{E}
\left[ 
\left\langle 
\Delta  \mu_{n}
, 
\Delta 
\left(
\mathcal{L}^{j_{1}}  \sigma_{j_{2}} 
\right)_{n} 
\Pi_{j_{1}, j_{2}}^{t_{n}, t_{n+1}}
\right\rangle 
\right]
= 
\mathbb{E}
\left[
\left\langle 
\Delta  \sigma_{n} \Delta W_{n}
, 
\Delta 
\left(
\mathcal{L}^{j_{1}}  \sigma_{j_{2}} 
\right)_{n} 
\Pi_{j_{1}, j_{2}}^{t_{n}, t_{n+1}}
\right\rangle 
\right]  = 0.\\
\end{aligned}
\end{equation}
Owing to Lemma \ref{lemma:expression-of-yf-in-a-coarse-timestep}, Lemma \ref{lemma:expression-of-ya-in-a-coarse-timestep} and a conditional expectation argument, we obtain that
\begin{equation}
\begin{aligned}
\mathbb{E} 
\left[
\langle  
e^{\mathscr{P}}_{n}
, 
M_{n+1}   
\rangle 
\right] 
= \mathbb{E} 
\left[ 
\langle  
\Delta\mu_{n}
, 
M_{n+1}   
\rangle 
\right] 
= 0.
\end{aligned}
\end{equation}
Furthermore, by \eqref{equation:property-of-brownain-motion} and \eqref{equation:moments-property-of-bronwain-motion} 
we deduce that
\begin{equation} \label{equation:estimate-2-in-variance-analysis}
\begin{aligned}
\mathbb{E}
\left[ 
\|\Delta  \sigma_{n} \Delta W_{n} 
\|^{2}
\right] 
&= 
h 
{\color{black}
\mathbb{E}
\left[ 
\|\Delta  \sigma_{n} 
\|_{F}^{2}
\right],
}\\
\mathbb{E}
\Big[
\big\|
\sum_{j_{1},j_{2} = 1}^{m} 
\Delta 
\left(
\mathcal{L}^{j_{1}}  \sigma_{j_{2}} 
\right)_{n} 
\Pi_{j_{1}, j_{2}}^{t_{n}, t_{n+1}} 
\big\|^{2}
\Big] 
&\leq Ch^{2} 
\sum_{j_{1},j_{2} = 1}^{m} 
\mathbb{E} 
\Big[ 
\|
\Delta 
\left(
\mathcal{L}^{j_{1}}  \sigma_{j_{2}}
\right)_{n}
\|^{2}
\Big].
\end{aligned}
\end{equation}
Taking the mathematical expectation of \eqref{equation:squared-en} and using estimates \eqref{equation:estimate-1-in-variance-analysis}-\eqref{equation:estimate-2-in-variance-analysis} above yield
\begin{footnotesize}
\begin{equation} \label{equation:expectation-squared-en}
\begin{aligned}
&\mathbb{E}
\left[
\|  
e_{n+1} 
\|^{2}  \right]  \\
& \leq 
\mathbb{E}
\left[
\| e^{\mathscr{P}}_{n} \|^{2}
\right] 
+h^{2} 
\mathbb{E}
\left[
\|\Delta  \mu_{n}  \|^{2}
\right] 
+ 
{\color{black}
h\mathbb{E}
\left[
\|\Delta  \sigma_{n} \|_{F}^{2}
\right] 
}
+ Ch^{2} \sum_{j_{1},j_{2} = 1}^{m}
\mathbb{E}
\left[
\left\|
\Delta 
\left(
\mathcal{L}^{j_{1}}  \sigma_{j_{2}}
\right)_{n}  
\right\| ^{2}
\right] 
+ \mathbb{E}
\left[
\|M_{n+1} \| ^{2}
\right] \\
&\ + 
\mathbb{E}
\left[
\|B_{n+1}  \| ^{2}
\right]  
+ 2h\mathbb{E}
\left[
\left\langle 
e^{\mathscr{P}}_{n} 
,  
\Delta  \mu_{n}
\right\rangle  
\right] 
+ 2\mathbb{E}
\left[
\langle  
e^{\mathscr{P}}_{n}
, 
B_{n+1}  \rangle 
\right] 
+ 2h\mathbb{E}
\left[ 
\langle 
\Delta  \mu_{n}
, 
B_{n+1}   
\rangle 
\right] 
+ 2\mathbb{E}
\left[ 
\langle 
\Delta  \sigma_{n} \Delta W_{n}
,  
M_{n+1}  
\rangle
\right] \\
&\ + 
2\mathbb{E}
\left[
\langle 
\Delta  \sigma_{n} \Delta W_{n}
, 
B_{n+1}   
\rangle
\right] 
+ 2\mathbb{E} 
\left[
\left\langle 
M_{n+1 }
, 
B_{n+1} 
\right\rangle 
\right] 
+2\sum_{j_{1},j_{2} = 1}^{m}
\mathbb{E}
\left[ 
\left\langle 
\Delta \left(
\mathcal{L}^{j_{1}}  \sigma_{j_{2}} 
\right)_{n} 
\Pi_{j_{1}, j_{2}}^{t_{n}, t_{n+1}}
, 
M_{n+1}  
\right\rangle 
\right] \\
& \ + 2\sum_{j_{1},j_{2} = 1}^{m}
\mathbb{E}
\left[ 
\left\langle
\Delta \left(
\mathcal{L}^{j_{1}}  \sigma_{j_{2}} 
\right)_{n} 
\Pi_{j_{1}, j_{2}}^{t_{n}, t_{n+1}}
, 
B_{n+1}   
\right\rangle 
\right] .
\end{aligned}
\end{equation}
\end{footnotesize}
The Young inequality indicates, for some positive constant $\tilde{p} \in (1,p_{1})$,
\begin{equation} \label{equation:cauchy-schwartz-inequality-}
\begin{aligned}
2\mathbb{E}
\left[
\langle 
e^{\mathscr{P}}_{n}
, 
B_{n+1}  
\rangle 
\right] 
&\leq 
h\mathbb{E}
\left[
\| e^{\mathscr{P}}_{n} \|^{2} 
\right] 
+ \tfrac{1}{h} 
\mathbb{E}
\left[
\|B_{n+1} \|^{2}
\right]  ,\\
2h\mathbb{E}
\left[
\langle 
\Delta  \mu_{n}
, 
B_{n+1}   
\rangle 
\right] 
&\leq 
h^{2}
\mathbb{E}
\left[
\|\Delta  \mu_{n}  \|^{2}
\right] 
+  
\mathbb{E}
\left[
\|B_{n+1} \|^{2}
\right] ,\\
2\mathbb{E}
\left[
\langle 
\Delta  \sigma_{n} \Delta W_{n}
, 
M_{n+1}
\rangle
\right] 
&\leq 
(\tilde{p}-1)h
{\color{black}
\mathbb{E}
\left[
\|\Delta  \sigma_{n} \|_{F}^{2}
\right] 
}
+ \tfrac{1}{\tilde{p}-1} 
\mathbb{E}
\left[
\|M_{n+1} \|^{2}
\right] , \\
2\mathbb{E}
\left[
\langle 
\Delta  \sigma_{n} \Delta W_{n}
, 
B_{n+1}   
\rangle
\right] 
&\leq (\tilde{p}-1)h
{\color{black}
\mathbb{E}
\left[
\|\Delta  \sigma_{n} \|_{F}^{2}
\right] 
}
+ \tfrac{1}{\tilde{p}-1} 
\mathbb{E}
\left[
\|B_{n+1} \|^{2}
\right],  \\
2\mathbb{E}
\left[
\left\langle 
\Delta \left(
\mathcal{L}^{j_{1}}  \sigma_{j_{2}} 
\right)_{n} 
\Pi_{j_{1}, j_{2}}^{t_{n}, t_{n+1}}
, 
M_{n+1}   \right\rangle 
\right] 
&\leq 
Ch^{2} 
\mathbb{E}
\left[
\left\|
\Delta \left(
\mathcal{L}^{j_{1}}  \sigma_{j_{2}} 
\right)_{n}  
\right\| ^{2}
\right] 
+  
\mathbb{E}
\left[
\|M_{n+1}  \|^{2}
\right],  \\
2\mathbb{E}
\left[
\left\langle 
\Delta \left(
\mathcal{L}^{j_{1}}  \sigma_{j_{2}} 
\right)_{n} 
\Pi_{j_{1}, j_{2}}^{t_{n}, t_{n+1}}
, 
B_{n+1}   
\right\rangle 
\right] 
&\leq Ch^{2} 
\mathbb{E}
\left[
\left\| 
\Delta \left(
\mathcal{L}^{j_{1}}  \sigma_{j_{2}} 
\right)_{n}  
\right\| ^{2}
\right] 
+ \mathbb{E}
\left[
\|B_{n+1} \|^{2}
\right],  \\
2 \mathbb{E} 
\left[ 
\left\langle 
M_{n+1 }
, 
B_{n+1} 
\right\rangle 
\right] 
&\leq  
\mathbb{E}
\left[
\|M_{n+1}  \|^{2}
\right] 
+ \mathbb{E}
\left[
\|B_{n+1} \|^{2}
\right].
\end{aligned}
\end{equation}
Inserting \eqref{equation:cauchy-schwartz-inequality-} into \eqref{equation:expectation-squared-en} yields
\begin{equation}
\begin{aligned}
\mathbb{E}
\left[
\| e_{n+1} \|^{2}  
\right]  
& \leq 
(1+h)\mathbb{E}
\left[
\| e^{\mathscr{P}}_{n} \|^{2}
\right] 
+ 2h^{2} \mathbb{E}
\left[
\|\Delta \mu_{n}  \|^{2}
\right] 
+ 2h\mathbb{E}
\left[
\langle 
e^{\mathscr{P}}_{n} 
,  
\Delta  \mu_{n}
\rangle 
\right] 
+ (2\tilde{p}-1)h
{\color{black}
\mathbb{E}
\left[
\|\Delta  \sigma_{n} \|_{F}^{2}
\right] 
}\\
&\ + 
Ch^{2} \sum_{j_{1},j_{2} = 1}^{m}
\mathbb{E}
\left[
\left\| 
\Delta 
\left(
\mathcal{L}^{j_{1}} \sigma_{j_{2}} 
\right)_{n}
\right\| ^{2}
\right] 
+C_{\tilde{p}} 
\mathbb{E}
\left[
\|M_{n+1}  \| ^{2}
\right] 
+ \tfrac{C_{\tilde{p}}}{h}
\mathbb{E}
\left[
\|B_{n+1} \|^{2}
\right].  \\
\end{aligned}
\end{equation}
Noting that
\begin{equation} \label{equation:transformations-in-mlmc}
\begin{aligned}
\Delta \mu_{n} 
&= 
\mu\big(
\mathscr{P}(\overline{Y}_{n}^{f})
\big) 
- 
\mu\big(
\mathscr{P}(Y^{c}_{n})
\big) 
+ 
\mathbb{T}_{h,\mu}
(
\mathscr{P}(Y^{c}_{n})
)
-
\mathbb{T}_{h,\mu}
(
\mathscr{P}(\overline{Y}_{n}^{f})
),\\
\Delta \sigma_{n} 
&= 
\sigma
\big(
\mathscr{P}(\overline{Y}_{n}^{f})
\big) 
- 
\sigma
\big(
\mathscr{P}(Y^{c}_{n})
\big) 
+ \mathbb{T}_{h,\sigma}
(
\mathscr{P}(Y^{c}_{n})
)
-\mathbb{T}_{h,\sigma}
(
\mathscr{P}(\overline{Y}_{n}^{f})
),\\
\Delta 
\left(
\mathcal{L}^{j_{1}} \sigma_{j_{2}} 
\right)_{n} 
&= 
\mathcal{L}^{j_{1}} \sigma_{j_{2}}
\big(
\mathscr{P}(\overline{Y}_{n}^{f})
\big) 
-  
\mathcal{L}^{j_{1}} \sigma_{j_{2}} 
\left(
\mathscr{P}(Y^{c}_{n})
\right) 
+ \mathbb{T}_{h,\mathcal{L}_{j_{1}j_{2}}}
(
\mathscr{P}(Y^{c}_{n})
)
-\mathbb{T}_{h,\mathcal{L}_{j_{1}j_{2}}}
(
\mathscr{P}(\overline{Y}_{n}^{f})
),
\end{aligned}
\end{equation}
we employ the Young inequality to deduce that,
for some positive constant $\epsilon_{3} \in (0,\frac{2p_{1}-2\tilde{p}}{2\tilde{p}-1} ]$,
\begin{footnotesize}
\begin{equation}
\begin{aligned}
&\mathbb{E}
\left[
\| e_{n+1} \|^{2}  
\right]  \\
& \leq 
(1+h)
\mathbb{E}
\left[
\| e^{\mathscr{P}}_{n} \|^{2}
\right] 
+ 
h\mathbb{E}
\Bigg[
2\left\langle 
e^{\mathscr{P}}_{n} 
,  
\mu
\big(
\mathscr{P}(\overline{Y}_{n}^{f})
\big) 
- \mu
\big(
\mathscr{P}(Y^{c}_{n})
\big)  
\right\rangle 
+  (2\tilde{p}-1)(1+\epsilon_{3})
{\color{black}
\left\| 
\sigma\big(
\mathscr{P}(\overline{Y}_{n}^{f})
\big) 
- \sigma\left(
\mathscr{P}(Y^{c}_{n})
\right)
\right\|_{F}^{2} 
}
\Bigg] \\
&  \ 
+\underbrace{
2h\mathbb{E} 
\left[
\left\langle 
e^{\mathscr{P}}_{n} 
, 
\mathbb{T}_{h,\mu}
(
\mathscr{P}(Y^{c}_{n})
)
-
\mathbb{T}_{h,\mu}
(
\mathscr{P}(\overline{Y}_{n}^{f})
) 
\right\rangle
\right]
}_{=: J_{1}}
+ \underbrace{
C_{\epsilon_{3}} 
h
{\color{black}
\mathbb{E} 
\left[ 
\left\| 
\mathbb{T}_{h,\sigma}
(
\mathscr{P}(Y^{c}_{n})
)
-
\mathbb{T}_{h,\sigma}
(
\mathscr{P}(\overline{Y}_{n}^{f})
) 
\right\|_{F}^{2} 
\right]
}
}_{=:J_{2}}\\
&\ 
+\underbrace{
Ch^{2} \sum_{j_{1},j_{2} = 1}^{m}  \mathbb{E}\Big[\big\|  \mathbb{T}_{h,\mathcal{L}_{j_{1}j_{2}}}(\mathscr{P}(Y^{c}_{n}))-\mathbb{T}_{h,\mathcal{L}_{j_{1}j_{2}}}(\mathscr{P}(\overline{Y}_{n}^{f}))\big\| ^{2} \Big]
}_{=:J_{3}}
+ \underbrace{
4h^{2}   
\mathbb{E}
\left[
\left\|
\mathbb{T}_{h,\mu}
(
\mathscr{P}(Y^{c}_{n})
)
-
\mathbb{T}_{h,\mu}
(
\mathscr{P}(\overline{Y}_{n}^{f})
)  
\right\|^{2}
\right]
  }_{=:J_{4}}\\
&\ +  
\underbrace{
4h^{2}   
\mathbb{E}
\left[
\big\|
\mu\big(
\mathscr{P}(Y^{c}_{n})
\big)
-
\mu\big(
\mathscr{P}(\overline{Y}_{n}^{f}) 
\big)  
\big\|^{2}
\right]
}_{=:J_{5}}
 + 
\underbrace{
Ch^{2} 
\sum_{j_{1},j_{2} = 1}^{m}  
\mathbb{E} 
\Big[
\big\| 
\mathcal{L}^{j_{1}} \sigma_{j_{2}}
\big(
\mathscr{P}(\overline{Y}_{n}^{f})
\big) 
-  
\mathcal{L}^{j_{1}} \sigma_{j_{2}} 
\left( 
\mathscr{P}(Y^{c}_{n})
\right) 
\big\| ^{2} 
\Big]
}_{=:J_{6}}\\
&\ 
+C_{\tilde{p}} 
\mathbb{E}
\left[
\|M_{n+1}  \| ^{2}
\right] 
+ 
\tfrac{C_{\tilde{p}}}{h}
\mathbb{E}
\left[
\|B_{n+1} \|^{2}
\right].   \\
\end{aligned}
\end{equation}
\end{footnotesize}
Owing to Assumption \ref{assumption:condition-on-the-MM}, the Young inequality
and the elementary inequality, we get
\begin{equation} \label{equation:estimate-of-J1}
\begin{aligned}
J_{1} 
&\leq 
h 
\mathbb{E} 
\big[
\| e^{\mathscr{P}}_{n}  \|^{2}  
\big] 
+ h \mathbb{E} 
\Big[
\left\|
\mathbb{T}_{h,\mu}
(
\mathscr{P}(Y^{c}_{n})
)
-
\mathbb{T}_{h,\mu}
(
\mathscr{P}(\overline{Y}_{n}^{f})
)  
\right\|^{2} 
\Big]  \\
& \leq 
h \mathbb{E} 
\big[ 
\| e^{\mathscr{P}}_{n}  \|^{2}  
\big] 
+ 
Ch^{3}  
\sup_{0 \leq r \leq 2^{\ell -1} } 
\mathbb{E} 
\left[
\left(
1 + \left\|Y_{r}  \right\| 
\right)^{2\alpha_{1}}
\right].
\end{aligned}
\end{equation}
By virtue of Assumption \ref{assumption:condition-on-the-MM}, Theorem \ref{theorem:moments-bound-of-MM}, Lemma \ref{lemma:representation-of-the-average-process-in-a-coarse-timestep} and  elementary inequality, one obtains 
\begin{equation} \label{equation:estimates-of-J2-to-J4}
\begin{aligned}
J_{2} \vee J_{3}\vee J_{4} 
\leq  
Ch^{3} 
\sup_{0 \leq r \leq 2^{\ell -1} } 
\mathbb{E} 
\left[
\left(
1 + \left\|Y_{r}  \right\| 
\right)^{2\max\{\alpha_{1}, \alpha_{2}, \alpha_{3} \} }
\right].
\end{aligned}
\end{equation}
Further, by means of \eqref{equation:growth-of-drift}, \eqref{equation:growth-of-milstein-diffusion-term}, Assumption \ref{assumption:condition-on-the-MM}, Lemma \ref{lemma:general-error-estimate} and the H\"older inequality, we have, for $j_{1}, j_{2} \in \{ 1, \dots, m \}$,
\begin{equation} \label{equation:estimates-of-J5-to-J6}
\begin{aligned}
J_{5} \vee J_{6}
&\leq 
Ch^{2}\mathbb{E} 
\Big[ 
\big(
1 + \| \mathscr{P}(\overline{Y}_{n}^{f})\| + \|\mathscr{P}(Y^{c}_{n})  \| 
\big)^{2\gamma -2} 
\|
\mathscr{P}(\overline{Y}_{n}^{f}) 
- \mathscr{P}(Y^{c}_{n}) 
\|^{2} 
\Big]  \\
& \leq 
Ch^{3} 
\sup_{0 \leq r \leq 2^{\ell -1} } 
\mathbb{E} 
\left[
\left(
1 + \left\|Y_{r}  \right\| 
\right)^{
2\max\{2\gamma-1. \alpha_{1}, \alpha_{2}, \mathbf{a}  \} + 2\gamma-2
} 
\right].
\end{aligned}
\end{equation}
With the estimates \eqref{equation:estimate-of-J1}-\eqref{equation:estimates-of-J5-to-J6}, Lemma \ref{lemma:representation-of-the-average-process-in-a-coarse-timestep} and Assumption \ref{assumption:polynomial-growth-condition} at hand, we obtain that
\begin{equation}
\begin{aligned}
\mathbb{E}
\left[
\|   e_{n+1}\|^{2}  
\right]  
&
\leq (1+Ch)
\mathbb{E}
\left[
\| e^{\mathscr{P}}_{n} \|^{2}
\right] 
+ C_{\tilde{p},\epsilon_{3}}h^{3} 
\sup_{0 \leq r \leq 2^{\ell -1} }
\mathbb{E}
\left[
(1+ \|Y_{r}\| )^{
\max\{2\boldsymbol{\alpha}, 2\boldsymbol{\alpha}+ \gamma-2 \}
} 
\right],
\end{aligned}
\end{equation}
where we again recall $
\boldsymbol{\alpha}:= \max\{4\gamma-2, 2\alpha_{1}, 2\alpha_{2}, \alpha_{3}, 2\mathbf{a}, \mathbf{a}+\gamma-1, \mathbf{b}+\gamma-1 \}.
$
Noting $e_{0} = 0$ and by iteration one can obtain the desired assertion
\begin{equation}
\begin{aligned}
\mathbb{E}
\left[
\|   e_{n+1}\|^{2}  
\right]  
\leq Ch^{2}
\end{aligned}
\end{equation}
with the aid of Assumption \ref{assumption:condition-on-the-MM} and Theorem \ref{theorem:moments-bound-of-MM}.
The proof is thus completed.
\end{proof}

Thanks to Lemma \ref{lemma:strong-convergence-rate-of-yf-and-ya} and Lemma \ref{lemma:final-variance-estimate}, Theorem \ref{theorem:convergence-order-of-variance} is obtained, leading to the fact that \eqref{equation:variance-in-main-result} is verified. Therefore, in view of Theorem \ref{theorem:complexity-of-antithetic-mlmc-with-lipschitz-sde}, if the mean-square error of approximating \eqref{introduction:expectation-of-functional-of-sde} is $\epsilon^{2}$, the overall complexity of the antithetic MLMC-modified Milstein method is $\mathcal{O}(\epsilon^{-2})$ 
and Theorem \ref{theorem:main-result-of-complexity} is validated.

\section{Numerical examples} \label{section:numerical-exapmles}
In this section, some numerical results are performed to verify the theoretical analysis in this work.

\noindent
{\color{black}
\textbf{Example 1: A generalized stochastic FitzHugh-Nagumo model}
}

Let us consider the generalized stochastic FitzHugh-Nagumo (FHN) model in the form of:
\begin{equation} \label{equation:fhn-model}
\begin{array}{l}
\left(\begin{array}{c}
\mathrm{d} X_{1}(t) \\
\mathrm{d} X_{2}(t)
\end{array}\right)
=\left(
\begin{array}{cc}
\tfrac{1}{\epsilon}\big(X_{1}(t)-X_{1}^{3}(t)-X_{2}(t)\big) \\
{\color{black}\vartheta} X_{1}(t)-X_{2}(t)+\beta
\end{array}
\right) \mathrm{d} t \\ 
\hspace{4em}
+\left(
\begin{array}{cc}
c_{1} X_{1}(t)+ c_{2} X_{2}(t) + d_{1} & 0 \\
0 & c_{3} X_{1}(t)+c_{4} X_{2}(t) + d_{2}
\end{array}
\right) 
\left(\begin{array}{c}
\mathrm{d} W_{1,t} \\
\mathrm{d} W_{2,t}
\end{array}\right)
, \quad t \in(0,T],
\end{array}
\end{equation}
where the initial value $X_{0} := (0, 0 )^{T}$, $\epsilon = 0.5$, {\color{black}$\vartheta = 0.5$}, $\beta = 0.5$, $c_{1} = 0.5$, $c_{2} = 0.3$, $c_{3} = 0$, $c_{4} = 0.5$, $d_{1}=0.1$, $d_{2}= 0.1$ and $T=1$.
Evidently, such model satisfies Assumption \ref{assumption:coercivity-condition} and fails to fulfill the commutative condition.
{\color{black}
Moreover,
the polynomial growth conditions and the differentiability condition in Assumption \ref{assumption:polynomial-growth-condition} are easily verified to be satisfied with $\gamma = 3$.
To check the coupled monotonicity  condition in Assumption \ref{assumption:polynomial-growth-condition},
we remark that the drift coefficient $\mu(\cdot): \mathbb{R}^{2} \rightarrow \mathbb{R}^{2}$ of the generalized stochastic FHN model \eqref{equation:fhn-model} can be reformulated as follows,
\begin{equation}
\label{equation:decomposition-fhn}
\mu
\left(\begin{array}{c}
x_{1} \\
x_{2}
\end{array}\right)
:=
\left(
\begin{array}{cc}
\tfrac{1}{\epsilon}\big(x_{1}-x_{1}^{3}-x_{2}\big) \\
\vartheta x_{1}-x_{2}+\beta
\end{array}
\right) 
=
\underbrace{
\left(
\begin{array}{cc}
0 & -\tfrac{1}{\epsilon} \\
\vartheta & -1
\end{array}
\right) 
}_{=:A}
\left(\begin{array}{c}
x_{1} \\
x_{2}
\end{array}\right)
+
\left(\begin{array}{c}
\tfrac{1}{\epsilon} \big(x_{1}-x_{1}^{3} \big) \\
\beta
\end{array}\right),
\end{equation}
which corresponds to a semi-linear SDE.
The first term on the right of \eqref{equation:decomposition-fhn} is linear so that $\langle x - \tilde{x}, A x - A\tilde{x} \rangle \leq C \|x- \tilde{x}\|^{2}$, for $\forall x, \tilde{x} \in \mathbb{R}^{2}$, is fulfilled.
In addition, one has
\begin{equation}
\begin{aligned}
&\Bigg\langle
\left(\begin{array}{c}
x_{1} - \tilde{x}_{1}\\
x_{2} - \tilde{x}_{2}
\end{array}\right),
\left(\begin{array}{c}
\tfrac{1}{\epsilon} \big(x_{1}-x_{1}^{3} \big) \\
\beta
\end{array}\right)
-
\left(\begin{array}{c}
\tfrac{1}{\epsilon} \big(\tilde{x}_{1}-\tilde{x}_{1}^{3} \big) \\
\beta
\end{array}\right)
\Bigg\rangle \\
&=
\tfrac{1}{\epsilon}
|x_{1} - \tilde{x}_{1}|^{2}
-
\tfrac{1}{\epsilon}
|x_{1} - \tilde{x}_{1}|^{2}
(x^{2}_{1}
+
x_{1}\tilde{x}_{1}
+
\tilde{x}^{2}_{1}
)\\
& \leq
\tfrac{1}{\epsilon}
\|x-\tilde{x} \|^{2},
\quad 
\forall x, \tilde{x} \in \mathbb{R}^{2},
\end{aligned}
\end{equation}
due to the fact that $x^{2}_{1}
+
x_{1}\tilde{x}_{1}
+
\tilde{x}^{2}_{1}
=\frac{1}{2}
[(x_{1}+\tilde{x}_{1})^{2} +x^{2}_{1} + \tilde{x}^{2}_{1}] \geq 0
$.
In addition, the diffusion of model \eqref{equation:fhn-model} is globally continuous.
As a result,
coupled monotonicity condition in Assumption \ref{assumption:polynomial-growth-condition} holds with $\forall p_{1} \in (1,\infty)$
for
the generalized stochastic FHN model \eqref{equation:fhn-model}.
}
Next we use the tamed Milstein scheme \textit{(TMS1)} \eqref{equation:tamed-milstein-scheme-model-1} as an example. 

Firstly, we test the strong convergence rate of the \textit{(TMS1)} by averaging over {\color{black}$10^{4}$} numerically generated samples. Besides, the "exact" solutions are identiﬁed  numerically using a fine stepsize $h_{\text{exact}} = 2^{-15}$.  We depict the mean-square approximation errors $\|X_{T}-Y_{T/h}\|_{L^{2}(\Omega,\mathbb{R}^{2})}$ and quadratic approximation errors $\|X_{T}-Y_{T/h}\|_{L^{4}(\Omega,\mathbb{R}^{2})}$ against six different stepsizes $h = 2^{-6}, 2^{-7}, \dots, 2^{-11}$ on a log-log scale.
Also included are two reference lines with slopes of $0.5$ and $1$, which reflect order $0.5$ and order $1$, respectively.
From Figure 1 one can easily observe that both approximation errors decrease at a slope close to $0.5$ when stepsizes shrink, coinciding with the theoretical findings obtained in Theorem \ref{theorem:strong-convergence-rates}.
\begin{figure}[h]
\centering
\subfigure{
    \begin{minipage}[t]{0.45\textwidth}
    \centering
    \includegraphics[width=\textwidth]
      {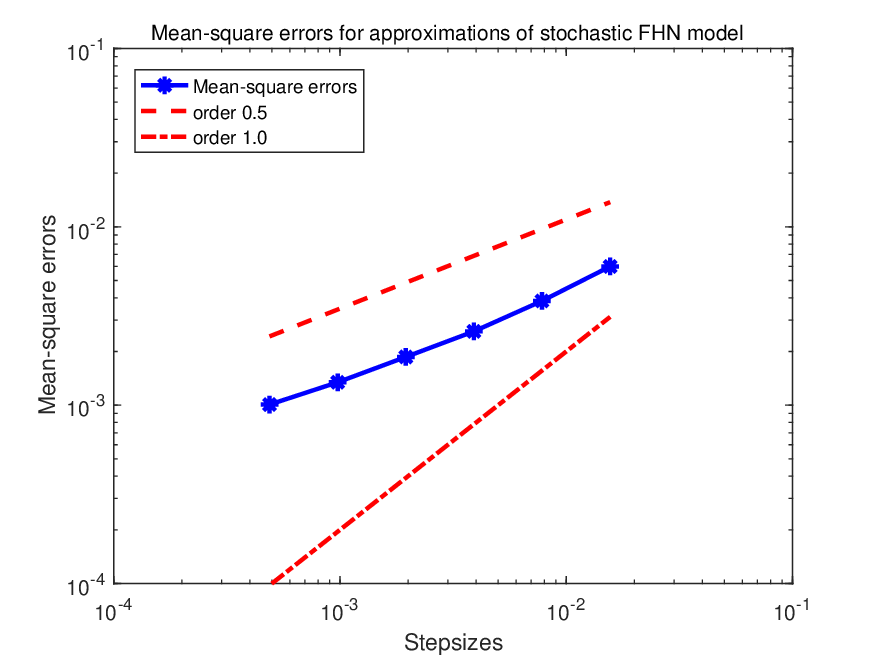}
    \end{minipage}
} 
\subfigure{
    \begin{minipage}[t]{0.45\textwidth}
    \centering
    \includegraphics[width=\textwidth]
      {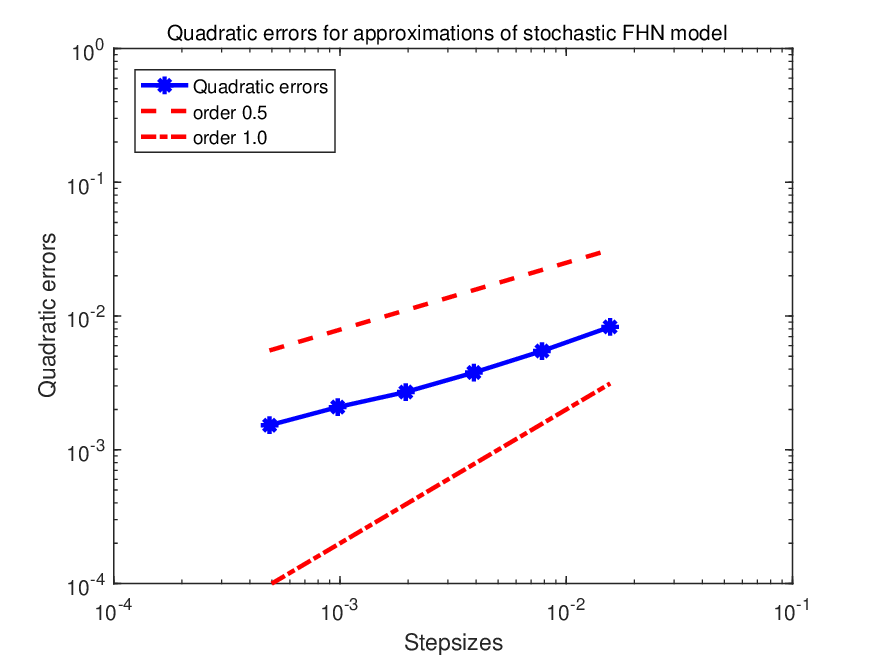}
    \end{minipage}
}

\caption{Strong convergence rates of the Tamed Milstein scheme \textit{(TMS1)}  for the generalized stochastic FHN model \eqref{equation:fhn-model} in mean-squred sense (Left) and in quadratic sense (Right)}
\end{figure}

Let us proceed by following Theorem \ref{theorem:mlmc-complexity-theorem} to investigate the convergence behavior of variance as a function of the level of approximation and the computational cost for a smooth payoff function $P = 2X_{1}(T)+\sin(X_{2}(T))$. On the left side of Figure 2, we employ  $10^{5}$ samples  to show the variance of the multilevel estimators at 10 different levels. Reference lines with slopes of $-1$ and $-2$ are given, which confirms $\beta = 1$ with the standard MLMC and  $\beta = 2$ with the antithetic MLMC, respectively.
On the right side of  Figure 2, the solid line represents the dependence of complexity as function of the desired accuracy $\epsilon$ with the antithetic MLMC technique while the dotted line represents this by the standard MLMC. We see that the complexity of antithetic MLMC is proportional to 
{\color{black}$\epsilon^{-2}$} and is much lower than standard MLMC.

\begin{figure}[h]
\centering
\subfigure{
    \begin{minipage}[t]{0.45\textwidth}
    \centering
    \includegraphics[width=\textwidth]
      {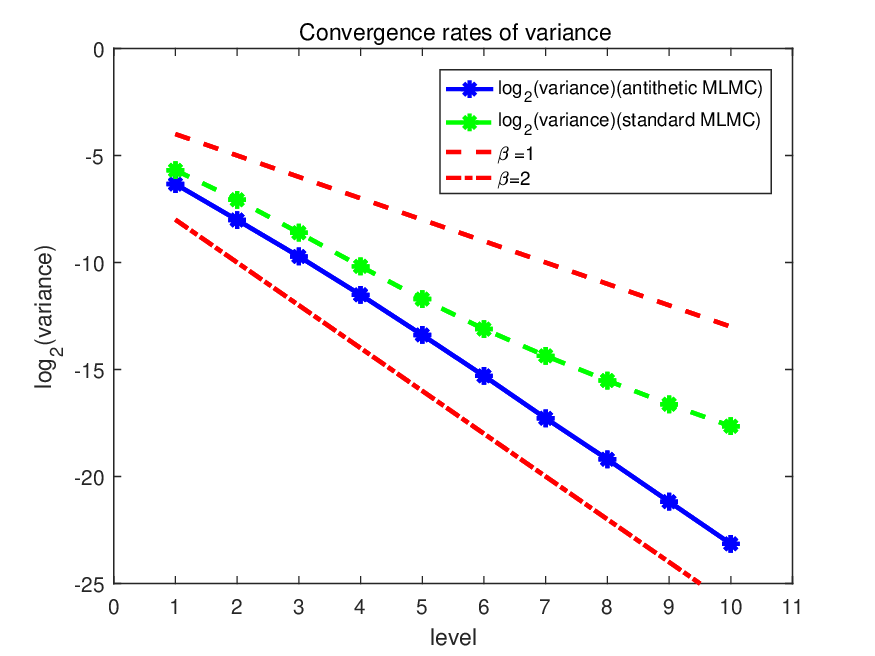}
    \end{minipage}
}
\subfigure{
    \begin{minipage}[t]{0.45\textwidth}
    \centering
    \includegraphics[width=\textwidth]
      {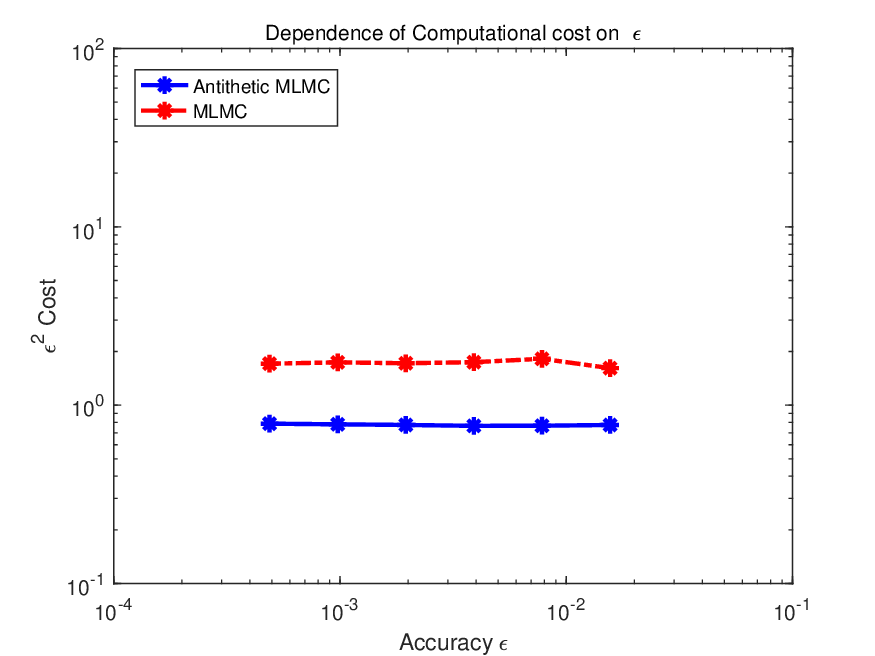}
    \end{minipage}
}

\caption{Convergence rates of variance (Left) and computational cost (Right) for a smooth payoff function $P = 2X_{1}(T)+\sin(X_{2}(T))$}
\end{figure}
{\color{black}
\noindent
\textbf{Example 2: A $3/2$-Heston model}

As the second example, let us consider the $3/2$-Heston model arising from computational finance \cite{kouarfate2021explicit}, which characterizes the dynamics of the stock price  by a two-dimensional process $\{(S_{t}, V_{t}) \}_{ t\in [0,T]}$ satisfying
\begin{equation}
\label{equation:heston-model}
\left\{\begin{array}{l}
\mathrm{d} S_t
=
r S_t \ \mathrm{d}t 
+
\sqrt{V_t } S_t \  \mathrm{d} W_{1,t} \\
\mathrm{d} V_t
=\kappa V_t
\left(
\theta-V_t
\right)
\mathrm{d} t
+\upsilon V_t ^{3 / 2} \mathrm{d} W_{2,t}.
\end{array}\right.
\end{equation}
with the initial value $S_{0} = 1$, $V_{0} = 0.15$ and $r=0.05$, $\kappa=1.5$, $\theta=0.4$. $\upsilon=0.25$, $T=1$.
To avoid the superlinear growth of the diffusion $\sqrt{V_t}S_t$,
we work with the corresponding log-$3/2$-Heston model as follows,
\begin{equation}
\label{equation:log-heston-model}
\left\{\begin{array}{l}
\mathrm{d} \log S_t
=
\left(
r - \tfrac{1}{2} V_{t} \right) \ \mathrm{d}t 
+\sqrt{V_t }  \  \mathrm{d} W_{1,t} \\
\mathrm{d} V_t
=\kappa V_t
\left(
\theta-V_t
\right)
\mathrm{d} t
+\upsilon V_t^{3 / 2} \mathrm{d} W_{2,t}.
\end{array}\right.
\end{equation}
Due to the presence of the superlinearly growing coefficients for $V_t$, the usual Milstein-type scheme is not expected to be a good candidate \cite{hutzenthaler2011strong} and we rely on the modified Milstein scheme \eqref{equation:modified-milstein-method} to tackle the superlinear growth of the coefficients.  Obviously,
such a model \eqref{equation:log-heston-model} violates the commutative condition.
In addition,
Assumption \ref{assumption:coercivity-condition} holds for any $p_{0} \in [1,\frac{49}{2}]$ and the  polynomial growth conditions in Assumption \ref{assumption:polynomial-growth-condition} are met with $\gamma=2$.
Although the coupled monotonicity condition in Assumption \ref{assumption:polynomial-growth-condition} is not satisfied,
the following test still shows that the antithetic MLMC-modified Milstein method performs well.
In the following tests, we use the projected Milstein scheme (\textit{PMS}) \eqref{equation:projected-Milstein-scheme} with $\gamma=2$ as an example.

Figure 3 illustrates the strong convergence rates of the \textit{PMS} in both the mean-square and the quadratic senses against six different stepsizes $h = 2^{-6}, 2^{-7}, \dots, 2^{-11}$ on a log-log scale, where the empirical mean 
is estimated by a Monte Carlo approximation involving $10^4$ independent trajectories and   the "exact" solutions are identiﬁed  numerically using a fine stepsize $h_{\text{exact}} = 2^{-15}$.
It is easy to see that
the strong convergence rate of the approximation errors of the \textit{PMS} decrease at a slope close to $0.5$ in both the mean-square sense and the quadratic sense.
}
\begin{figure}[h]
\centering
\subfigure{
    \begin{minipage}[t]{0.45\textwidth}
    \centering
    \includegraphics[width=\textwidth]
      {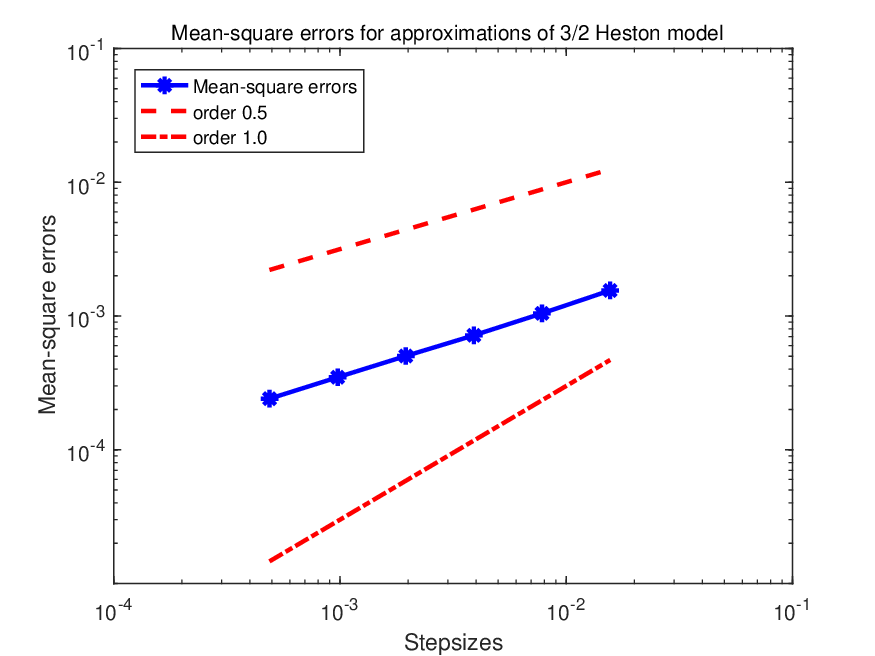}
    \end{minipage}
} 
\subfigure{
    \begin{minipage}[t]{0.45\textwidth}
    \centering
    \includegraphics[width=\textwidth]
      {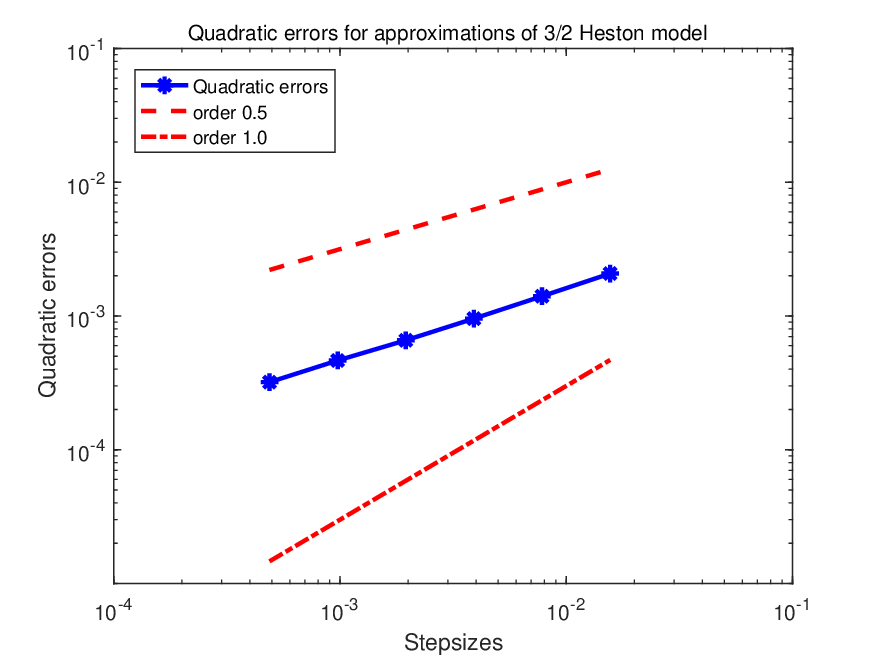}
    \end{minipage}
}

\caption{Strong convergence rates of the projected Milstein scheme (\textit{PMS})  for the 3/2 Heston model \eqref{equation:heston-model} in mean-squred sense (Left) and in quadratic sense (Right)}
\end{figure}

{\color{black}
In Figure 4 we perform numerical tests on the variance for the  3/2 Heston model \eqref{equation:heston-model} with a smooth payoff $P = S_T$. 
In the left picture of Figure 4,
two reference lines with slopes $-1$ and $-2$ are added
so that one can easily see,
the variance of the antithetic MLMC is $\mathcal{O}(h^{2}_{\ell})$,
whereas  the variance of the standard MLMC is $\mathcal{O}(h_{\ell})$  at 10 different levels.
In the  right picture of Figure 4,
one immediately observes the $\mathcal{O}(\epsilon^{-2})$ complexity of the antithetic MLMC method as well as the
computational savings it offers in comparison to the 
 standard MLMC method.
}

\begin{figure}[h]
\centering
\subfigure{
    \begin{minipage}[t]{0.45\textwidth}
    \centering
    \includegraphics[width=\textwidth]
      {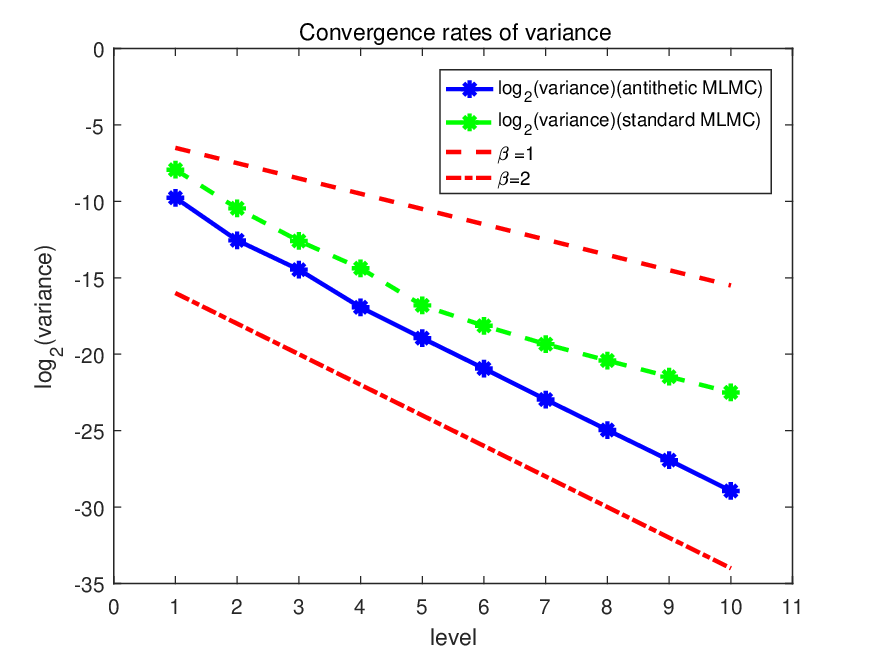}
    \end{minipage}
}
\subfigure{
    \begin{minipage}[t]{0.45\textwidth}
    \centering
    \includegraphics[width=\textwidth]
      {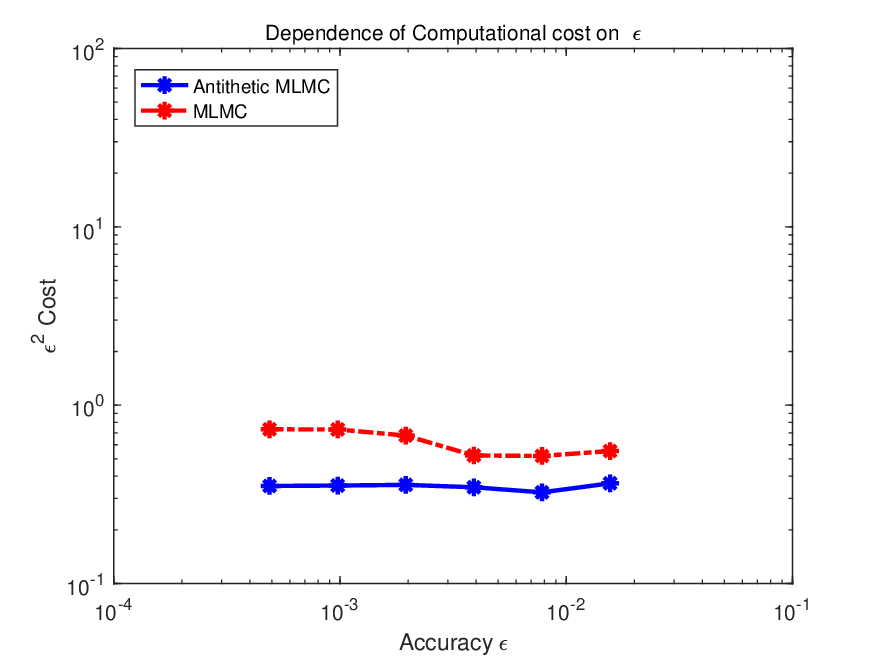}
    \end{minipage}
}

\caption{Convergence rates of variance (Left) and computational cost (Right) for a smooth payoff function $P = S_{T}$}
\end{figure}

\begin{figure}[h]
\centering
\subfigure{
    \begin{minipage}[t]{0.45\textwidth}
    \centering
    \includegraphics[width=\textwidth]
      {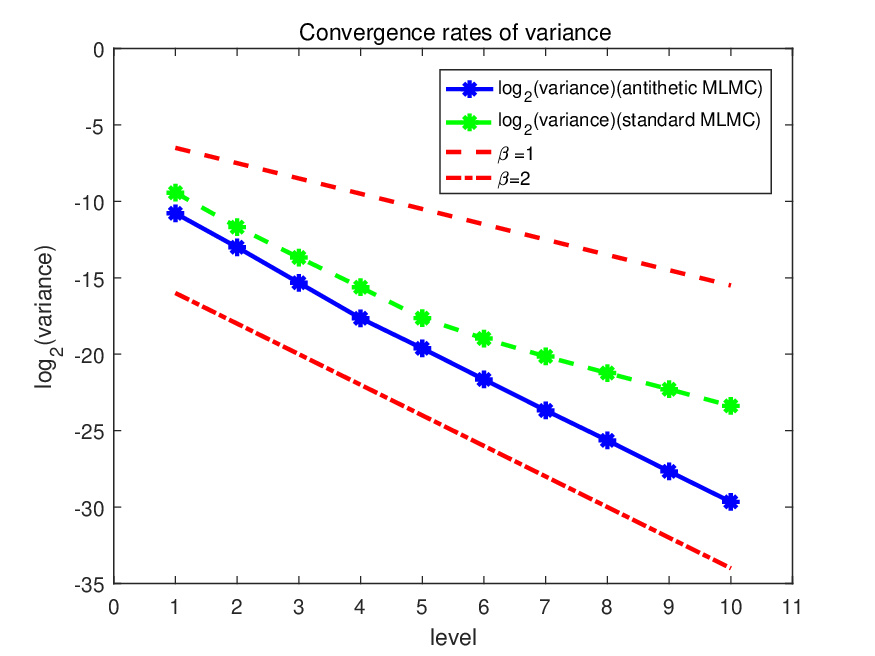}
    \end{minipage}
}
\subfigure{
    \begin{minipage}[t]{0.45\textwidth}
    \centering
    \includegraphics[width=\textwidth]
      {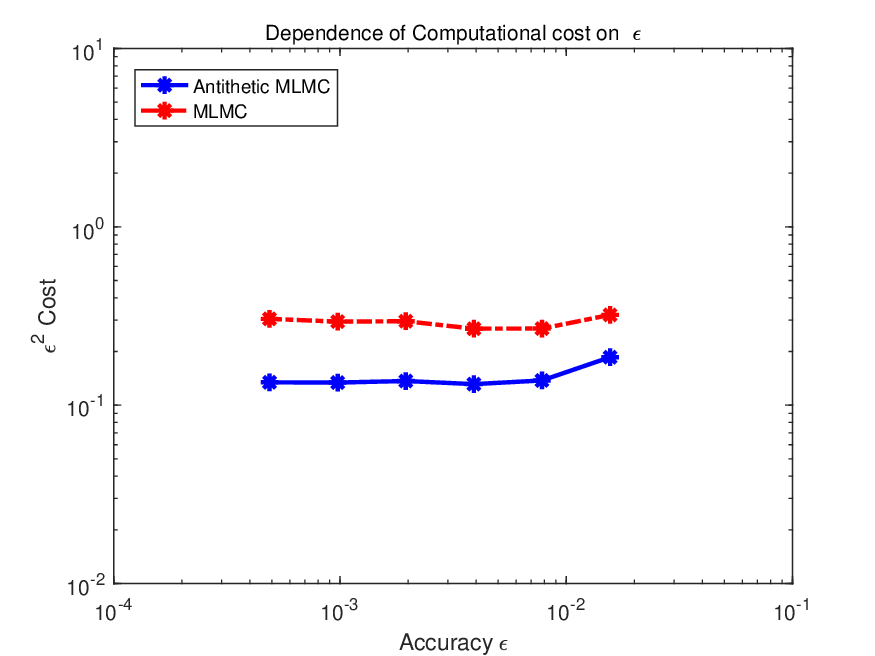}
    \end{minipage}
}

\caption{Convergence rates of variance (Left) and computational cost (Right) for a call option $P = e^{-rT} \max\{ S_{T} -K,0 \}$}
\end{figure}
{\color{black}
The final numerical results are presented in Figure 5 also for the $3/2$-Heston model \eqref{equation:heston-model}, but with the call option payoff $P = e^{-rT} \max\{ S_{T} - K, 0 \}$, where $r = 0.05$ stems from \eqref{equation:heston-model} and $K$ is known as the strike price at time $t=T$.  Here we set $K = 1.2$
and also add two reference lines with slopes $-1$ and $-2$ in the left picture of Figure 5 .
It can be easily observed that
the variance of the multilevel estimators associated with the
antithetic MLMC method is $\mathcal{O}(h^{2}_{\ell})$,
whereas the variance of the multilevel estimators associated with the standard MLMC is $\mathcal{O}(h_{\ell})$.
In addition,
the right picture of Figure 5 illustrates the computational efficiencies of the antithetic
MLMC method compared to the standard MLMC method,
which suggests that the antithetic MLMC method also performs well for some payoffs not belonging to $C^{2}_{b}(\mathbb{R}^{d}, \mathbb{R})$.
}

{\color{black}
\section{Conclusion and future research}
This paper aims to justify the feasibility of the antithetic multilevel Monte Carlo (MLMC) method for SDEs with non-globally Lipschitz continuous coefficients.
To this end, we
develop
a general framework of modified Milstein-type schemes,
which covers the tamed Milstein scheme and the projected Milstein scheme,
without simulating the computationally expensive L\'evy areas.
The study investigates the strong convergence of the scheme and derives the expected one-half order of strong convergence rate under a coupled monotonicity condition and certain polynomial growth condition (Theorem \ref{theorem:strong-convergence-rates}).
Moreover,
by combining the proposed scheme with the antithetic MLMC method,
we obtain the order 2 variance of the multilevel estimator (Theorem \ref{theorem:convergence-order-of-variance})
and thus reveal the optimal complexity $\mathcal{O}(\epsilon^{-2})$ under non-globally Lipschitz condition (Theorem \ref{theorem:main-result-of-complexity}).

Finally,
we would like to discuss briefly several potential extensions of the framework-based numerical scheme and the antithetic sampling method to more general settings. 
Note that the authors of \cite{10.1093/imanum/drad064} presented a tamed Milstein scheme with antithetic MLMC sampling for the MV-SDEs.
So one possible future direction is to study 
whether the present
study can be generalized to encompass McKean–Vlasov SDEs or stochastic interacting particle systems.
Another possible future work would be the proposed numerical scheme and the antithetic sampling method for ergodic SDEs in infinite time interval \cite{fang2019multilevel}.
}


\appendix
\section{Proof of Theorem \ref{theorem:moments-bound-of-MM}} \label{section:proof-of-moments-bound}
\begin{proof}[Proof of Theorem \ref{theorem:moments-bound-of-MM}]
Baesd on \eqref{equation:modified-milstein-method}, 
we can deduce that, for any $n\in \{ 0,1,\dots,N-1 \}$, 
$N \in \mathbb{N}$,
\begin{small}
\begin{equation} 
\begin{aligned}
&\|Y_{n+1} \|^{2}\\
&= \Big\|
\mathscr{P}(Y_{n}) 
+\mu_{h}\left(\mathscr{P}(Y_{n})\right)h 
+\sigma_{h}\left(\mathscr{P}(Y_{n})\right) \Delta W_{n} 
+\sum_{j_{1}, j_{2}=1}^{m} 
\left(
\mathcal{L}^{j_{1}} \sigma_{j_{2}}
\right)_{h}
\left(
\mathscr{P}(Y_{n})
\right) 
\Pi_{j_{1}, j_{2}}^{t_{n}, t_{n+1}} 
\Big\|^{2}  \\
& = 
\| \mathscr{P}(Y_{n}) \|^{2} 
+h^{2}\|\mu_{h}\left(\mathscr{P}(Y_{n})\right) \|^{2} 
+\| 
\sigma_{h}\left(\mathscr{P}(Y_{n})\right) \Delta W_{n} 
\|^{2} 
+\Big\| 
\sum_{j_{1}, j_{2}=1}^{m} 
\left(
\mathcal{L}^{j_{1}} \sigma_{j_{2}}
\right)_{h}
\left(
\mathscr{P}(Y_{n})
\right) 
\Pi_{j_{1}, j_{2}}^{t_{n}, t_{n+1}} 
\Big\|^{2}  \\
& \ + 2h\left\langle 
\mathscr{P}(Y_{n}) , \mu_{h}\left(\mathscr{P}(Y_{n})\right)
\right\rangle 
+ 2 \left\langle 
\mathscr{P}(Y_{n}) ,
\sigma_{h}\left(\mathscr{P}(Y_{n})\right) \Delta W_{n}
\right\rangle 
+ 2 \sum_{j_{1}, j_{2}=1}^{m} 
\left\langle 
\mathscr{P}(Y_{n}) , 
\left(
\mathcal{L}^{j_{1}} \sigma_{j_{2}}
\right)_{h}
\left(
\mathscr{P}(Y_{n})
\right) 
\Pi_{j_{1}, j_{2}}^{t_{n}, t_{n+1}}  
\right\rangle \\
& \ + 2h\left\langle  
\mu_{h}\left(\mathscr{P}(Y_{n})\right) , 
\sigma_{h}\left(\mathscr{P}(Y_{n})\right) \Delta W_{n}
\right\rangle  
+ 2 \sum_{j_{1}, j_{2}=1}^{m}  
\left\langle 
\sigma_{h}
\left(\mathscr{P}(Y_{n})\right) \Delta W_{n} , 
\left(
\mathcal{L}^{j_{1}} \sigma_{j_{2}}
\right)_{h}
\left(\mathscr{P}(Y_{n})\right) 
\Pi_{j_{1}, j_{2}}^{t_{n}, t_{n+1}} 
\right\rangle  \\
& \ + 2h \sum_{j_{1}, j_{2}=1}^{m}  
\left\langle  
\mu_{h}\left(\mathscr{P}(Y_{n})\right)  ,  
\left(
\mathcal{L}^{j_{1}} \sigma_{j_{2}}
\right)_{h}
\left(
\mathscr{P}(Y_{n})
\right) 
\Pi_{j_{1}, j_{2}}^{t_{n}, t_{n+1}} 
\right\rangle .
\end{aligned}
\end{equation}
\end{small}
Using the
Young inequality $2ab \leq \frac{\epsilon_{1}}{2}  a^{2} +  \frac{2}{\epsilon_{1}} b^{2}$, $\epsilon_{1} \in (0, 2p_{0}-2p]$
and Assumption \ref{assumption:condition-on-the-MM}
implies that, for $j_{1}, j_{2} \in \{1, \dots, m \}$,
\begin{footnotesize}
\begin{equation}
\begin{aligned}
2h\left\langle  
\mu_{h}\left(\mathscr{P}(Y_{n})\right) , 
\sigma_{h}\left(\mathscr{P}(Y_{n})\right) \Delta W_{n} 
\right\rangle 
&\leq  
\tfrac{2}{\epsilon_{1}} h^{2} 
\| \mu_{h}\left(\mathscr{P}(Y_{n})\right)\|^{2} 
+ \tfrac{\epsilon_{1}}{2} 
\|
\sigma_{h}\left(\mathscr{P}(Y_{n})\right) \Delta W_{n} 
\|^{2} , \\
& \leq 
C_{\epsilon_{1}}h 
(1+ \|\mathscr{P}(Y_{n})\|)^{2} 
+ \tfrac{\epsilon_{1}}{2} 
\|  
\sigma_{h}\left(\mathscr{P}(Y_{n})\right) \Delta W_{n}
\|^{2}, \\
2 \left\langle 
\sigma_{h}\left(\mathscr{P}(Y_{n})\right) \Delta W_{n} ,
\left(
\mathcal{L}^{j_{1}} \sigma_{j_{2}}
\right)_{h}
\left(\mathscr{P}(Y_{n})\right) 
\Pi_{j_{1}, j_{2}}^{t_{n}, t_{n+1}} 
\right\rangle  
&\leq  
\tfrac{\epsilon_{1}}{2} 
\left\| 
\sigma_{h}\left(\mathscr{P}(Y_{n})\right) \Delta W_{n}
\right\|^{2} 
+ 
\tfrac{2}{\epsilon_{1}} 
\Big\|  
\left(
\mathcal{L}^{j_{1}} \sigma_{j_{2}}
\right)_{h}
\left(\mathscr{P}(Y_{n})\right) 
\Pi_{j_{1}, j_{2}}^{t_{n}, t_{n+1}} 
\Big\|^{2}  \\
&\leq 
\tfrac{\epsilon_{1}}{2} 
\left\|  
\sigma_{h}\left(\mathscr{P}(Y_{n})\right) \Delta W_{n} 
\right\|^{2} 
+ 
\tfrac{C_{\epsilon_{1}}}{h} 
(1+ \|\mathscr{P}(Y_{n})\|)^{2}  
\left|
\Pi_{j_{1}, j_{2}}^{t_{n}, t_{n+1}} 
\right| ^{2} ,\\
2h \left\langle  
\mu_{h}\left(\mathscr{P}(Y_{n})\right)  ,
\left(
\mathcal{L}^{j_{1}} \sigma_{j_{2}}
\right)_{h}
\left(
\mathscr{P}(Y_{n})
\right) 
\Pi_{j_{1}, j_{2}}^{t_{n}, t_{n+1}} 
\right\rangle 
&\leq h^{2} 
\| 
\mu_{h}\left(\mathscr{P}(Y_{n})\right)
\|^{2} 
+ 
\Big\|
 \left(
 \mathcal{L}^{j_{1}} \sigma_{j_{2}}
 \right)_{h}
 \left(
 \mathscr{P}(Y_{n})
 \right) 
 \Pi_{j_{1}, j_{2}}^{t_{n}, t_{n+1}} 
 \Big\|^{2} \\
& \leq Ch (1+ \|\mathscr{P}(Y_{n})\|)^{2} 
+
{\color{black}
\tfrac{C}{h} 
}
(1+ \|\mathscr{P}(Y_{n})\|)^{2} 
\left|
\Pi_{j_{1}, j_{2}}^{t_{n}, t_{n+1}} 
\right| ^{2},
\end{aligned}
\end{equation}
\end{footnotesize}
{\color{black}where we also used the inequality $2ab \leq a^{2} + b^{2}$ for the third term.}
Therefore, one can derive  that
\begin{equation} \label{equation:representation-MM-as-xi}
\begin{aligned}
1 + \|Y_{n+1}  \|^{2} 
\leq 
(
1 + \| \mathscr{P}(Y_{n}) \|^{2}
) 
(
1 + \xi_{n+1}
),
\end{aligned}
\end{equation}
where
\begin{small}
\begin{equation} 
\begin{aligned}
&\xi_{n+1}\\
&:= 
\underbrace{
\tfrac{
C_{\epsilon_{1}}h (1+ \|\mathscr{P}(Y_{n})\|)^{2}
}
{
1+ \|\mathscr{P}(Y_{n}) \|^{2}
} 
}_{\leq C_{\epsilon_{1}}h} 
+ 
\underbrace{
\tfrac{
(1 + \epsilon_{1}) 
\| 
\sigma_{h}\left(\mathscr{P}(Y_{n})\right) \Delta W_{n}  
\|^{2} 
}
{
1+ \|\mathscr{P}(Y_{n}) \|^{2}
}
}_{=:I_{1}} 
+ 
\underbrace{
\tfrac{
\frac{C_{\epsilon_{1}}}{h} 
(1+ \|\mathscr{P}(Y_{n})\|)^{2}
\sum_{j_{1}, j_{2}=1}^{m} 
\left|\Pi_{j_{1}, j_{2}}^{t_{n}, t_{n+1}} \right| ^{2} 
}
{
1+ \|\mathscr{P}(Y_{n}) \|^{2}
}
}_{=:I_{2}} 
+ \underbrace{
\tfrac{
2h \left\langle 
\mathscr{P}(Y_{n}) ,
\mu_{h}\left(\mathscr{P}(Y_{n})\right) 
\right\rangle
}
{
1+ \|\mathscr{P}(Y_{n}) \|^{2}
} 
}_{=:I_{3}}\\
& \ +\underbrace{
\tfrac{ 
2 \sum_{j_{1}, j_{2}=1}^{m} 
\left\langle 
\mathscr{P}(Y_{n}) ,  
\left(
\mathcal{L}^{j_{1}} \sigma_{j_{2}}
\right)_{h}
\left(\mathscr{P}(Y_{n})\right) 
\Pi_{j_{1}, j_{2}}^{t_{n}, t_{n+1}}  
\right\rangle
}
{
1+ \|\mathscr{P}(Y_{n}) \|^{2}
} 
}_{=:I_{4}} 
+\underbrace{
\tfrac{ 
2\langle 
\mathscr{P}(Y_{n}) ,
\sigma_{h}\left(\mathscr{P}(Y_{n})\right) \Delta W_{n} 
\rangle
}
{
1+ \|\mathscr{P}(Y_{n}) \|^{2}
} 
}_{=:I_{5}} .
\end{aligned}
\end{equation}
\end{small}
Taking the $p$-th power, $ p \in [1,p_{0})$, and the conditional mathematical expectation with respect to $\mathcal{F}_{t_{n}}$ on both sides of \eqref{equation:representation-MM-as-xi} yields 
\begin{equation} 
\begin{aligned}
\mathbb{E}
\left[
\left( 1 + \|Y_{n+1}  \|^{2} \right)^{p} 
\big| 
\mathcal{F}_{t_{n}} 
\right] 
\leq 
(1 + \| \mathscr{P}(Y_{n}) \|^{2})^{p} 
\mathbb{E} 
\left[
\left( 1 + \xi_{n+1}\right)^{p} 
\big| \mathcal{F}_{t_{n}}
\right].
\end{aligned}
\end{equation}
According to Lemma 3.3 in \cite{yang2018explicit}, given any $p \in (i, i+1]$,
$i \in \mathbb{N}$, the following estimate
\begin{equation} \label{equation:taylor-expansion-of-xi}
\begin{aligned}
\mathbb{E} 
\left[
\left( 1 + \xi_{n+1}\right)^{p} 
\big| \mathcal{F}_{t_{n}}
\right] 
\leq 
1 
+ 
p
\mathbb{E} 
\left[
\xi_{n+1} \big| \mathcal{F}_{t_{n}} 
\right] 
+ 
\tfrac{p(p-1)}{2}
\mathbb{E} 
\left[ 
\xi_{n+1}^{2} \big| \mathcal{F}_{t_{n}} 
\right]  
+  
\mathbb{E} 
\left[ 
\xi_{n+1}^{3} l_{i}(\xi_{n+1}) \big| \mathcal{F}_{t_{n}}
\right] 
\end{aligned}
\end{equation}
holds true for any $\xi_{n+1} > -1$, where $ l_{i}(\xi)$ is represented as a polynomial of order $i$ with respect to $\xi$ whose coefficients rely only on $p$. We thus have to estimate
only the last three terms in \eqref{equation:taylor-expansion-of-xi}.

\textbf{Step I: estimate of $\mathbb{E} \left[ \xi_{n+1}| \mathcal{F}_{t_{n}} \right]$}

Since $\Delta W_{n}$ and $\Pi_{j_{1}, j_{2}}^{t_{n}, t_{n+1}}$ are independent of $\mathcal{F}_{t_{n}} $, we have
\begin{equation} \label{equation:property-of-brownain-motion}
\begin{aligned}
&\mathbb{E} 
\left[ 
\Delta W_{j, n} \big| \mathcal{F}_{t_{n}}  
\right] 
=  
\mathbb{E} 
\left[
\Pi_{j_{1}, j_{2}}^{t_{n}, t_{n+1}} 
\big| \mathcal{F}_{t_{n}}
\right]  
=0, 
\quad 
\mathbb{E} 
\left[
|\Delta W_{j, n}|^{2} \big| \mathcal{F}_{t_{n}}  
\right] = h, 
\quad  
\forall n \in \mathbb{N},\ j \in \{1, \ldots, m\}
\end{aligned}
\end{equation}
and thus for any $p \in [1, p_{0})$,
\begin{equation} \label{equation:moments-property-of-bronwain-motion}
\begin{aligned}
\mathbb{E} 
\left[
\left| \Delta W_{j, n} \right|^{p}  \big| \mathcal{F}_{t_{n}}
\right] 
\leq 
Ch^{p/2}, 
\quad 
\mathbb{E} 
\left[ 
\big| 
\Pi_{j_{1}, j_{2}}^{t_{n}, t_{n+1}} 
\big|^{p}  
\big| \mathcal{F}_{t_{n}} 
\right] 
\leq Ch^{p}, 
\quad 
\forall n \in \mathbb{N},\ j, j_{1}, j_{2} \in \{1, \ldots, m\}.
\end{aligned}
\end{equation}
Hence, we get
\begin{equation} \label{equation:expectation-of-I4-and-I5}
\mathbb{E} 
\left[ 
I_{4} \big| \mathcal{F}_{t_{n}}  
\right]
=
\mathbb{E} 
\left[ 
I_{5} \big| \mathcal{F}_{t_{n}}  
\right]
=0.
\end{equation}
Furthermore, due to \eqref{equation:property-of-brownain-motion}-\eqref{equation:expectation-of-I4-and-I5}, one can easily observe that
\begin{equation}  
\begin{aligned}
\mathbb{E}
\left[
\xi_{n+1} \big| \mathcal{F}_{t_{n}} 
\right] 
&\leq  
C_{\epsilon_{1}}h 
+ 
\tfrac{
2h\langle 
\mathscr{P}(Y_{n}) , \mu_{h}\left(\mathscr{P}(Y_{n})\right) 
\rangle 
+ (1 + \epsilon_{1})h
{\color{black}
\| \sigma_{h}\left(\mathscr{P}(Y_{n})\right)  \|_{F}^{2} 
}
}
{
1+ \|\mathscr{P}(Y_{n}) \|^{2}
} . \\
\end{aligned}
\end{equation}
%
%
\textbf{Step II: estimate of $\mathbb{E} \left[ \xi^{2}| \mathcal{F}_{t_{n}} \right]$}

Before proceeding further, we introduce a series of useful estimates.
%
By virtue of \eqref{equation:moments-property-of-bronwain-motion}, \eqref{equation:upper-bound-of-diffusion} and the fact that $\Delta W_{n}$ is independent from $\mathcal{F}_{t_{n}}$, one can infer that, for $k\in \mathbb{N}$,
\begin{equation} \label{equation:p-th-moment-of-I1}
\begin{aligned}
\mathbb{E}
\left[
|I_{1}|^{k}
\big| \mathcal{F}_{t_{n}} 
\right] 
&= 
\tfrac
{
(1 + \epsilon_{1})^{k}
\mathbb{E}
\left[ 
\| 
\sigma_{h}\left(\mathscr{P}(Y_{n})\right) \Delta W_{n}
\|^{2k}
\big| \mathcal{F}_{t_{n}} 
\right] 
}
{
(1+ \|\mathscr{P}(Y_{n}) \|^{2})^{k}
} 
\leq 
\tfrac{
Ch^{\frac{k}{2}}
\left(
1 + \| \mathscr{P}(Y_{n}) \|^{2} 
\right)^{k}
}
{
(1+ \|\mathscr{P}(Y_{n}) \|^{2})^{k}
}  
\leq Ch^{k/2}.
\end{aligned}
\end{equation}
Owing to \eqref{equation:moments-property-of-bronwain-motion}, we then acquire
\begin{equation} 
\begin{aligned}
\mathbb{E}
\left[
|I_{2}|^{k}
\big| \mathcal{F}_{t_{n}}
\right]  
\leq 
C_{\epsilon_{1}}h^{k}.
\end{aligned}
\end{equation}
By Assumption \ref{assumption:condition-on-the-MM} and the Cauchy-Schwarz inequality, we also obtain that,
\begin{equation} \label{equation:p-th-moment-of-I3}
\begin{aligned} 
\mathbb{E}
\left[
|I_{3}|^{k} 
\big| \mathcal{F}_{t_{n}} 
\right] 
= 
\tfrac{
2^{k}h^{k} 
\left| 
\langle 
\mathscr{P}(Y_{n}) , \mu_{h}\left(\mathscr{P}(Y_{n})\right) 
\rangle  
\right|^{k}  
}
{
(1+ \|\mathscr{P}(Y_{n}) \|^{2})^{k}
} 
\leq 
\tfrac{
Ch^{k}
\|\mathscr{P}(Y_{n}) \|^{k} 
\| \mu_{h}\left(\mathscr{P}(Y_{n})\right) \|^{k}
}
{
(1+ \|\mathscr{P}(Y_{n}) \|^{2})^{k}
} 
\leq Ch^{k/2} .
\end{aligned}
\end{equation}
Analogous to \eqref{equation:p-th-moment-of-I3},  using the Cauchy-Schwarz inequality and \eqref{equation:moments-property-of-bronwain-motion}  implies that
{\color{black}
\begin{equation} \label{equation:p-th-moment-of-I4}
\begin{aligned}
\mathbb{E}
\left[
|I_{4}|^{k} 
\big| \mathcal{F}_{t_{n}} 
\right] 
&\leq
\tfrac{
C
\sum_{j_{1}, j_{2}=1}^{m} 
\left\|
\mathscr{P}(Y_{n}) 
\right\|^{k}
\mathbb{E}
\left[
\left\|
\left(
\mathcal{L}^{j_{1}} \sigma_{j_{2}}
\right)_{h}
\left(\mathscr{P}(Y_{n})\right) 
\Pi_{j_{1}, j_{2}}^{t_{n}, t_{n+1}}  
\right\|^{k}
\big|
\mathcal{F}_{t_{n}}
\right]
}{
(1+ \|\mathscr{P}(Y_{n}) \|^{2})^{k}
}
\\
&\leq 
\tfrac{
Ch^{k} 
\sum_{j_{1}, j_{2}=1}^{m} 
\|  \mathscr{P}(Y_{n}) \|^{k} 
\big\|  
\left(\mathcal{L}^{j_{1}} \sigma_{j_{2}} \right)_{h}
\left(\mathscr{P}(Y_{n})\right)  
\big\|^{k}  
}
{
(1+ \|\mathscr{P}(Y_{n}) \|^{2})^{k}
} \\
&\leq Ch^{k/2}.
\end{aligned}
\end{equation}
}
Also, due to \eqref{equation:upper-bound-of-diffusion},\eqref{equation:moments-property-of-bronwain-motion} and the Cauchy-Schwarz inequality,
for $k \geq 4$, one observes
\begin{equation} \label{equation:p-th-moment-of-I5}
\begin{aligned}
\mathbb{E}
\left[   
|I_{5}|^{k} 
\big| \mathcal{F}_{t_{n}} 
\right] 
&\leq  
\tfrac{
2^{k} 
\| \mathscr{P}(Y_{n}) \|^{k} 
\mathbb{E} 
\left[ 
\| 
\sigma_{h}\left(\mathscr{P}(Y_{n})\right)\Delta W_{n}
\|^{k} 
\big|\mathcal{F}_{t_{n}} 
\right] 
}
{
\left(1+ \| \mathscr{P}(Y_{n}) \|^{2}\right)^{k}
}   
\leq 
\tfrac{
Ch^{k/4} 
\| \mathscr{P}(Y_{n}) \|^{k} 
\left(   
1 + \| \mathscr{P}(Y_{n}) \|^{2}    
\right)^{k/2} 
}
{
\left(1+ \| \mathscr{P}(Y_{n}) \|^{2}\right)^{k}
}  
\leq Ch^{k/4} .
\end{aligned}
\end{equation}
When it comes to the estimate of $\mathbb{E} \left[ \xi_{n+1}^{2}\big| \mathcal{F}_{t_{n}} \right]$, 
we use elementary inequalities to attain 
\begin{equation}
\label{equation:expansion-of-2nd-moment-of-xi}
\begin{aligned}
\mathbb{E}
\big[ 
\xi_{n+1}^{2} 
\big| \mathcal{F}_{t_{n}} 
\big] 
&= 
\mathbb{E}
\Big[
\Big( C_{\epsilon_{1}}h 
+ 
\sum_{i=1}^{4} I_{i} + I_{5} \Big )^{2} 
\Big| \mathcal{F}_{t_{n}} 
\Big] \\
&\leq 
C_{\epsilon_{1}}h
+ 4\sum_{i=1}^{4}
\mathbb{E}
\left[
I_{i}^{2}\big| \mathcal{F}_{t_{n}} 
\right] 
+ 
\mathbb{E}\left[
I_{5}^{2}\big| \mathcal{F}_{t_{n}} 
\right] 
+ 2 \sum_{i=1}^{4}
\mathbb{E} 
\left[
I_{i} I_{5}\big| \mathcal{F}_{t_{n}} 
\right], \\
\end{aligned}
\end{equation}
where, by \eqref{equation:property-of-brownain-motion} we know
\begin{equation} \label{equation:2-moment-bound-of-I5}
\begin{aligned}
\mathbb{E}
\left[
|I_{5}|^{2} 
\big| \mathcal{F}_{t_{n}} 
\right] 
=& 
\tfrac{
4h 
\|  
\mathscr{P}(Y_{n})^{T} 
\sigma_{h}\left(\mathscr{P}(Y_{n})\right)   
\|^{2}  
}
{
(1+ \|\mathscr{P}(Y_{n}) \|^{2})^{2}
}.
\end{aligned}
\end{equation}
Moreover, recalling some power properties of Brownian motions, for any $\overline{p} \in [1, p^{*}) \cap \mathbb{N}$ and $j, j_{1}, j_{2} \in \{ 1 , \dots, m\}$, we obtain that 
\begin{equation} \label{equation:odd-power-property-of-brownian-motion}
\mathbb{E} 
\left[ 
\Delta W_{j, n}  
\ \Pi_{j_{1}, j_{2}}^{t_{n}, t_{n+1}} 
\big| \mathcal{F}_{t_{n}} 
\right] 
= 0,\quad
\mathbb{E} 
\left[
\left( \Delta W_{j, n}  \right)^{2\overline{p}-1} 
\big| \mathcal{F}_{t_{n}} 
\right] 
= 0,\quad
\mathbb{E} 
\left[ 
\Delta W_{j, n} 
\ \big| \Pi_{j_{1}, j_{2}}^{t_{n}, t_{n+1}}\big|^{2} 
\Big| \mathcal{F}_{t_{n}} 
\right] = 0.
\end{equation}
In view of \eqref{equation:property-of-brownain-motion}, it is straightforward to see that
\begin{equation} \label{equation:crossing-term}
\begin{aligned}
\mathbb{E} 
\left[  
I_{1}I_{5}\big| \mathcal{F}_{t_{n}} 
\right] 
= 
\mathbb{E} 
\left[
I_{2}I_{5}
\big| \mathcal{F}_{t_{n}} 
\right] 
=\mathbb{E} 
\left[
I_{3}I_{5} 
\big| \mathcal{F}_{t_{n}} 
\right]
=  
\mathbb{E} 
\left[
I_{4}I_{5} 
\big| \mathcal{F}_{t_{n}} 
\right]  = 0 .
\end{aligned}
\end{equation}
Using \eqref{equation:p-th-moment-of-I1}-\eqref{equation:p-th-moment-of-I5}, \eqref{equation:2-moment-bound-of-I5} and \eqref{equation:crossing-term}, we derive from \eqref{equation:expansion-of-2nd-moment-of-xi} that
\begin{equation} 
\begin{aligned}
\mathbb{E}
\left[
\xi_{n+1}^{2} 
\big| \mathcal{F}_{t_{n}} 
\right] 
\leq 
C_{\epsilon_{1}}h 
+
\tfrac{
4h 
\|  
\mathscr{P}(Y_{n})^{T}
\sigma_{h}\left(\mathscr{P}(Y_{n})\right)  
\|^{2}  
}
{
(1+ \|\mathscr{P}(Y_{n}) \|^{2})^{2}
} .
\end{aligned}
\end{equation}
\textbf{Step III: estimate of $\mathbb{E} \left[ \xi_{n+1}^{3}| \mathcal{F}_{t_{n}} \right]$}

We begin with the following estimate from \eqref{equation:odd-power-property-of-brownian-motion}
\begin{equation} 
\begin{aligned}
\mathbb{E} 
\big[  
\left(I_{5}\right)^{3} 
\big| \mathcal{F}_{t_{n}} 
\big] 
= 0.
\end{aligned}
\end{equation}
This together with \eqref{equation:p-th-moment-of-I1}-\eqref{equation:p-th-moment-of-I5} and the elementary inequality yields
\begin{equation} 
\begin{aligned}
\mathbb{E} 
\left[
\xi_{n+1}^{3} 
\big| \mathcal{F}_{t_{n}} 
\right] 
& = 
\mathbb{E} 
\left[ 
\Big(
C_{\epsilon_{1}}h + I_{5} + \sum_{i=1}^{4} I_{i}
\Big)^{3} 
\Big| \mathcal{F}_{t_{n}} 
\right]  \\
& \leq 
C_{\epsilon_{1}}h 
+ \mathbb{E}
\left[
\Big(  I_{5} + \sum_{i=1}^{4} I_{i}\Big)^{3} 
\Big| \mathcal{F}_{t_{n}} 
\right] \\
&\leq 
C_{\epsilon_{1}}h 
+ \mathbb{E}\left[
\Big( \sum_{i=1}^{4} I_{i}\Big)^{3} 
+ 3I_{5}\Big( \sum_{i=1}^{4} I_{i}\Big)^{2} 
+ 3I_{5}^{2}\Big( \sum_{i=1}^{4} I_{i}\Big) 
\Big| \mathcal{F}_{t_{n}}
\right]\\
&\leq 
C_{\epsilon_{1}}h 
+ C \sum_{i=1}^{4}
\Big( 
\mathbb{E}
\left[
|I_{i}|^{3} 
\big| \mathcal{F}_{t_{n}} 
\right] 
+ 
\mathbb{E} 
\left[ 
|I_{5} |\cdot (I_{i})^{2} 
\big| \mathcal{F}_{t_{n}} 
\right] 
+ 
\mathbb{E} 
\left[
 |I_{i}| \cdot (I_{5})^{2}   \big| \mathcal{F}_{t_{n}} 
\right] 
\Big) . \\
\end{aligned}
\end{equation}
It follows from \eqref{equation:p-th-moment-of-I1} to \eqref{equation:p-th-moment-of-I5} and the H\"older inequality that, for $i\in \{1,2,3,4\}$,
\begin{equation}
\begin{aligned}
\mathbb{E}
\left[   
|I_{5} | \cdot (I_{i})^{2}
\big| \mathcal{F}_{t_{n}} 
\right]
\leq Ch^{5/4} 
\leq Ch , 
\quad
\mathbb{E}
\left[  
|I_{i}| \cdot (I_{5})^{2} 
\big| \mathcal{F}_{t_{n}} 
\right] 
\leq Ch .
\end{aligned}
\end{equation}
As discussed above, we get
\begin{equation} \label{equation:3rd-moment-of-xi}
\mathbb{E} 
\left[
\xi_{n+1}^{3}
\big| \mathcal{F}_{t_{n}} 
\right] 
\leq 
C_{\epsilon_{1}}h .
\end{equation}
\textbf{Step IV: estimate of $\mathbb{E} \left[ \xi_{n+1}^{k}| \mathcal{F}_{t_{n}} \right]$, $k \geq 4$}
%

By virtue of the estimates from \eqref{equation:p-th-moment-of-I1} to \eqref{equation:p-th-moment-of-I5}, one can achieve that
\begin{equation} \label{equation:4th-moment-of-xi}
\begin{aligned}
\mathbb{E} 
\left[ 
\xi_{n+1}^{k}
\big| \mathcal{F}_{t_{n}} 
\right] 
\leq 
&C_{\epsilon_{1}}h 
+
C\mathbb{E} 
\left[
\Big(\sum_{i=1}^{5} I_{i}\Big)^{k}
\Big| \mathcal{F}_{t_{n}} 
\right] 
\leq 
C_{\epsilon_{1}}h 
+C\sum_{i=1}^{5}
\mathbb{E} 
\left[ | I_{i}|^{k}
\big| \mathcal{F}_{t_{n}} 
\right]    
\leq C_{\epsilon_{1}}h .
\end{aligned}
\end{equation}
Combining theses estimates from \textbf{Step I} to \textbf{Step IV}, along with Assumption \ref{assumption:condition-on-the-MM}, 
shows that,  for any $p\in [1,p_{0})$,
\begin{footnotesize}
\begin{equation} 
\begin{aligned}
&\mathbb{E}
\left[
\left( 1 + \|Y_{n+1}\|^{2} \right)^{p} 
\big| \mathcal{F}_{t_{n}} 
\right] \\
&\leq 
\left( 1 + \| \mathscr{P}(Y_{n}) \|^{2} \right)^{p} 
\left(  
1 + C_{\epsilon_{1}}h 
+ph 
\tfrac{
(1+ \|\mathscr{P}(Y_{n}) \|^{2} )
\big(
2 \langle 
\mathscr{P}(Y_{n}) 
, \mu_{h}(\mathscr{P}(Y_{n}) ) 
\rangle 
+ (1 + \epsilon_{1} ) 
{\color{black}
\| \sigma_{h}\left(\mathscr{P}(Y_{n})\right)  \|_{F}^{2} 
}
\big) 
+ 2(p-1) 
\| 
\mathscr{P}(Y_{n})^{T} 
\sigma_{h}\left(\mathscr{P}(Y_{n})\right)   
\|^{2}   
}
{
(1+ \|\mathscr{P}(Y_{n}) \|^{2})^{2}
}   
\right)  \\
&\leq 
\left(
1 + \| \mathscr{P}(Y_{n}) \|^{2}
\right)^{p} 
\left(  
1 + C_{\epsilon_{1}}h 
+ph 
\tfrac{
(1+ \|\mathscr{P}(Y_{n}) \|^{2}) 
\big( 
2\langle 
\mathscr{P}(Y_{n}), 
\mu_{h}\left(\mathscr{P}(Y_{n})\right) 
\rangle   
+(2p-1+ \epsilon_{1}) 
{\color{black}
\| \sigma_{h}\left(\mathscr{P}(Y_{n})\right) \|_{F}^{2} 
}
\big) 
 }
{
( 1 + \|\mathscr{P}(Y_{n}) \|^{2})^{2}
}  
\right) \\
&\leq 
\left(
1 + \| Y_{n} \|^{2}
\right)^{p} 
(1+C_{\epsilon_{1}}h).  \\
\end{aligned}
\end{equation}
\end{footnotesize}
%
Taking expectations on both sides and the proof is completed by iteration.
\end{proof}
\section{Proof of lemmas in Section \ref{section:strong-convergence-order-of-the-MM} }  \label{section:strong-convergence-rate}
\subsection{Proof of Lemma \ref{lemma:strong-convergence-rate-of-sde-and-auxiliary-process}} \label{section:proof-of-lemma-and-auxiliary-lemmas}

\begin{proof} [Proof of Lemma \ref{lemma:strong-convergence-rate-of-sde-and-auxiliary-process}] \label{proof:lemma-strong-convergence-rate-of-sde-and-auxiliary-process}
We first show the moment estimate of the auxiliary process $\{\widetilde{Y}_{n}\}_{0\leq n \leq N}$ \eqref{equation:auxiliary-process}.
It follows from the iteration that, for every $n \in \{0,1,\dots,N-1 \}$,
\begin{equation}
\label{equation:auxiliary-process-after-iteration}
\begin{aligned}
\widetilde{Y}_{n+1} 
=
\underbrace{
X_{0} 
+\int_{0}^{t_{n+1}} \mu(X_{s}) \ \mathrm{d} s 
+\int_{0}^{t_{n+1}} \sigma(X_{s}) \ \mathrm{d} W_{s}
}_{=X_{t_{n+1}}} 
+ 
\underbrace{
\sum_{i=0}^{n} \sum_{j_{1},j_{2}=1}^{m} 
\left(
\mathcal{L}^{j_{1}} \sigma_{j_{2}}
\right)_{h}
\left(\mathscr{P}(Y_{i})\right) 
\Pi_{j_{1}, j_{2}}^{t_{i}, t_{i+1}}
}_{ \text{discrete}  \ \text{martingales} }.
\end{aligned}
\end{equation}
Indeed, the last term in \eqref{equation:auxiliary-process-after-iteration} is a discrete martingales because
\begin{equation}
\begin{aligned}
\mathbb{E}
\left[ 
\sum_{i=0}^{n} \sum_{j_{1},j_{2}=1}^{m} 
\left(
\mathcal{L}^{j_{1}} \sigma_{j_{2}}
\right)_{h}
\left(\mathscr{P}(Y_{i})\right) 
\Pi_{j_{1}, j_{2}}^{t_{i}, t_{i+1}}
\Big| \mathcal{F}_{t_{n}}
\right]
=
\sum_{i=0}^{n-1} \sum_{j_{1},j_{2}=1}^{m} 
\left(
\mathcal{L}^{j_{1}} \sigma_{j_{2}}
\right)_{h}
\left(\mathscr{P}(Y_{i})\right) 
\Pi_{j_{1}, j_{2}}^{t_{i}, t_{i+1}},
\end{aligned}
\end{equation}
where we have used \eqref{equation:property-of-brownain-motion} and a conditional expectation argument.
By the well-posedness of the SDE, 
the Burkholder-Davis-Gundy inequality (see Lemma 4.2 in \cite{wang2017strong}), 
condition \textit{(1)} 
in Assumption \ref{assumption:condition-on-the-MM}, 
\eqref{equation:moments-property-of-bronwain-motion} and Theorem \ref{theorem:moments-bound-of-MM},
it suffices to show that, 
\begin{equation}
\begin{aligned}
\big\|
\widetilde{Y}_{n+1}
\big\|_{L^{2p}(\Omega, \mathbb{R}^{d})} 
&
\leq 
\|X_{t_{n+1}}\|_{L^{2p}(\Omega, \mathbb{R}^{d})} 
+\Big\|  
\sum_{i=0}^{n} \sum_{j_{1},j_{2}=1}^{m} 
\left(
\mathcal{L}^{j_{1}} \sigma_{j_{2}}
\right)_{h}
\left(\mathscr{P}(Y_{i})\right) 
\Pi_{j_{1}, j_{2}}^{t_{i}, t_{i+1}}
 \Big\|_{L^{2p}(\Omega, \mathbb{R}^{d})} \\
& 
\leq 
C(1+ \|X_{0}\|_{L^{2p}(\Omega, \mathbb{R}^{d})}) 
+C\left(  \sum_{i=0}^{N} 
\Big\|
\sum_{j_{1},j_{2}=1}^{m} 
\left(
\mathcal{L}^{j_{1}} \sigma_{j_{2}}
\right)_{h}
\left(\mathscr{P}(Y_{i})\right) 
\Big\|^{2}_{L^{2p}(\Omega, \mathbb{R}^{d})} 
\right)^{\frac{1}{2}}\\
& \leq 
C(1+ \|X_{0}\|_{L^{2p}(\Omega, \mathbb{R}^{d})}).
\end{aligned}
\end{equation}
%
Now we turn to the proof of Lemma \ref{lemma:strong-convergence-rate-of-sde-and-auxiliary-process}. By iteration again, we have
\begin{equation}
\begin{aligned}
X_{t_{n+1}} - \widetilde{Y}_{n+1} 
= -\sum_{i=0}^{n}   \sum_{j_{1},j_{2}=1}^{m} 
\left(
\mathcal{L}^{j_{1}} \sigma_{j_{2}}
\right)_{h}
\left(\mathscr{P}(Y_{i})\right) 
\Pi_{j_{1}, j_{2}}^{t_{i}, t_{i+1}}.
\end{aligned}
\end{equation}
Thanks to \eqref{equation:moments-property-of-bronwain-motion}, \eqref{equation:growth-of-milstein-diffusion-term},
and the Burkholder-Davis-Gundy inequality \cite{wang2017strong}, 
one can similarly deduce
\begin{equation}
\begin{aligned}
\big\|
X_{t_{n+1}} - \widetilde{Y}_{n+1}
\big\|_{L^{2p}(\Omega, \mathbb{R}^{d})} 
&= 
\Big\|
\sum_{i=0}^{n} \sum_{j_{1},j_{2}=1}^{m}
\left(
\mathcal{L}^{j_{1}} \sigma_{j_{2}}
\right)_{h}
\left(\mathscr{P}(Y_{i})\right) 
\Pi_{j_{1}, j_{2}}^{t_{i}, t_{i+1}}
\Big\|_{L^{2p}(\Omega, \mathbb{R}^{d})} \\
&\leq 
C \left(
\sum_{i=0}^{N} 
\Big\|   
\sum_{j_{1},j_{2}=1}^{m} 
\left(
\mathcal{L}^{j_{1}} \sigma_{j_{2}}
\right)_{h}
\left(\mathscr{P}(Y_{i})\right) 
\Pi_{j_{1}, j_{2}}^{t_{i}, t_{i+1}}
\Big\|^{2}_{L^{2p}(\Omega, \mathbb{R}^{d})} 
\right)^{\frac{1}{2}} \\
& \leq 
Ch^{\frac{1}{2}}  
\Big(
1 
+ \sup_{0 \leq r \leq N} 
\left\| 
Y_{r}  
\right\|^{\gamma}_{L^{2p\gamma}(\Omega, \mathbb{R}^{d})}  
\Big)
\end{aligned}
\end{equation}
as required.
\end{proof}
\subsection{Proof of Lemma \ref{lemma:holder-continuity-of-MM}} \label{proof:lemma-holder-continuity-of-MM}
\begin{proof}[Proof of Lemma \ref{lemma:holder-continuity-of-MM}]
In view of \eqref{equation:moments-property-of-bronwain-motion}, \eqref{equation:continuous-time-version-of-MM-method},  the elementary inequality and a conditional expectation argument, we get, for $n\in \{0,1, \dots, N-1\}$, $N \in \mathbb{N}$,
\begin{equation} 
\begin{aligned} 
&\|
\mathbb{Y}^{n}(s)
\|_{L^{2p}(\Omega, \mathbb{R}^{d})}\\
&\leq 
\|
\mathbb{Y}^{n}(t_{n}) 
\|_{L^{2p}(\Omega, \mathbb{R}^{d})} 
+ 
(s-t_{n}) 
\big\|
\mu_{h}
\big(
\mathbb{Y}^{n}(t_{n}) 
\big)
\big\|_{L^{2p}(\Omega, \mathbb{R}^{d})} 
+ (s-t_{n})^{\frac{1}{2}} 
\big\|
\sigma_{h}
\big(
\mathbb{Y}^{n}(t_{n}) 
\big)
\big\|_{L^{2p}(\Omega, \mathbb{R}^{d\times m})} \\
&\ 
+C (s-t_{n}) 
\big\| 
\left(
\mathcal{L}^{j_{1}} \sigma_{j_{2}}
\right)_{h}
\big(
\mathbb{Y}^{n}(t_{n}) 
\big) 
\big\|_{L^{2p}(\Omega, \mathbb{R}^{d})}.
\end{aligned} 
\end{equation}
From \eqref{equation:upper-bound-of-diffusion}, it suffices to show that
\begin{equation}
\begin{aligned} 
(s-t_{n})^{p}  
{\color{black}
\mathbb{E} 
\left[
\left\|
\sigma_{h}
\left(
\mathbb{Y}^{n}(t_{n}) 
\right)
\right\|_{F}^{2p}
\right] 
}
&\leq 
C(s-t_{n})^{p} h^{-\frac{p}{2}}
\mathbb{E}
\left[
\left(
1 
+ 
\left\| 
\mathbb{Y}^{n}(t_{n})  
\right\|^{2} 
\right)^{p} 
\right]  \\
&\leq 
C\mathbb{E}
\left[ 
\left(
1 + 
\left\|  
Y_{n} 
\right\|^{2}  
\right)^{p}  
\right].
\end{aligned} 
\end{equation}
This together with condition \textit{(1)} in Assumption \ref{assumption:condition-on-the-MM} and Theorem \ref{theorem:moments-bound-of-MM}, leads to the estimate \eqref{eq:moment-bounds-of-the-continuous-time-version-of-MM-method} in Lemma \ref{lemma:holder-continuity-of-MM}.
Next, 
it is apparent from \eqref{equation:continuous-time-version-of-MM-method}, Assumption \ref{assumption:polynomial-growth-condition}, Assumption \ref{assumption:condition-on-the-MM} and Theorem \ref{theorem:moments-bound-of-MM} to obtain estimate \eqref{equation:holder-continuity-of-the-MM-method} in Lemma \ref{lemma:holder-continuity-of-MM}.
In addition,
employing these estimates with \eqref{equation:growth-of-drift}, \eqref{equation:growth-of-diffusion} and the H\"older inequality straightforwardly gives
\begin{equation}
\begin{aligned} 
\mathbb{E} 
\left[
\left\|
\mu\big(\mathbb{Y}^{n}(s)\big) -\mu\big(\mathbb{Y}^{n}(t_{n}) \big)  
\right\|^{2p}
\right] 
&\leq 
C\mathbb{E} 
\left[
\left(
1
+\big\|
\mathbb{Y}^{n}(s)   
\big\|
+\big\|
\mathbb{Y}^{n}(t_{n}) 
\big\| 
\right)^{2p(\gamma-1)} 
\big\|
\mathbb{Y}^{n}(s) - \mathbb{Y}^{n}(t_{n}) 
\big\|^{2p} 
\right]  \\
& 
\leq 
Ch^{p}  
\sup_{0 \leq r \leq N}
\mathbb{E}
\left[ 
\left( 1+\|Y_{r}\| \right)^{4p\gamma-2p} 
\right],
\end{aligned} 
\end{equation}
and
\begin{equation}
\begin{aligned} 
{\color{black}
\mathbb{E} 
\left[
\left\| 
\sigma\big( \mathbb{Y}^{n}(s)\big)
-\sigma\big(\mathbb{Y}^{n}(t_{n}) \big) 
\right\|_{F}^{2p}
 \right] 
 }
&\leq
C\mathbb{E} 
\left[ 
\left( 
1
+\big\|
\mathbb{Y}^{n}(s)
\big\|
+
\big\|
\mathbb{Y}^{n}(t_{n}) 
\big\| 
\right)^{p(\gamma-1)}
\big\|
\mathbb{Y}^{n}(s) -\mathbb{Y}^{n}(t_{n}) 
\big\|^{2p} 
\right]  \\
& \leq Ch^{p}  
\sup_{0 \leq r \leq N}
\mathbb{E}
\left[ 
\left(1+\|Y_{r}\|  \right)^{3p\gamma-p} 
\right].
\end{aligned} 
\end{equation}
Plugging Theorem \ref{theorem:moments-bound-of-MM}  into the estimate above with $p \in [1, \frac{p_{0}}{2\gamma-1} ]$  helps us complete the proof. 
\end{proof}
\subsection{Proof of Lemma \ref{lemma:errors-between-sde-and-auxiliary-process-in-continuous-version}} \label{proof:lemma-errors-between-sde-and-auxiliary-process-in-continuous-version}
\begin{proof}[Proof of Lemma \ref{lemma:errors-between-sde-and-auxiliary-process-in-continuous-version}]
Here we take the same argument as the proof of Lemma \ref{lemma:strong-convergence-rate-of-sde-and-auxiliary-process}. It follows that
\begin{equation}
\begin{aligned}
X_{s} - \widetilde{Y}(s) 
= X_{t_{n}} - \widetilde{Y}(t_{n}) 
- \sum_{j_{1},j_{2}=1}^{m} 
\left(
\mathcal{L}^{j_{1}} \sigma_{j_{2}}
\right)_{h}
\big(
\mathscr{P}(Y_{n})
\big) 
\Pi_{j_{1}, j_{2}}^{t_{n}, s}.
\end{aligned}
\end{equation}
Using the elementary inequality, Assumption \ref{assumption:condition-on-the-MM}, Lemma \ref{lemma:strong-convergence-rate-of-sde-and-auxiliary-process} and \eqref{equation:moments-property-of-bronwain-motion}  yields
\begin{equation}
\begin{aligned}
\big\|
X_{s} - \widetilde{Y}(s) 
\big\|_{L^{2p}(\Omega, \mathbb{R}^{d})} 
&\leq 
\big\|
X_{t_{n}} - \widetilde{Y}(t_{n}) 
\big\|_{L^{2p}(\Omega, \mathbb{R}^{d})} 
+ 
\Big\|
\sum_{j_{1},j_{2}=1}^{m} 
\left(
\mathcal{L}^{j_{1}} \sigma_{j_{2}}
\right)_{h}
\big(
\mathscr{P}(Y_{n}) 
\big) 
\Pi_{j_{1}, j_{2}}^{t_{n}, s}
\Big\|_{L^{2p}(\Omega, \mathbb{R}^{d})} \\
& \leq 
Ch^{\frac{1}{2}}  
\big(
1+\|Y_{r}\|^{\gamma}_{L^{2p\gamma}(\Omega, \mathbb{R}^{d})}  \big).
\end{aligned}
\end{equation}
To obtain the estimate \eqref{equation:estimates-in-lemma:errors-between-sde-and-auxiliary-process-in-continuous-version} in Lemma \ref{lemma:errors-between-sde-and-auxiliary-process-in-continuous-version}, one just needs to follow the same  analysis as the proof of Lemma \ref{lemma:holder-continuity-of-MM} and the proof is omitted.
\end{proof}
\subsection{Proof of Lemma \ref{lemma:strong-convergence-rate-of-MM-and-auxiliary-process}} 
\begin{proof} [Proof of Lemma \ref{lemma:strong-convergence-rate-of-MM-and-auxiliary-process}] \label{proof:lemma-strong-convergence-rate-of-MM-and-auxiliary-process}
By \eqref{equation:continuous-time-version-of-MM-method} and \eqref{equation:continuous-time-version-of-auxiliary-process}, we set $\widetilde{e}_{s} := \widetilde{Y}(s) - \mathbb{Y}^{n}(s)$ with $s \in [t_{n}, t_{n+1}]$, $n \in \{1,2, \dots, N \}$, $N\in \mathbb{N}$ and obtain
\begin{equation}
\begin{aligned}
\widetilde{e}_{t_{n+1}} 
=  
\widetilde{e}_{t_{n}} 
+ \int_{t_{n}}^{t_{n+1}} \mu(X_{s}) 
- \mu_{h}\big(\mathbb{Y}^{n}(t_{n}) \big) \mathrm{d} s 
+ \int_{t_{n}}^{t_{n+1}} 
\sigma(X_{s}) - \sigma_{h}
\big
(\mathbb{Y}^{n}(t_{n}) 
\big) 
 \mathrm{d} W_{s},
\end{aligned}
\end{equation}
which is an It\^o process. Using the It\^o formula yields, for any  positive constant $\eta > 0$, 
\begin{equation}
\begin{aligned} 
\left(
\eta + \| \widetilde{e}_{t_{n+1}} \|^{2} 
\right)^{p} 
&\leq 
\left(
\eta + \| \widetilde{e}_{t_{n}} \|^{2} 
\right)^{p} 
+ 
2p \int_{t_{n}}^{t_{n+1}}  
\left(
\eta + \| \widetilde{e}_{s} \|^{2} 
\right)^{p-1} 
\left\langle 
\widetilde{e}_{s} , 
\mu(X_{s}) 
- \mu_{h}
\big(
\mathbb{Y}^{n}(t_{n}) 
\big) 
\right\rangle \mathrm{d} s \\
& \ +p(2p-1)
\int_{t_{n}}^{t_{n+1}}  
\left(
\eta + \| \widetilde{e}_{s} \|^{2} 
\right)^{p-1}  
{\color{black}
\big\| 
\sigma(X_{s}) 
- \sigma_{h}
\big(
\mathbb{Y}^{n}(t_{n}) 
\big)
\big\|_{F}^{2}
}\mathrm{d} s . \\
\end{aligned} 
\end{equation}
Before proceeding further, it is easy to see that
\begin{small}
\begin{equation}
\begin{aligned} 
&\mu(X_{s}) 
- 
\mu_{h}
\big(
\mathbb{Y}^{n}(t_{n}) 
\big)\\
&= 
\mu\big(\widetilde{Y}(s)\big) 
- \mu\big(\mathbb{Y}^{n}(s)\big) 
+ 
\underbrace{
\mu(X_{s}) - \mu\big(\widetilde{Y}(s)\big) 
+ 
\mu\big(\mathbb{Y}^{n}(s)\big) 
- \mu\big(\mathbb{Y}^{n}(t_{n}) \big) 
+ \mathbb{T}_{h,\mu}\big(\mathbb{Y}^{n}(t_{n})  \big)
 }_{=: \tilde{R}_{s}^{\mu}}, \\
\end{aligned} 
\end{equation}
\end{small}
and 
\begin{small}
\begin{equation}
\begin{aligned} 
&\sigma(X_{s}) 
- \sigma_{h}\big(\mathbb{Y}^{n}(t_{n}) \big) \\
&= 
\sigma\big(\widetilde{Y}(s)\big) 
- \sigma\big(\mathbb{Y}^{n}(s)\big) 
+ 
\underbrace{
\sigma(X_{s}) - \sigma\big(\widetilde{Y}(s)\big) 
+ \sigma\big(\mathbb{Y}^{n}(s)\big) 
- \sigma\big(\mathbb{Y}^{n}(t_{n}) \big) 
+ \mathbb{T}_{h,\sigma}
\big(
\mathbb{Y}^{n}(t_{n}) 
\big)
 }_{=: \tilde{R}_{s}^{\sigma}}.
\end{aligned} 
\end{equation}
\end{small}
Taking expectations on both sides and using the Young inequality imply,
for some $\epsilon_{2} \in (0, \frac{2p_{1}-2p}{2p-1}]$,
\begin{footnotesize}
\begin{equation} \label{equation:ito-formula-1st-step}
\begin{aligned} 
&\mathbb{E}
\left[
\left(
\eta + \| \widetilde{e}_{t_{n+1}} \|^{2} 
\right)^{p}
\right] \\
&\leq 
\mathbb{E} 
\left[
\left(
\eta + \| \widetilde{e}_{t_{n}} \|^{2} 
\right)^{p}
\right] \\
&\ + p 
\mathbb{E} 
\left[
\int_{t_{n}}^{t_{n+1}}  
\left(
\eta + \| \widetilde{e}_{s} \|^{2} 
\right)^{p-1} 
\Big(
2\big\langle 
\widetilde{e}_{s} , \mu\big(\widetilde{Y}(s)\big) - \mu\big(\mathbb{Y}^{n}(s)\big) 
\big\rangle 
+ (2p_{1}-1) 
{\color{black}
\big\| 
\sigma\big(\widetilde{Y}(s)\big) 
 - \sigma\big(\mathbb{Y}^{n}(s)\big)
\big\|_{F}^{2}
}
\Big) \mathrm{d} s 
 \right]\\
&\ + 2p 
\mathbb{E} 
\left[
\int_{t_{n}}^{t_{n+1}}  
\left(
\eta + \| \widetilde{e}_{s} \|^{2} 
\right)^{p-1}
\big\langle 
\widetilde{e}_{s} , 
\widetilde{R}_{s}^{\mu} 
\big\rangle \mathrm{d} s 
\right] 
+ p(2p-1)
(
1+ \tfrac{1}{\epsilon_{2}}
)
\mathbb{E}
\left[
\int_{t_{n}}^{t_{n+1}}  
\left(
 \eta + \| \widetilde{e}_{s} \|^{2} 
\right)^{p-1} 
{\color{black}
\big\| 
\widetilde{R}_{s}^{\sigma} 
\big\|_{F}^{2} 
}
\mathrm{d} s 
\right] .
\end{aligned} 
\end{equation}
\end{footnotesize}
Due to  Assumption \ref{assumption:polynomial-growth-condition}, the Cauchy-Schwarz inequality and the Young inequality, one obtains
\begin{footnotesize}
\begin{equation} \label{equation:ito-formula-2nd-step}
\begin{aligned} 
&\mathbb{E} 
\left[
\left(
\eta + \| \widetilde{e}_{t_{n+1}} \|^{2} 
\right)^{p}
\right] 
\leq 
\mathbb{E} 
\left[
\left(
\eta + \| \widetilde{e}_{t_{n}} \|^{2} 
\right)^{p}
\right] 
+ C_{\epsilon_{2}} 
\mathbb{E} 
\left[
\int_{t_{n}}^{t_{n+1}}  
\left(
\eta + \| \widetilde{e}_{s} \|^{2} 
\right)^{p} \mathrm{d} s
\right] 
+ C_{\epsilon_{2}} 
\int_{t_{n}}^{t_{n+1}}  
\mathbb{E}
\left[
\big\|\widetilde{R}_{s}^{\mu} \big\|^{2p} 
\right] 
+ 
{\color{black}
\mathbb{E}
\left[ 
\big\|\widetilde{R}_{s}^{\sigma} \big\|_{F}^{2p} 
\right] 
}
\mathrm{d} s. \\ 
\end{aligned} 
\end{equation}
\end{footnotesize}
Bearing Assumption \ref{assumption:condition-on-the-MM}, Lemma \ref{lemma:holder-continuity-of-MM} and Lemma \ref{lemma:errors-between-sde-and-auxiliary-process-in-continuous-version} in mind, it is straightforward to derive the moment estimates for $\widetilde{R}_{s}^{\mu}$ and $\widetilde{R}_{s}^{\sigma}$, that is, for $p \in [1, \frac{p_{0}}{\max\{2\gamma-1, \alpha_{1}, \alpha_{2}\} } ] $,
\begin{equation} \label{equation:estimate-of-the-remaining-term-Rf-and-Rg}
\begin{aligned} 
&\mathbb{E}
\left[ 
 \big\|\widetilde{R}_{s}^{\mu} \big\|^{2p} 
\right] 
\bigvee 
{\color{black}
\mathbb{E}
\left[ 
\big\|\widetilde{R}_{s}^{\sigma} \big\|_{F}^{2p} 
\right] 
}
\leq Ch^{p}.
\end{aligned} 
\end{equation}
Letting $\eta \rightarrow 0^{+}$ and 
recalling that  
$$\widetilde{e}_{t_{n+1}} = \widetilde{Y}(t_{n+1}) - \mathbb{Y}^{n}(t_{n+1}) = \widetilde{Y}_{n+1} - Y_{n+1}$$ 
and  
$$\widetilde{e}_{t_{n}} = \widetilde{Y}(t_{n}) - \mathbb{Y}^{n}(t_{n}) = \widetilde{Y}_{n} - \mathscr{P}(Y_{n}) = \widetilde{Y}_{n} - Y_{n} + \mathcal{E}(Y_{n}),$$
one can derive from \eqref{equation:ito-formula-2nd-step} that
\begin{small}
\begin{equation}
\begin{aligned}
&\mathbb{E}
\left[ 
\big\|\widetilde{Y}_{n+1} - Y_{n+1} \big\|^{2p} 
\right] \\
&\leq 
\mathbb{E}
\left[ 
 \big\|
 \widetilde{Y}_{n} - Y_{n} + \mathcal{E}(Y_{n})
 \big\|^{2p} 
\right] 
+ C_{\epsilon_{2}}  
\mathbb{E} 
\left[\int_{t_{n}}^{t_{n+1}}  \left(\eta + \| \widetilde{e}_{s} \|^{2} \right)^{p} \mathrm{d} s 
\right] 
+C_{\epsilon_{2}} h^{p+1} \\
& \leq  
\mathbb{E}
\left[ 
\big\|\widetilde{Y}_{n} - Y_{n} \big\|^{2p} 
\right] 
+ 
\sum_{i=1}^{2p} \mathrm{C}_{2p}^{i}   
\mathbb{E}
\left[ 
\big\|
\widetilde{Y}_{n} - Y_{n}
\big\|^{2p-i} 
\cdot \|\mathcal{E}(Y_{n})\|^{i} 
\right]
+ C_{\epsilon_{2}}  
\mathbb{E} 
\left[
\int_{t_{n}}^{t_{n+1}}  
\left(
\eta + \| \widetilde{e}_{s} \|^{2} 
\right)^{p} \mathrm{d} s 
\right] 
+ C_{\epsilon_{2}} h^{p+1},
\end{aligned}
\end{equation}
\end{small}
where $\mathrm{C}_{2p}^{i}:= \frac{2p!}{i! (2p-i)!}$, $i \in \{1, \dots, 2p \}$. Using the Young inequality and Assumption \ref{assumption:condition-on-the-MM} gives
\begin{equation}
\begin{aligned}
\mathbb{E}
\left[ 
\big\|
\widetilde{Y}_{n} - Y_{n}
\big\|^{2p-i} 
\cdot 
\|\mathcal{E}(Y_{n})\|^{i} 
\right] 
&= 
\mathbb{E}
\left[  
h^{\frac{2p-i}{2p}}
\big\|\widetilde{Y}_{n} - Y_{n}\big\|^{2p-i} 
\cdot 
h^{-\frac{2p-i}{2p}}
\|\mathcal{E}(Y_{n})\|^{i} 
\right] \\
& \leq 
Ch \mathbb{E}
\left[ 
\big\|\widetilde{Y}_{n} - Y_{n}\big\|^{2p}
\right] 
+ Ch^{-\frac{2p-i}{i}} 
\mathbb{E}
\left[
\|\mathcal{E}(Y_{n})\|^{2p} 
\right] \\
& \leq 
Ch \mathbb{E}
\left[ 
\big\|\widetilde{Y}_{n} - Y_{n}\big\|^{2p}
\right] 
+ Ch^{1+4p-\frac{2p}{i}}  
\sup_{0 \leq r \leq N}
\mathbb{E}
\left[ 
\left(1+\|Y_{r}\|  \right)^{2\mathbf{a}p}
\right].
\end{aligned}
\end{equation}
Plugging the estimate \eqref{equation:estimate-of-the-remaining-term-Rf-and-Rg} into \eqref{equation:ito-formula-2nd-step} yields,
for $p \in [1,p_{1})\cap [1, \frac{p_{0}}{\max\{2\gamma-1, \alpha_{1}, \alpha_{2}, \mathbf{a}\} } ]  $,
\begin{equation}
\begin{aligned}
\mathbb{E}
\left[ 
\big\|\widetilde{Y}_{n+1} - Y_{n+1} \big\|^{2p} 
\right] 
\leq (1+Ch) 
\mathbb{E}
\left[ 
\big\|\widetilde{Y}_{n} - Y_{n} \big\|^{2p} 
\right] 
+ C_{\epsilon_{2}}  
\mathbb{E} 
\left[
\int_{t_{n}}^{t_{n+1}}  
\| \widetilde{e}_{s} \|^{2p}  \mathrm{d} s 
\right] 
+ C_{\epsilon_{2}} h^{p+1}.
\end{aligned}
\end{equation}
Since $\widetilde{Y}_{0} - Y_{0}=0$, we deduce
\begin{equation}
\begin{aligned}
\mathbb{E}
\left[ 
\big\|\widetilde{Y}_{n+1} - Y_{n+1} \big\|^{2p} 
\right] 
= 
\mathbb{E}
\left[ 
\big\|\widetilde{e}_{t_{n+1}} \big\|^{2p} 
\right] 
\leq  
C_{\epsilon_{2}}  
\mathbb{E} 
\left[
\int_{0}^{t_{n+1}}  
\| \widetilde{e}_{s} \|^{2p} \mathrm{d} s 
\right] 
+ C_{\epsilon_{2}} h^{p}.
\end{aligned}
\end{equation}
The proof  is completed with the help of the Gronwall inequality. 
\end{proof}
\section{Proof of lemmas in Section \ref{section:variance-analysis}}  \label{section:analysis-of-antithetic-MLMC-method}
\subsection{Proof of Lemma \ref{lemma:expression-of-yf-in-a-coarse-timestep}} \label{proof:expression-of-yf-in-a-coarse-timestep}
\begin{proof} [Proof of Lemma \ref{lemma:expression-of-yf-in-a-coarse-timestep}]
Combining \eqref{equation:MM-in-first-fine-timestep} with \eqref{equation:MM-in-second-fine-timestep}, we can then get a representation of $\{Y^{f} \}$ within the coarser timestep:
\begin{footnotesize}
\begin{equation} \label{equation:prior-expansion-of-yf}
\begin{aligned}
&Y_{n+1}^{f} \\
&= \mathscr{P}\big(Y_{n}^{f} \big) 
+ \mu_{h}
\big(
\mathscr{P}\big(Y_{n}^{f} \big)
\big)h 
+ \sigma_{h}
\big(
\mathscr{P}\big(Y_{n}^{f} \big)
\big) \Delta W_{n} 
+\sum_{j_{1},j_{2}=1}^{m} 
\left(
\mathcal{L}^{j_{1}} \sigma_{j_{2}}
\right)_{h}
\big(
\mathscr{P}\big(Y_{n}^{f} \big)
\big)
\Pi_{j_{1}, j_{2}}^{t_{n}, t_{n+1}} \\
& \ + 
\Big( 
 \mu_{\frac{h}{2}}
 \big(
 \mathscr{P}\big(Y_{n+1/2}^{f} \big)
 \big)
 +\mu_{\frac{h}{2}}
 \big(
 \mathscr{P}\big(Y_{n}^{f} \big)
 \big) 
\Big)\tfrac{h}{2} 
- \mu_{h}
\big(
\mathscr{P}\big(Y_{n}^{f} \big)
\big)h - \mathcal{E}(Y^{f}_{n+1/2}) \\
& \ + \sigma_{\frac{h}{2}}
\big(
\mathscr{P}\big(Y_{n}^{f} \big)
\big) \delta W_{n} 
+ \sigma_{\frac{h}{2}}
\big(
\mathscr{P}\big(Y_{n+1/2}^{f} \big)
\big) \delta W_{n+1/2} 
- \sigma_{h}
\big(
\mathscr{P}\big(Y_{n}^{f} \big)
\big) \Delta W_{n} \\
& \ 
+ \sum_{j_{1},j_{2}=1}^{m}
\left(
\left(
\mathcal{L}^{j_{1}} \sigma_{j_{2}}
\right)_{\frac{h}{2}}
\big(
\mathscr{P}\big(Y_{n}^{f} \big)
\big) 
\Pi_{j_{1}, j_{2}}^{t_{n}, t_{n+1/2}} 
+ 
\left(
\mathcal{L}^{j_{1}} \sigma_{j_{2}}
\right)_{\frac{h}{2}}
\big(
\mathscr{P}\big(Y_{n+1/2}^{f} \big)
\big) 
\Pi_{j_{1}, j_{2}}^{t_{n+1/2}, t_{n+1}}
- 
\left(
\mathcal{L}^{j_{1}} \sigma_{j_{2}}
\right)_{h}
\big(
\mathscr{P}\big(Y_{n}^{f} \big)
\big) 
\Pi_{j_{1}, j_{2}}^{t_{n}, t_{n+1}} 
\right).
\end{aligned}
\end{equation}
\end{footnotesize}
Since
\begin{equation} \label{equation:delta-and-Delta-brownian-increments}
\begin{aligned}
\Delta W_{n} = \delta W_{n} +  \delta W_{n+1/2}, 
\end{aligned}
\end{equation}
we have
\begin{equation}
\begin{aligned}
\Pi_{j_{1}, j_{2}}^{t_{n}, t_{n+1}} 
= \Pi_{j_{1}, j_{2}}^{t_{n}, t_{n+1/2}} 
+ \Pi_{j_{1}, j_{2}}^{t_{n+1/2}, t_{n+1}} 
+ \tfrac{1}{2} 
\left(
\delta W_{j_{1}, n} \delta W_{j_{2}, n+1/2} 
+ \delta W_{j_{1}, n+1/2}\delta W_{j_{2}, n} 
\right).
\end{aligned}
\end{equation}
Plugging these estimates into \eqref{equation:prior-expansion-of-yf} yields
\begin{footnotesize}
\begin{equation} \label{equation:corrected-expansion-of-yf}
\begin{aligned}
&Y_{n+1}^{f} \\
&= \mathscr{P}\big(Y_{n}^{f} \big) 
+ \mu_{h}
\big(
\mathscr{P}\big(Y_{n}^{f} \big)
\big)h 
+ \sigma_{h}
\big(
\mathscr{P}\big(Y_{n}^{f} \big)
\big) \Delta W_{n} 
+ \sum_{j_{1},j_{2}=1}^{m} 
\left(
\mathcal{L}^{j_{1}} \sigma_{j_{2}}
\right)_{h}
\big(
 \mathscr{P}\big(Y_{n}^{f} \big)
\big)
\Pi_{j_{1}, j_{2}}^{t_{n}, t_{n+1}} \\
& \ -\tfrac{1}{2}\sum_{j_{1},j_{2}=1}^{m}
\left(
\mathcal{L}^{j_{1}} \sigma_{j_{2}}
\right)_{h}
\big(
\mathscr{P}\big(Y_{n}^{f} \big)
\big) 
\left(
\delta W_{j_{1}, n} \delta W_{j_{2}, n+1/2} 
- \delta W_{j_{1}, n+1/2} \delta W_{j_{2}, n} 
\right) \\
& \ 
+ \underbrace{
\Big( 
\mu_{\frac{h}{2}}
\big(
\mathscr{P}\big(Y_{n+1/2}^{f} \big)
\big)
+\mu_{\frac{h}{2}}
\big(
\mathscr{P}\big(Y_{n}^{f} \big)
\big) 
\Big)\tfrac{h}{2} 
- \mu_{h}
\big(
\mathscr{P}\big(Y_{n}^{f} \big)
\big)h
}_{=:R_{n+1/2,f}} 
- \mathcal{E}(Y^{f}_{n+1/2}) \\
& \
+ \underbrace{
\sigma_{\frac{h}{2}}
\big(
\mathscr{P}\big(Y_{n}^{f} \big)
\big) \delta W_{n} 
+ 
\sigma_{\frac{h}{2}}
\big(
\mathscr{P}\big(Y_{n+1/2}^{f} \big)
\big) \delta W_{n+1/2} 
- \sigma_{h}
\big(
\mathscr{P}\big(Y_{n}^{f} \big)
\big) \Delta W_{n} 
- \sum_{j_{1}, j_{2}=1}^{m}
\left(
\mathcal{L}^{j_{1}} \sigma_{j_{2}}
\right)_{h}
\big(
\mathscr{P}\big(Y_{n}^{f} \big)
\big) 
\delta W_{j_{2}, n}\delta W_{j_{1}, n+1/2} 
}_{=:M_{n+1,f}^{(1)}} \\
& \ 
+ \sum_{j_{1},j_{2}=1}^{m} 
\bigg[
\left( 
\left(  
\mathcal{L}^{j_{1}} \sigma_{j_{2}}
\right)_{\frac{h}{2}}
\big(
\mathscr{P}\big(Y_{n}^{f} \big)
\big) 
- 
\left(
\mathcal{L}^{j_{1}} \sigma_{j_{2}}
\right)_{h}
\big(
\mathscr{P}\big(Y_{n}^{f} \big)
\big) 
 \right) 
 \Pi_{j_{1}, j_{2}}^{t_{n}, t_{n+1/2}}  \\
& \underbrace{
\hspace{12em} 
+ \left(
\left(
\mathcal{L}^{j_{1}} \sigma_{j_{2}}
\right)_{\frac{h}{2}}
\big(
\mathscr{P}\big(Y_{n+1/2}^{f} \big)
\big)
- 
\left(
\mathcal{L}^{j_{1}} \sigma_{j_{2}}
\right)_{h}
\big(
\mathscr{P}\big(Y_{n}^{f} \big)
\big) 
\right)
\Pi_{j_{1}, j_{2}}^{t_{n+1/2}, t_{n+1}}
\bigg] 
}_{=:M_{n+1,f}^{(2)}}.
\end{aligned}
\end{equation}
\end{footnotesize}
For the term $R_{n+1/2,f}$, we recall the notations in \eqref{equation:tamed-correction-function} and \eqref{equation:projected-corrected-function} and use 
the Taylor expansion  \eqref{equation:taylor-expansion-in-mlmc} to show
\begin{equation} \label{equation: remaining term in yf}
\begin{aligned}
&R_{n+1/2,f} \\
&=  
\left(
\mu\big(Y_{n+1/2}^{f} \big)
-\mu\big(\mathscr{P}\big(Y_{n}^{f} \big)\big) 
+ \mathbb{T}_{\frac{h}{2}, \mu}
(
\mathscr{P}\big(Y_{n}^{f} \big)
) 
- 
\mathbb{T}_{\frac{h}{2},\mu}(Y_{n+1/2}^{f}) 
+ 2\mathbb{T}_{h, \mu}
(
\mathscr{P}\big(Y_{n}^{f} \big)
) 
- 
\mathcal{E}_{\mu}
(
Y^{f}_{n+1/2}
)
\right) \tfrac{h}{2} \\
&= 
\left(
\mathcal{M}_{n+1/2,f}^{(\mu)} 
+ \mathcal{N}_{n+1/2,f}^{(\mu)}
+ \mathbb{T}_{\frac{h}{2}, \mu}
(
\mathscr{P}\big(Y_{n}^{f} \big)
) 
- 
\mathbb{T}_{\frac{h}{2},\mu}
(
Y_{n+1/2}^{f}
) 
+ 2\mathbb{T}_{h, \mu}
(
\mathscr{P}\big(Y_{n}^{f} \big)
)  
- \mathcal{E}_{\mu}(Y^{f}_{n+1/2}) 
\right)\tfrac{h}{2}.
\end{aligned} 
\end{equation}
This together with \eqref{equation:corrected-expansion-of-yf}, completes the proof of \eqref{equation:expression-of-yf-in-lemma} in Lemma \ref{lemma:expression-of-yf-in-a-coarse-timestep}.
Based on \eqref{equation:growth-of-partial-drift-and-diffusion}, \eqref{equation:growth-of-diffusion}, we obtain that
\begin{equation}\label{equation:estimate-of-remaining-term-M(mu)}
\begin{aligned}
\mathbb{E} 
\left[ 
\big\|\mathcal{M}_{n+1/2,f}^{(\mu)} \big\|^{2p}  
\right] 
= 
\mathbb{E} 
\left[
\big\| 
\tfrac{\partial \mu}{\partial y}
\big(
\mathscr{P}\big(Y_{n}^{f} \big)
\big)
\sigma_{\frac{h}{2}}
\big(
\mathscr{P}\big(Y_{n}^{f} \big)
\big) \delta W_{n} 
\big\|^{2p}  
\right]  
\leq 
Ch^{p}  
\sup_{0 \leq r \leq 2^{\ell -1}} 
\mathbb{E}
\left[ 
(1+ \|Y_{r}\| )^{3p\gamma-p} 
\right].
\end{aligned}
\end{equation}
Also, in light of Assumption \ref{assumption:polynomial-growth-condition}
 and Lemma \ref{lemma:holder-continuity-of-MM}, one can further use the H\"older inequality to get
\begin{equation} \label{equation:moments-bound-of-remaining-term-R}
\begin{aligned}
&\mathbb{E}
\left[
\big\|
\mathcal{R}_{\mu}
\big(
\mathscr{P}\big(Y_{n}^{f} \big), Y_{n+1/2}^{f}
\big) 
\big\|^{2p}
\right] \\
&\leq 
\int_{0}^{1}
\mathbb{E}
\left[
\Big\|
\Big[
\tfrac{\partial \mu}{\partial x}
\Big(
\mathscr{P}\big(Y_{n}^{f} \big)
+r
\big(
Y_{n+1/2}^{f}-\mathscr{P}\big(Y_{n}^{f} \big)
\big)
\Big)
-\tfrac{\partial \mu}{\partial x}
\big(
\mathscr{P}\big(Y_{n}^{f} \big)
\big) 
\Big]
\big(
Y_{n+1/2}^{f}-\mathscr{P}\big(Y_{n}^{f} \big)
\big) 
\Big\|^{2p}
\right]\mathrm{d} r \\
&\leq 
C \int_{0}^{1} 
\mathbb{E} 
\left[
\left(
1 + 
\big\|
r Y_{n+1/2}^{f} + (1-r)\mathscr{P}\big(Y_{n}^{f} \big)
\big\| 
+ 
\big\| 
\mathscr{P}\big(Y_{n}^{f} \big)  
\big\| 
\right)^{2p(\gamma-2)} 
\big\| 
Y_{n+1/2}^{f} - \mathscr{P}\big(Y_{n}^{f} \big)  
\big\| ^{4p}  
\right] \mathrm{d} r \\
& \leq Ch^{2p} \sup_{0 \leq r \leq 2^{\ell -1} } \mathbb{E} \left[ \left(1 + \|Y_{r}  \| \right)^{\max\{4p\gamma,6p\gamma-4p\} } \right]. 
\end{aligned} 
\end{equation}
That is to say, the estimate \eqref{equation:moments-bound-of-remaining-term-R}, together with \eqref{equation:growth-of-partial-drift-and-diffusion}, \eqref{equation:growth-of-drift} and \eqref{equation:growth-of-milstein-diffusion-term}, shows that
\begin{equation} 
\begin{aligned}
\mathbb{E} 
\left[
\big\|
\mathcal{N}_{n+1/2,f}^{(\mu)} 
\big\|^{2p}  
\right]  
\leq Ch^{2p} 
\sup_{0 \leq r \leq 2^{\ell -1} } 
\mathbb{E}
\left[ 
(1+ \|Y_{r}\| )^{\max\{4p\gamma,6p\gamma-4p\}} 
\right].
\end{aligned} 
\end{equation}
By Assumption \ref{assumption:condition-on-the-MM} and elementary inequalities, it follows that
%
\begin{equation} \label{equation:tamed-correction-in-mlmc}
\begin{aligned}
&\mathbb{E} 
\left[ 
\big\|
\mathbb{T}_{\frac{h}{2}, \mu}
\big(
\mathscr{P}\big(Y_{n}^{f} \big)
\big) 
- 
\mathbb{T}_{\frac{h}{2},\mu}
\big(
\mathscr{P}\big(Y_{n+1/2}^{f} \big)
\big) 
+ 2\mathbb{T}_{h, \mu}
\big(
\mathscr{P}\big(Y_{n}^{f} \big)
\big) 
- 
\mathcal{E}_{\mu}\big(Y^{f}_{n+1/2}\big) 
\big\|^{2p}  
\right]  \\
& \hspace{16em}
\leq 
Ch^{2p} 
\sup_{0 \leq r \leq 2^{\ell -1} } 
\mathbb{E}
\left[
(1+ \|Y_{r}\| )^{
2p\max\{\alpha_{1}, \mathbf{a}+\gamma-1 \}
}
\right].
\end{aligned}
\end{equation}
Next we focus on the estimate of the term $M_{n+1,f}^{(1)}$. 
Thanks to \eqref{equation:milstein-term}, \eqref{equation:MM-in-first-fine-timestep}, the Taylor expansion \eqref{equation:taylor-expansion-in-mlmc},  \eqref{equation:delta-and-Delta-brownian-increments} and a similar argument as \eqref{equation: remaining term in yf}, $M_{n+1,f}^{(1)}$ can be rewritten as 
\begin{footnotesize}
\begin{equation}
\begin{aligned}
&M_{n+1,f}^{(1)} \\
&= 
\sum_{j_{1}=1}^{m}
\Big(
\sigma_{j_{1}}\big(Y_{n+1/2}^{f}\big)
-\sigma_{j_{1}}
\big(
\mathscr{P}\big(Y_{n}^{f} \big)
\big)
-\sum_{j_{2}=1}^{m}
\mathcal{L}^{j_{1}} \sigma_{j_{2}}
\big(
\mathscr{P}\big(Y_{n}^{f} \big)
\big)
\delta W_{j_{2}, n}  
\Big)
\delta W_{j_{1}, n+1/2} 
+ 
\mathbb{T}_{h, \sigma}
\big(
\mathscr{P}\big(Y_{n}^{f} \big)
\big) \Delta W_{n}  \\
& \ - 
\left(
\mathbb{T}_{\frac{h}{2}, \sigma}
\big(
\mathscr{P}\big(Y_{n}^{f} \big)
\big) 
+ 
\mathbb{T}_{\frac{h}{2}, \sigma}
\big(
\mathscr{P}\big(Y_{n+1/2}^{f} \big)
\big)
+\mathcal{E}_{\sigma}(Y^{f}_{n+1/2}) 
\right) \delta W_{n+1/2} 
+ \sum_{j_{1},j_{2}=1}^{m} 
\mathbb{T}_{\frac{h}{2}, \mathcal{L}_{j_{1} j_{2}}}
\big(
\mathscr{P}\big(Y_{n}^{f} \big)
\big)
\delta W_{j_{2}, n} \delta W_{j_{1}, n+1/2}   \\
& = \sum_{j_{1}=1}^{m} 
\Bigg[
\tfrac{\partial \sigma_{j_{1}}}{\partial y}
\big(
\mathscr{P}\big(Y_{n}^{f} \big)
\big)  
\bigg(
\mu_{\frac{h}{2}}
\big(
\mathscr{P}\big(Y_{n}^{f} \big)
\big) \tfrac{h}{2} 
+ \sum_{j'_{1}, j'_{2}=1}^{m} 
\left(
\mathcal{L}^{j'_{1}} \sigma_{j'_{2}}
\right)_{\frac{h}{2}}
\big(
\mathscr{P}\big(Y_{n}^{f} \big)
\big)
\Pi_{j'_{1}, j'_{2}}^{t_{n}, t_{n+1/2}} 
\bigg) \\
&\ + \mathcal{R}_{\sigma_{j_{1}}}
\left(
\mathscr{P}\big(Y_{n}^{f} \big), Y_{n+1/2}^{f}
\right) 
\Bigg] \delta W_{j_{1}, n+1/2} 
+ \mathbb{T}_{h, \sigma}
\big(
\mathscr{P}\big(Y_{n}^{f} \big)
\big) \Delta W_{n}
+  
\sum_{j_{1},j_{2}=1}^{m} 
\mathbb{T}_{\frac{h}{2}, \mathcal{L}_{j_{1} j_{2}}}
\big(
\mathscr{P}\big(Y_{n}^{f} \big)
\big)
\delta W_{j_{2}, n} \delta W_{j_{1}, n+1/2}\\
&\ - 
\left(
\mathbb{T}_{\frac{h}{2}, \sigma}
\big(
\mathscr{P}\big(Y_{n}^{f} \big)
\big) 
+ 
\mathbb{T}_{\frac{h}{2}, \sigma}
\big(
\mathscr{P}\big(Y_{n+1/2}^{f} \big) 
\big) 
+ \mathcal{E}_{\sigma}(Y^{f}_{n+1/2})
\right) \delta W_{n+1/2}\\
& = \sum_{j_{1}=1}^{m}
\mathcal{N}_{n+1/2,f}^{(\sigma_{j_{1}})} \delta W_{j_{1}, n+1/2} 
+ \mathbb{T}_{h, \sigma}
\big(
\mathscr{P}\big(Y_{n}^{f} \big)
\big) \Delta W_{n}
+ \sum_{j_{1},j_{2}=1}^{m} 
\mathbb{T}_{\frac{h}{2}, \mathcal{L}_{j_{1} j_{2}}}
\big(
\mathscr{P}\big(Y_{n}^{f} \big)
\big)
\delta W_{j_{2}, n} \delta W_{j_{1}, n+1/2}\\
&\ - 
\left(
\mathbb{T}_{\frac{h}{2}, \sigma}
\big(
\mathscr{P}\big(Y_{n}^{f} \big)
\big) 
+ \mathbb{T}_{\frac{h}{2}, \sigma}
\big(
\mathscr{P}\big(Y_{n+1/2}^{f} \big) 
\big) 
+ \mathcal{E}_{\sigma}(Y^{f}_{n+1/2})
\right) \delta W_{n+1/2}. \\
\end{aligned}
\end{equation}
\end{footnotesize}
Following the same arguments as \eqref{equation:moments-bound-of-remaining-term-R}-\eqref{equation:tamed-correction-in-mlmc} and using \eqref{equation:moments-property-of-bronwain-motion} imply
\begin{equation}
\begin{aligned}
\mathbb{E}
\left[
\big\| M_{n+1,f}^{(1)}  \big\|^{2p}
\right] 
\leq 
Ch^{3p} 
\sup_{0 \leq r \leq 2^{\ell -1} } 
\mathbb{E}
\left[ 
\left(1+ \|Y_{r}\| \right)^{
p\max\{5\gamma-3, 2\mathbf{a}+\gamma-1, 2\alpha_{2}, 2\alpha_{3} \}
} 
\right].
\end{aligned}
\end{equation}
In a similar way, $M_{n+1,f}^{(2)} $ can be divided into several parts as follows:
\begin{footnotesize}
\begin{equation}
\begin{aligned}
M_{n+1,f}^{(2)} 
&= \sum_{j_{1},j_{2}=1}^{m} 
\bigg[ 
\left(
\mathbb{T}_{h,\mathcal{L}_{j_{1} j_{2}}}
\big(
\mathscr{P}\big(Y_{n}^{f} \big)
\big)
-
\mathbb{T}_{\frac{h}{2},\mathcal{L}_{j_{1} j_{2}}}
\big(
\mathscr{P}\big(Y_{n}^{f} \big) 
\big) 
\right) 
\Pi_{j_{1}, j_{2}}^{t_{n}, t_{n+1/2}} \\
&\ + 
\left(
\mathbb{T}_{h,\mathcal{L}_{j_{1} j_{2}}}
\big(
\mathscr{P}\big(Y_{n}^{f} \big)
\big)
-\mathbb{T}_{\frac{h}{2},\mathcal{L}_{j_{1} j_{2}}}
\big(
\mathscr{P}\big(Y_{n+1/2}^{f} \big)
\big)
- \mathcal{E}_{\mathcal{L}_{j_{1}j_{2}}}
\big(Y_{n+1/2}^{f}\big)
\right) \Pi_{j_{1}, j_{2}}^{t_{n+1/2}, t_{n+1}} \\
&\ 
+ \left(
\mathcal{L}^{j_{1}} \sigma_{j_{2}} 
\big( Y_{n+1/2}^{f} \big)
-\mathcal{L}^{j_{1}} \sigma_{j_{2}}
\big(
\mathscr{P}\big(Y_{n}^{f} \big)
\big) \right) 
\Pi_{j_{1}, j_{2}}^{t_{n+1/2}, t_{n+1}} 
\bigg]. 
\end{aligned}
\end{equation}
\end{footnotesize}
Obviously, \eqref{equation:moments-property-of-bronwain-motion}, Lemma \ref{lemma:holder-continuity-of-MM} and the analysis in \eqref{equation:tamed-correction-in-mlmc} help us to get
\begin{equation}
\begin{aligned}
\mathbb{E} 
\left[
\big\|   M_{n+1,f}^{(2)}  \big\|^{2p} 
\right] 
\leq Ch^{3p} 
\sup_{0 \leq r \leq 2^{\ell -1} } 
\mathbb{E}
\left[
\left(1+ \|Y_{r}\| \right)^{2p\max\{2\gamma-1, \alpha_{3} \}}
\right].
\end{aligned}
\end{equation}
Recalling
\begin{equation}
\begin{aligned}
M_{n+1}^{f} 
&= 
M_{n+1,f}^{(1)}
+M_{n+1,f}^{(2)}
+\tfrac{h}{2}
\mathcal{M}_{n+1/2,f}^{(\mu)},\\
B_{n+1}^{f} 
&=
\left(
\mathcal{N}_{n+1/2,f}^{(\mu)} 
+ \mathbb{T}_{\frac{h}{2},\mu}
\big(
\mathscr{P}\big(Y_{n}^{f} \big)
\big) 
- \mathbb{T}_{\frac{h}{2},\mu}(Y_{n+1/2}^{f}) 
+ 2\mathbb{T}_{h,\mu}(\mathscr{P}\big(Y_{n}^{f} \big))
- \mathcal{E}_{\mu}(Y^{f}_{n+1/2})
\right)\tfrac{h}{2} 
- \mathcal{E}( Y^{f}_{n+1/2} ),
\end{aligned}
\end{equation}
we can see $\mathbb{E} [M_{n+1}^{f} | \mathcal{F}_{t_{n}}] = 0$ due to the fact that $\delta W_{n}$  and $\delta W_{n+1/2}$ are independent of $\mathscr{P}\big(Y_{n}^{f} \big)$ and $\mathscr{P}\big(Y_{n+1/2}^{f} \big)$, respectively.
To sum up, in light of Assumption \ref{assumption:condition-on-the-MM} and Theorem \ref{theorem:moments-bound-of-MM}, for any $p \in [1,\frac{p_{0}}{\max\{3\gamma-2, 2\gamma, \mathbf{a}+\gamma-1, \alpha_{1}, \alpha_{2}, \alpha_{3} \}}]$,  we have
\begin{equation}
\begin{aligned}
\mathbb{E}
\left[
\big\| M^{f}_{n+1} \big\|^{2p} 
\right] 
\leq Ch^{3p}, 
\quad 
\mathbb{E}
\left[
\big\| B^{f}_{n+1} \big\|^{2p} 
\right] 
\leq  Ch^{4p}.
\end{aligned}
\end{equation}
The proof is completed.
\end{proof}
\subsection{Proof of Lemma \ref{lemma:expression-of-ya-in-a-coarse-timestep}} \label{proof:expression-of-ya-in-a-coarse-timestep}
\begin{proof}[Proof of Lemma \ref{lemma:expression-of-ya-in-a-coarse-timestep}]
We only check \eqref{equation:estimate-of-martingale-term-of-ya} in Lemma \ref{lemma:expression-of-ya-in-a-coarse-timestep} since the main proof and the notations are almost the same as Lemma \ref{lemma:expression-of-yf-in-a-coarse-timestep}.
Take the estimate of $M_{n+1,a}^{(1)}$ as an example. It is clear that  $\delta W_{n}$, $\delta W_{n+1/2}$ and $\Delta W_{n}$ are independent of $\mathcal{F}_{t_{n}}$. It follows
\begin{equation}
\mathbb{E} 
\left[
\sigma_{\frac{h}{2}}
\big(
\mathscr{P}\big(Y_{n}^{a}\big)
\big) \delta W_{n+1/2} 
| \mathcal{F}_{t_{n}}
\right] 
=         
\mathbb{E} 
\left[
\sigma_{h}
\big(
\mathscr{P}\big(Y_{n}^{a}\big)
\big) \Delta W_{n} | \mathcal{F}_{t_{n}}
\right]  
=0,
\end{equation}
and
\begin{equation}
\sum_{j_{1}, j_{2}=1}^{m}
\mathbb{E} 
\left[
\left(
\mathcal{L}^{j_{1}} \sigma_{j_{2}}
\right)_{h}
\big(
\mathscr{P}\big(Y_{n}^{a}\big)
\big) 
\delta W_{j_{1}, n}\delta W_{j_{2}, n+1/2}
\big| \mathcal{F}_{t_{n}}
\right]=0.
\end{equation}
For the remaining term $\sigma_{\frac{h}{2}}(\mathscr{P}(Y_{n+1/2}^{a})) \delta W_{n}$ in $M_{n+1,a}^{(1)}$, in view of \eqref{equation:antithetic-estimator-in-first-fine-step}, $\sigma_{\frac{h}{2}}(\mathscr{P}(Y_{n+1/2}^{a}))$ can be regarded as a function $G$ of $\mathscr{P}(Y_{n}^{a})$ and $\delta W_{n+1/2}$, i.e.,
\begin{equation}
\sigma_{\frac{h}{2}}
\big(
\mathscr{P}\big(Y_{n+1/2}^{a}\big)
\big) \delta W_{n} 
:= G
\big(
\mathscr{P}\big(Y_{n}^{a}\big),  \delta W_{n+1/2}
\big) \delta W_{n}.
\end{equation}
Note that $\mathscr{P}\big(Y_{n}^{a}\big)$ is measurable of $\mathcal{F}_{t_{n}}$ and $\delta W_{n}$ is independent of $\delta W_{n+1/2}$, we use a conditional expectation argument (see Lemma 5.1.12 in \cite{jentzen2016stochastic}) to show that
\begin{small}
\begin{equation}
\mathbb{E} 
\left[
\sigma_{\frac{h}{2}}
\big(
\mathscr{P}\big(Y_{n+1/2}^{a}\big)
\big) 
\delta W_{n} | \mathcal{F}_{t_{n}}
\right] 
= \mathbb{E} 
\big[
G\big(\mathscr{P}\big(Y_{n}^{a}\big),  \delta W_{n+1/2}\big) \delta W_{n} | \mathcal{F}_{t_{n}}
\big] 
= 
\mathbb{E} 
\big[
G\big(x,  \delta W_{n+1/2}\big) \delta W_{n}
\big] 
\Big|_{
x=\mathscr{P}\big(Y_{n}^{a}\big)
} =0.
\end{equation}
\end{small}
Following the same idea, one can get
\begin{equation}
\mathbb{E}
\big[
M_{n+1}^{a} | \mathcal{F}_{t_{n}}
\big] 
= 
\mathbb{E} 
\left[
M_{n+1,a}^{(1)}
+M_{n+1,a}^{(2)}
+\tfrac{h}{2}
\mathcal{M}_{n+1,a}^{(\mu)} 
\Big| \mathcal{F}_{t_{n}} 
\right]  = 0.
\end{equation}
In addition, it is noteworthy that the sign change in \eqref{equation:expression-of-ya-in-a-coarse-timestep} arises from the swapping of the Brownian motion increments  when constructing $\{{Y}^{a}\}$.
\end{proof}
\subsection{Proof of Lemma \ref{lemma:representation-of-the-average-process-in-a-coarse-timestep}} \label{proof:representation-of-the-average-process-in-a-coarse-timestep}
\begin{proof}[Proof of Lemma \ref{lemma:representation-of-the-average-process-in-a-coarse-timestep}]
%
Thanks to Theorem \ref{theorem:moments-bound-of-MM} and the fact that $Y^{f}$ and ${Y}^{a}$ have the same distribution,
it is obvious to show, 
for $p \in [1, p_{0})$, 
\begin{equation} \label{equation:moments-bound-of-the-average-process-ybar}
 \sup_{0 \leq n \leq N } 
 \mathbb{E} 
 \left[
 \big\|  \overline{Y}_{n}^{f}  \big\| ^{2p} 
 \right] 
 \leq C 
 \left(
 1+ \mathbb{E} 
 \left[
 \|X_{0}\|^{2p}
 \right] 
 \right) .
\end{equation}
The expression \eqref{equation:expression-of-the-average-process-in-lemma} is direct to be verified by combining Lemma \ref{lemma:expression-of-yf-in-a-coarse-timestep} and Lemma \ref{lemma:expression-of-ya-in-a-coarse-timestep}.
Set $\overline{\mathscr{P}}(x,y):= \frac{1}{2}(\mathscr{P}(x)+ \mathscr{P}(y))$ for $\forall x, y\in \mathbb{R}^{d}$.
Thereafter, the Taylor expansion is used for $\mu\big(\mathscr{P}\big(Y_{n}^{f} \big)\big)$ and $\mu\big(\mathscr{P}\big(Y_{n}^{a}\big)\big)$ at $\overline{\mathscr{P}}(Y_{n}^{f},Y_{n}^{a})$,  respectively, to show that
\begin{footnotesize}
\begin{equation}
\begin{aligned}
B_{n}^{(1)} 
&= 
\left[
\tfrac{1}{2} 
\Big(
\mu
\big(
\mathscr{P}\big(Y_{n}^{f} \big)
\big)
+
\mu
\big(
\mathscr{P}\big(Y_{n}^{a}\big)
\big) 
\Big) 
- 
\mu
\big(
\overline{\mathscr{P}}(Y_{n}^{f},Y_{n}^{a})
\big) 
\right]h 
+ 
\left(
\mu
\big(
\overline{\mathscr{P}}(Y_{n}^{f},Y_{n}^{a})
\big) 
-\mu
\big(
\mathscr{P}(\overline{Y}_{n}^{f})
\big) 
\right)h \\
& \ 
+ 
\left[
\mathbb{T}_{h,\mu}
\big(
\mathscr{P}\big(\overline{Y}_{n}^{f}\big)
\big) 
- \tfrac{1}{2}
\Big(
\mathbb{T}_{h,\mu}
\big(
\mathscr{P}\big(Y_{n}^{f} \big)
\big) 
+ 
\mathbb{T}_{h,\mu}
\left(
\mathscr{P}\big(Y_{n}^{a}\big)
\right)
\Big)
\right]h \\
&=\tfrac{1}{2}
\left(
\mathcal{R}_{\mu}
\Big(
\overline{\mathscr{P}}
(
Y_{n}^{f},Y_{n}^{a}
)
, 
\mathscr{P}\big(Y_{n}^{f} \big)
\Big) 
+ \mathcal{R}_{\mu}
\Big(
\overline{\mathscr{P}}
(Y_{n}^{f},Y_{n}^{a})
, 
\mathscr{P}\big(Y_{n}^{a}\big)
\Big)
\right)h 
+\left(
\mu
\big(
\overline{\mathscr{P}}(Y_{n}^{f},Y_{n}^{a})
\big) 
- 
\mu
\big(
\mathscr{P}(\overline{Y}_{n}^{f})
\big) 
\right)h\\
&\ 
+ \left[
\mathbb{T}_{h,\mu}
(
\mathscr{P}(\overline{Y}_{n}^{f})
) 
- \tfrac{1}{2}
\Big(
\mathbb{T}_{h,\mu}
\big(
\mathscr{P}\big(Y_{n}^{f} \big)
\big) 
+ \mathbb{T}_{h,\mu}
\big(
\mathscr{P}\big(Y_{n}^{a}\big)
\big)
\Big)
\right]h .
\end{aligned} 
\end{equation}
\end{footnotesize}
In light of Lemma \ref{lemma:strong-convergence-rate-of-yf-and-ya} and the analysis in \eqref{equation:moments-bound-of-remaining-term-R}, we further attain that
\begin{equation} \label{equation:moment-estimate-of-the-remaining-term-in-the-average-process}
\begin{aligned}
&\mathbb{E}
\left[ 
\left\|
\mathcal{R}_{\mu}
\Big(
\overline{\mathscr{P}}(Y_{n}^{f},Y_{n}^{a})
, 
\mathscr{P}\big(Y_{n}^{f} \big)
\Big) 
\right\|^{2p}
\right] 
\ 
\bigvee 
\ 
\mathbb{E}
\left[ 
\left\|
\mathcal{R}_{\mu}
\Big(
\overline{\mathscr{P}}(Y_{n}^{f},Y_{n}^{a})
, 
\mathscr{P}\big(Y_{n}^{a}\big)
\Big) 
\right\|^{2p} 
\right] \\
&\leq Ch^{2p} 
\times
\left\{
\begin{array}{l}
\sup_{0 \leq r \leq 2^{\ell -1} } 
\mathbb{E}
\big[ 
\left(1+ \|Y_{r}\| \right)^{
4p\max\{2\gamma-1, \alpha_{1}, \alpha_{2}, \mathbf{a} \}
} 
\big], 
\quad \gamma \in [1,2], \\
\sup_{0 \leq r \leq 2^{\ell -1} } 
\mathbb{E}
\big[
\left(1+ \|Y_{r}\| \right)^{
4p\max\{2\gamma-1, \alpha_{1}, \alpha_{2}, \mathbf{a} \} + 2p(\gamma-2)
} 
\big], 
\quad \gamma \in (2,\infty).
\end{array} \right.
\end{aligned}
\end{equation}
Using Assumption \ref{assumption:polynomial-growth-condition}, Assumption \ref{assumption:condition-on-the-MM} and the H\"older inequality, we have
\begin{equation}
\begin{aligned}
\mathbb{E}
\left[ 
\big\|
\mu
\big(
\overline{\mathscr{P}}(Y_{n}^{f},Y_{n}^{a})
\big) 
- 
\mu
\big(
\mathscr{P}(\overline{Y}_{n}^{f})
\big) 
\big\|^{2p}
\right]
\leq 
Ch^{4p} 
\sup_{0 \leq r \leq 2^{\ell -1} } 
\mathbb{E}
\big[ 
\left(1+ \|Y_{r}\| \right)^{2p(\gamma-1+\textbf{b})} 
\big].
\end{aligned}
\end{equation}
The following estimate can be derived by the elementary inequality, \eqref{equation:moment-estimate-of-the-remaining-term-in-the-average-process} and Assumption \ref{assumption:condition-on-the-MM}
\begin{small}
\begin{equation}
\begin{aligned}
&\mathbb{E}
\left[ 
\|  B_{n}^{(1)}  \|^{2p} 
\right]  
\leq 
Ch^{4p} \times 
\left\{
\begin{array}{l}
\sup_{0 \leq r \leq 2^{\ell -1} } 
\mathbb{E}\big[
\left(1+ \|Y_{r}\| \right)^{
2p\max\{4\gamma-2, 2\alpha_{1}, 2\alpha_{2}, 2\mathbf{a}, \mathbf{b}+\gamma-1 \}
} 
\big], 
\quad \gamma \in [1,2], \\
\sup_{0 \leq r \leq 2^{\ell -1} } 
\mathbb{E}
\big[ 
\left(
1+ \|Y_{r}\| 
\right)^{
2p\max\{4\gamma-2, 2\alpha_{1}, 2\alpha_{2}, 2\mathbf{a}, \mathbf{b}+\gamma-1 \} + 2p(\gamma-2)
} 
\big], 
\quad \gamma \in (2,\infty).
\end{array} \right.
\end{aligned}
\end{equation}
\end{small}
In a similar way and also bearing \eqref{equation:moments-property-of-bronwain-motion} in mind, one obtains
\begin{small}
\begin{equation}
\begin{aligned}
&\mathbb{E}
\left[ 
\|  M_{n+1}^{(1)}  \|^{2p} 
\right] 
\bigvee 
\mathbb{E}
\left[ 
\|  M_{n+1}^{(2)}  \|^{2p} 
\right] \\
&\leq Ch^{3p} 
\times 
\left\{
\begin{array}{l}
\sup_{0 \leq r \leq 2^{\ell -1} } 
\mathbb{E}
\big[ 
\left(
1+ \|Y_{r}\| 
\right)^{
2p\max\{4\gamma-2, 2\alpha_{1}, 2\alpha_{2},  2\mathbf{a}, \mathbf{b}+\gamma-1 \}
} 
\big], 
\quad \gamma \in [1,2], \\
\sup_{0 \leq r \leq 2^{\ell -1} } 
\mathbb{E}
\big[ 
\left(
1+ \|Y_{r}\| 
\right)^{
2p\max\{
4\gamma-2, 2\alpha_{1}, 2\alpha_{2},  2\mathbf{a}, \mathbf{b}+\gamma-1  
\} 
+ 2p(\gamma-2)
} 
\big], 
\quad \gamma \in (2,\infty).
\end{array} \right.
\end{aligned}
\end{equation}
\end{small}
Before turning to the estimate of $M_{n+1} ^{(3)}$, we note that $\delta W_{n}$ and $\delta W_{n+1/2}$ are independent of $\mathcal{F}_{t_{n}}$, and, for any $p \in [1, \infty)$,
\begin{equation}
\mathbb{E} 
\left[ 
\left\| \delta W_{j_{1}, n} \delta W_{j_{2}, n+1/2}
- \delta W_{j_{1}, n+1/2} \delta W_{j_{2}, n} 
\right\|^{2p} 
\right] 
\leq Ch^{2p}, 
\quad j_{1}, j_{2} \in\{1, \ldots, m\}. 
\end{equation}
Owing to \eqref{equation:growth-of-milstein-diffusion-term},  Assumption \ref{assumption:condition-on-the-MM}, Lemma \ref{lemma:strong-convergence-rate-of-yf-and-ya} and the H\"older inequality, we derive
 \begin{footnotesize}
 \begin{equation}
\begin{aligned}
&\mathbb{E}
\left[ 
\|  M_{n+1}^{(3)}  \|^{2p} 
\right]\\
&\leq 
Ch^{2p}\sum_{j_{1},j_{2}=1}^{m} 
\bigg( 
\mathbb{E}
\left[
\left\|
\mathcal{L}^{j_{1}} \sigma_{j_{2}}
\big(
\mathscr{P}\big(Y_{n}^{a}\big)
\big) 
- 
\mathcal{L}^{j_{1}} \sigma_{j_{2}}
\big(
\mathscr{P}\big(Y_{n}^{f} \big)
\big) 
\right\|^{2p} 
\right] 
+ \mathbb{E}
\left[
\left\|
\mathbb{T}_{h,\mathcal{L}_{j_{1}j_{2}}}
\big(
\mathscr{P}\big(Y_{n}^{f} \big)
\big) 
- 
\mathbb{T}_{h,\mathcal{L}_{j_{1}j_{2}}}
\big(
\mathscr{P}\big(Y_{n}^{a}\big)
\big) 
\right\|^{2p} 
\right] 
\bigg)\\
& \leq Ch^{3p}  
\sup_{0 \leq r \leq 2^{\ell -1} } 
\mathbb{E}
\left[
\left(
1+ \|Y_{r}\| 
\right)^{
2p\max\{
2\gamma-1, \alpha_{1}, \alpha_{2}, \alpha_{3}, \mathbf{a} 
\} 
+ 2p(\gamma-1)
} 
\right] .
\end{aligned}
\end{equation}
\end{footnotesize}
Combining the above estimates with remaining terms $M_{n+1}^{f}$, $B_{n+1}^{f}$ and $M_{n+1}^{a}$, $B_{n+1}^{a}$, which has been derived in Lemma \ref{lemma:expression-of-yf-in-a-coarse-timestep} and Lemma \ref{lemma:expression-of-ya-in-a-coarse-timestep}, we can then get the desired assertion.
\end{proof}

\bibliographystyle{abbrv}


\bibliography{refer}

\begin{thebibliography}{10}

\bibitem{10.1093/imanum/drad064}
J.~Bao, C.~Reisinger, P.~Ren, and W.~Stockinger.
\newblock {{Milstein schemes and antithetic multilevel Monte Carlo sampling for delay McKean–Vlasov equations and interacting particle systems}}.
\newblock {\em IMA Journal of Numerical Analysis}, 44(4):2437--2479, 2024.

\bibitem{beyn2016stochastic}
W.-J. Beyn, E.~Isaak, and R.~Kruse.
\newblock Stochastic {C}-stability and {B}-consistency of explicit and implicit {Euler}-type schemes.
\newblock {\em Journal of Scientific Computing}, 67(3):955--987, 2016.

\bibitem{beyn2017stochastic}
W.-J. Beyn, E.~Isaak, and R.~Kruse.
\newblock Stochastic {C}-stability and {B}-consistency of explicit and implicit {Milstein}-type schemes.
\newblock {\em Journal of Scientific Computing}, 70(3):1042--1077, 2017.

\bibitem{clark1980maximum}
J.~M. Clark and R.~Cameron.
\newblock The maximum rate of convergence of discrete approximations for stochastic differential equations.
\newblock In {\em Stochastic differential systems filtering and control}, pages 162--171. Springer, 1980.

\bibitem{derouich2022interpolated}
M.~B. Derouich and A.~Kebaier.
\newblock The interpolated drift implicit {Euler} scheme {Multilevel Monte Carlo} method for pricing {Barrier} options and applications to the {CIR and CEV} models.
\newblock {\em arXiv preprint arXiv:2210.00779}, 2022.

\bibitem{duffie1995efficient}
D.~Duffie and P.~Glynn.
\newblock Efficient {Monte Carlo} simulation of security prices.
\newblock {\em The Annals of Applied Probability}, pages 897--905, 1995.

\bibitem{fang2019multilevel}
W.~Fang and M.~B. Giles.
\newblock {Multilevel Monte Carlo method for ergodic SDEs without contractivity}.
\newblock {\em Journal of Mathematical Analysis and Applications}, 476(1):149--176, 2019.

\bibitem{giles2008improved}
M.~Giles.
\newblock Improved multilevel {Monte Carlo} convergence using the {Milstein} scheme.
\newblock In {\em Monte Carlo and Quasi-Monte Carlo Methods 2006}, pages 343--358. Springer, 2008.

\bibitem{giles2008multilevel}
M.~B. Giles.
\newblock Multilevel monte carlo path simulation.
\newblock {\em Operations research}, 56(3):607--617, 2008.

\bibitem{giles2014antithetic}
M.~B. Giles and L.~Szpruch.
\newblock Antithetic multilevel {Monte Carlo} estimation for multi-dimensional {SDEs} without {L}{\'e}vy area simulation.
\newblock {\em The Annals of Applied Probability}, 24(4):1585--1620, 2014.

\bibitem{higham2015introduction}
D.~J. Higham.
\newblock An introduction to multilevel monte carlo for option valuation.
\newblock {\em International Journal of Computer Mathematics}, 92(12):2347--2360, 2015.

\bibitem{higham2002strong}
D.~J. Higham, X.~Mao, and A.~M. Stuart.
\newblock Strong convergence of {Euler}-type methods for nonlinear stochastic differential equations.
\newblock {\em SIAM Journal on Numerical Analysis}, 40(3):1041--1063, 2002.

\bibitem{hu96semi-implicit}
Y.~Hu.
\newblock {Semi-implicit Euler-Maruyama scheme for stiff stochastic equations}.
\newblock {\em in Stochastic Analysis and Related Topics V: The Silvri Workshop, Progr. Probab.}, 38:183--202, 1996.

\bibitem{hutzenthaler2011convergence}
M.~Hutzenthaler and A.~Jentzen.
\newblock Convergence of the stochastic {E}uler scheme for locally {Lipschitz} coefficients.
\newblock {\em Foundations of Computational Mathematics}, 11(6):657--706, 2011.

\bibitem{2010Strong}
M.~Hutzenthaler, A.~Jentzen, and P.~E. Kloeden.
\newblock Strong convergence of an explicit numerical method for {SDEs} with nonglobally {Lipschitz} continuous coefficients.
\newblock {\em The Annals of Applied Probability}, 22(4):1611--1641, 2010.

\bibitem{hutzenthaler2011strong}
M.~Hutzenthaler, A.~Jentzen, and P.~E. Kloeden.
\newblock Strong and weak divergence in finite time of {Euler's} method for stochastic differential equations with non-globally {Lipschitz} continuous coefficients.
\newblock {\em Proceedings of the Royal Society A: Mathematical, Physical and Engineering Sciences}, 467(2130):1563--1576, 2011.

\bibitem{hutzenthaler2012strong}
M.~Hutzenthaler, A.~Jentzen, and P.~E. Kloeden.
\newblock Strong convergence of an explicit numerical method for {SDEs} with nonglobally {Lipschitz} continuous coefficients.
\newblock {\em The Annals of Applied Probability}, 22(4):1611--1641, 2012.

\bibitem{hutzenthaler2013divergence}
M.~Hutzenthaler, A.~Jentzen, and P.~E. Kloeden.
\newblock Divergence of the multilevel {Monte Carlo Euler} method for nonlinear stochastic differential equations.
\newblock {\em The Annals of Applied Probability}, 23(5):1913--1966, 2013.

\bibitem{jentzen2016stochastic}
A.~Jentzen.
\newblock Stochastic partial differential equations: analysis and numerical approximations.
\newblock {\em Lecture notes, ETH Zurich, summer semester}, 2016.

\bibitem{kebaier2005statistical}
A.~Kebaier.
\newblock Statistical {Romberg} extrapolation: a new variance reduction method and applications to option pricing.
\newblock {\em The Annals of Applied Probability}, 15(4):2681--2705, 2005.

\bibitem{kelly2021adaptive}
C.~Kelly and G.~J. Lord.
\newblock Adaptive {Euler} methods for stochastic systems with non-globally {Lipschitz} coefficients.
\newblock {\em Numerical Algorithms}, 89:721--747, 2022.

\bibitem{kouarfate2021explicit}
I.~R. Kouarfate, M.~A. Kouritzin, and A.~MacKay.
\newblock {Explicit solution simulation method for the 3/2 model}.
\newblock In {\em Advances in Probability and Mathematical Statistics: CLAPEM 2019, M{\'e}rida, Mexico}, pages 123--145. Springer, 2021.

\bibitem{kumar2019milstein}
C.~Kumar and S.~Sabanis.
\newblock On {Milstein} approximations with varying coefficients: the case of super-linear diffusion coefficients.
\newblock {\em BIT Numerical Mathematics}, 59(4):929--968, 2019.

\bibitem{li2019explicit}
X.~Li, X.~Mao, and G.~Yin.
\newblock Explicit numerical approximations for stochastic differential equations in finite and infinite horizons: truncation methods, convergence in p th moment and stability.
\newblock {\em IMA Journal of Numerical Analysis}, 39(2):847--892, 2019.

\bibitem{mao2007stochastic}
X.~Mao.
\newblock {\em Stochastic differential equations and applications}.
\newblock Elsevier, 2007.

\bibitem{neuenkirch2014first}
A.~Neuenkirch and L.~Szpruch.
\newblock First order strong approximations of scalar {SDEs} defined in a domain.
\newblock {\em Numerische Mathematik}, 128(1):103--136, 2014.

\bibitem{sabanis2016euler}
S.~Sabanis.
\newblock Euler approximations with varying coefficients: the case of superlinearly growing diffusion coefficients.
\newblock {\em The Annals of Applied Probability}, 26(4):2083--2105, 2016.

\bibitem{wang2017strong}
X.~Wang.
\newblock {Strong convergence rates of the linear implicit Euler method for the finite element discretization of SPDEs with additive noise}.
\newblock {\em IMA Journal of Numerical Analysis}, 37(2):965--984, 2017.

\bibitem{wang2020mean1}
X.~Wang.
\newblock Mean-square convergence rates of implicit {Milstein} type methods for {SDEs} with non-{Lipschitz} coefficients.
\newblock {\em Advances in Computational Mathematics}, 49(3):37, 2023.

\bibitem{wang2013tamed}
X.~Wang and S.~Gan.
\newblock The tamed {Milstein} method for commutative stochastic differential equations with non-globally {Lipschitz} continuous coefficients.
\newblock {\em Journal of Difference Equations and Applications}, 19(3):466--490, 2013.

\bibitem{wang2020mean}
X.~Wang, J.~Wu, and B.~Dong.
\newblock Mean-square convergence rates of stochastic theta methods for {SDEs} under a coupled monotonicity condition.
\newblock {\em BIT Numerical Mathematics}, 60(3):759--790, 2020.

\bibitem{yang2018explicit}
H.~Yang and X.~Li.
\newblock Explicit approximations for nonlinear switching diffusion systems in finite and infinite horizons.
\newblock {\em Journal of Differential Equations}, 265(7):2921--2967, 2018.

\end{thebibliography}

\end{document}